\documentclass[reqno,11pt]{article}
\usepackage{mathrsfs}
\usepackage{amssymb}
\usepackage{amsthm}
\usepackage{amsmath}
\usepackage{amsfonts}
\usepackage{mathtools}
\usepackage{enumerate}
\usepackage{bm}

\usepackage{graphicx}

\usepackage{bbm}
\usepackage{mathrsfs}
\usepackage{amssymb}
\usepackage[thicklines]{cancel}

\UseRawInputEncoding

\usepackage[usenames,dvipsnames]{xcolor}
\usepackage{soul}

\usepackage{txfonts}

\usepackage{psfrag}
\usepackage[numbers,sort&compress]{natbib}
\usepackage{setspace,color,wrapfig}
\usepackage[colorlinks,linkcolor=blue,citecolor=cyan]{hyperref}

\numberwithin{equation}{section}
\newtheorem{theorem}{Theorem}[section]
\newtheorem{lemma}[theorem]{Lemma}
\newtheorem{corollary}[theorem]{Corollary}
\newtheorem{proposition}[theorem]{Proposition}
\newtheorem{problem}[theorem]{Problem}
\theoremstyle{definition}

\newtheorem{remark}[theorem]{Remark}
\theoremstyle{remark}

\begin{document}
\title{  Structural stability of    supersonic spiral flows with large angular velocity  for the   Euler-Poisson system }
\author{Chunpeng Wang\thanks{School of Mathematics, Jilin University, Changchun, Jilin Province, China, 130012. Email: wangcp@jlu.edu.cn}\and
Zihao Zhang\thanks{School of Mathematics, Jilin University, Changchun, Jilin Province, China, 130012. Email: zhangzihao@whu.edu.cn}}
\date{}

\newcommand{\de}{{\mathrm{d}}}
\def\div{{\rm div\,}}
\def\curl{{\rm curl\,}}
\def\th{\theta}
\newcommand{\ro}{{\rm rot}}
\newcommand{\sr}{{\rm supp}}
\newcommand{\sa}{{\rm sup}}
\newcommand{\va}{{\varphi}}
\newcommand{\me}{\mathcal{E}}
\newcommand{\ml}{\mathcal{V}}
\newcommand{\md}{\mathcal{D}}
\newcommand{\mg}{\mathcal{G}}
\newcommand{\mh}{\mathcal{H}}
\newcommand{\mf}{\mathcal{F}}
\newcommand{\ms}{\mathcal{S}}
\newcommand{\mt}{\mathcal{T}}
\newcommand{\mn}{\mathcal{N}}
\newcommand{\mb}{\mathcal{P}}
\newcommand{\mm}{\mathcal{B}}
\newcommand{\mj}{\mathcal{J}}
\newcommand{\mk}{\mathcal{K}}
\newcommand{\my}{\mathcal{U}}
\newcommand{\mw}{\mathcal{W}}
\newcommand{\mq}{\mathcal{Q}}
\newcommand{\ma}{\mathcal{L}}
\newcommand{\mc}{\mathcal{C}}
\newcommand{\mi}{\mathcal{I}}
\newcommand{\n}{\nabla}
\newcommand{\e}{\tilde}
\newcommand{\m}{\Omega}
\newcommand{\h}{\hat}
\newcommand{\x}{\bar}
 \newcommand{\q}{{\rm R}}
\newcommand{\p}{{\partial}}
\newcommand{\z}{{\varepsilon}}
\renewcommand\figurename{\scriptsize Fig}
\pagestyle{myheadings} \markboth{Smooth supersonic spiral   flows with nonzero vorticity }{Smooth supersonic spiral   flows with nonzero vorticity }\maketitle
\begin{abstract}
     This paper  concerns the structural stability of smooth cylindrically symmetric supersonic spiral flows with large angular velocity for  the
  steady Euler-Poisson system in a concentric cylinder. We  establish  the existence and uniqueness of  some smooth  supersonic Euler-Poisson  flows
      with nonzero angular velocity and vorticity including both cylindrical spiral flows and axisymmetric spiral flows.  The deformation-curl-Poisson decomposition for the steady Euler-Poisson system is utilized to deal with the hyperbolic-elliptic mixed structure in the supersonic
region. For smooth cylindrical supersonic spiral flows,   the key point lies on the  well-posedness  of a boundary value problem for a
linear second order hyperbolic-elliptic coupled system, which is achieved by finding an appropriate multiplier to  obtain the important basic  energy estimates.  The nonlinear structural stability   is established by designing a two-layer iteration and combining the estimates for the hyperbolic-elliptic system and the transport
equations. For smooth axisymmetric supersonic spiral  flows,  we use the special structure of the  steady  Euler-Poisson  system  to derive a priori estimates of  the  linearized second order elliptic system, which enable us to establish the structural stability of the background supersonic flow within  the class of axisymmetric  flows.
\end{abstract}
\begin{center}
\begin{minipage}{5.5in}
\end{minipage}
\end{center}
\section{Introduction and main results  }\noindent
\par  In this paper, we are concerned with smooth supersonic Euler-Poisson flows with large angular
velocity in a concentric cylinder $\mathbb{D}=\bigg\{(x_1,x_2,x_3): r_0<r=\sqrt{x_1^2+x_2^2+x_3^2}<r_1,-1< x_3< 1\bigg\}$, which is governed by the following  three-dimensional steady compressible Euler-Poisson system:
\begin{equation}\label{1-1-1}
\begin{cases}
\begin{aligned}
&\partial_{x_1}(\rho u_1)+\partial_{x_2}(\rho u_2)+\partial_{x_3}(\rho u_3)=0,\\
&\partial_{x_1}(\rho u_1^2)+\partial_{x_2}(\rho u_1 u_2)+\partial_{x_3}(\rho u_1 u_3)+\partial_{x_1} P =\rho \p_{x_1} \Phi,\\
&\partial_{x_1}(\rho u_1u_2)+\partial_{x_2}(\rho u_2^2)+\partial_{x_3}(\rho u_2 u_3)+\partial_{x_2} P=\rho \p_{x_2} \Phi,\\
&\partial_{x_1}(\rho u_1u_3)+\partial_{x_2}(\rho u_2 u_3)+\partial_{x_3}(\rho u_3^2)+\partial_{x_3} P=\rho \p_{x_3} \Phi,\\
&\div(\rho\mathcal{E}{\bf u}+P{\bf u})=\rho{\bf u}\cdot\nabla\Phi,\\
&{\Delta} \Phi= \rho-b({\bf x}).
\end{aligned}
\end{cases}
\end{equation}
In the  system \eqref{1-1-1},
 ${\bf u}=(u_1,u_2,u_3)$ is the macroscopic particle velocity, $\rho$ is the electron density, $P$ is the pressure, $\mathcal{E}$ is the total energy, respectively. $ \Phi $  is  the electrostatic potential generated by the Coulomb force of particles, and $b$  is a positive function representing the density of fixed,
positively charged background ions.  The system \eqref{1-1-1} is closed with the aid of definition of specific total energy and the equation of state
\begin{equation*}
\mathcal{E}=\frac{|{\bf u}|^2}{2}+\mathfrak{e} \quad \text{and}\quad P=P(\rho, \mathfrak{e}),
\end{equation*}
where $\mathfrak{e}$ is the internal energy.
 We consider the ideal polytropic gas for
which the total energy and the pressure are given by
\begin{equation*}
  \mathcal{E}=\frac1 2|{\bf u}|^2+\frac{ P}{(\gamma-1)\rho}
\quad{\rm {and}}\quad P(\rho,S)= e^{S}\rho^{\gamma},
\end{equation*}
where $S$ is  the entropy and $ \gamma> 1 $ is the adiabatic exponent. The Mach
number $M$ is given by
  $$M=\frac{|{\bf u}|}{\sqrt{P_\rho(\rho, S)}}, $$
where
$ \sqrt{P_\rho(\rho, S)}$ is called the local sound speed. Then the flow is subsonic if $M<1$, sonic if $M=1$,
and supersonic if $M>1$.
\par The system \eqref{1-1-1} consists of two parts, a compressible Euler system with source
terms and a Poisson equation, weakly coupled in a nonlinear way.   For the  past ten years, there have been plenty of results on   smooth Euler flows through various nozzles, such as subsonic, subsonic-sonic, sonic-supersonic  and transonic flows, see \cite{CX14,DWX14,DXY11,WX2013,WX15,wang15,WX2016,WX2019,WX2020,wang22,WS15,WX24} and references therein. On the other hand, there have been a few significant progress on the structural stability of steady Euler-Poisson flows in
 nozzles. For subsonic flows, when the background solutions have low Mach numbers and a
small electric field, the existence and uniqueness of three-dimensional subsonic flows with nonzero
vorticity has been obtained in \cite{W14}. Later on,  the
structural stability of subsonic solutions  was proved in \cite{BDX14,BDX15,BDX16,BW18} even
when the background solutions have large variations. For supersonic flows, the authors in \cite{BDXJ21} investigated  the structural stability  of two-dimensional supersonic irrotational  flows  and
flows with nonzero vorticity in a flat nozzle. The unique existence  of  supersonic  potential flows  in  two-dimensional divergent nozzles and three-dimensional cylindrical nozzles was obtained in \cite{DWY23,BP21,BP23,WZ25-1}.    For transonic flows, the structural stability  for
one-dimensional smooth transonic flows with positive acceleration to the steady Euler-Poisson system was proved in \cite{Bae1}.
\par
    By the hodograph transformation, Courant and Friedrichs
    \cite[Section 104]{CF48} transformed the potential equation into a linear second order differential equation on the flow speed and the flow angle plane, and found some radially symmetric flows including circulatory flows, radial
flows and their superpositions spiral flows.  The authors in \cite{WXY21,CMZ23} studied radially symmetric spiral flows with nonzero angular velocity for the steady Euler system and the steady Euler-Poisson system with relaxation effect  in an annulus, respectively. They
gave  a complete classification of all possible flow patterns for inviscid transonic flow with or without shocks.
The existence of the subsonic and subsonic-sonic spiral flow outside a porous
body was established  \cite{WZ22}. On the other hand, these  radially symmetric
flows can be regarded as cylindrically symmetric flows in three-dimensional setting. The authors in \cite{WYZ21}  investigated the structural stability of cylindrically symmetric transonic spiral  flows under the perturbations of suitable boundary conditions.     The  existence and uniqueness of three-dimensional smooth   subsonic  spiral flows to the steady Euler-Poisson system were obtained in \cite{WYZ24}. Recently, the authors in \cite{WZ25} proved the
structural stability of smooth cylindrically symmetric subsonic spiral flows with self-gravitation under
axisymmetric perturbations of suitable boundary conditions.
\par  In this paper, we investigate the structural stability of   a special class of   smooth cylindrically
symmetric supersonic spiral flows  moving from the inner cylinder  to the outer cylinder.  To this end, introduce the cylindrical coordinates
$(r, \theta, z)$:
 \begin{eqnarray}\label{coor}
r=\sqrt{x_1^2+x_2^2},\quad \theta=\arctan \frac{x_2}{x_1},\quad z=x_3.
\end{eqnarray}
Assume that the velocity, the density, the pressure and the eletrostatic
potential are of the form
\begin{equation*}
\begin{aligned}
&{\bf u}({\textbf x})= U_{1}(r,\th,z)\mathbf{e}_r+U_{2}(r,\th,z)\mathbf{e}_\theta+U_{3}(r,\th,z)\mathbf{e}_z, \\
&\rho({\textbf x})=\rho(r,\th,z),\ \ P({\textbf x})=P(r,\th,z),\ \ \Phi({\textbf x})=\Phi(r,\th,z),\ \
\end{aligned}
\end{equation*}
 with
\begin{equation*}
\mathbf{e}_r =
\begin{pmatrix}
\cos \th  \\
  \sin \th \\
  0
\end{pmatrix}, \quad
\mathbf{e}_{\th } =
\begin{pmatrix}
  - \sin \th  \\
  \cos \th\\
  0
  \end{pmatrix}, \quad
  \mathbf{e}_z =
\begin{pmatrix}
 0 \\
  0 \\
  1
\end{pmatrix},
\end{equation*}
where $ U_1 $,  $ U_2 $ and  $ U_3  $ represent the radial velocity, the angular velocity  and the vertical velocity, respectively.
Then
the steady compressible  Euler-Poisson system  in the cylindrical coordinates takes the
form
\begin{eqnarray}\label{1-2}
\begin{cases}
\begin{aligned}
&\partial_r(r\rho U_1)+\partial_{\theta}(\rho U_2)+\p_{z}(r\rho U_3)=0,\\
&\rho\left(U_1\partial_r +\frac{U_2}{r}\partial_{\theta}+U_3\p_{z}\right)U_1+\partial_r P-\frac{\rho U_2^2}r=\rho\p_r\Phi,\\
&\rho\left(U_1\partial_r +\frac{U_2}{r}\partial_{\theta}+U_3\p_{z}\right)U_2
+\frac{\partial_{\theta}P}{r}+\frac{\rho U_1U_2}{r}=\frac{\rho\p_{\th}\Phi} r,\\
&\rho\left(U_1\partial_r +\frac{U_2}{r}\partial_{\theta}+U_3\p_{z}\right)U_3+\partial_{{z}}P=
\rho\partial_{{z}}\Phi,\\
&\rho\left(U_1\partial_r +\frac{U_2}{r}\partial_{\theta}+U_3\p_{z}\right)S=0,\\
&\left(\p_r^2+\frac 1 r\p_r+\frac{1}{r^2}\p_{\th}^2+\p_{z}^2\right)\Phi=\rho-b.
\end{aligned}
\end{cases}
\end{eqnarray}
Furthermore,   the Bernoulli's function $K$ is defined as
\begin{equation}\label{1-2-k}
K=\frac{1}{2}|{\bf u}|^2+\frac{\gamma P}{(\gamma-1)\rho}-\Phi=\frac{1}{2}|{\bf u}|^2+\frac{\gamma e^S\rho^{\gamma-1}}{\gamma-1}-\Phi.
\end{equation}
The  domain $\mathbb{D}$ in the  cylindrical coordinates  becomes
\begin{equation*}
\mathbb{D}=\{(r,\th,z):r_0<r<r_1,\ \th\in\mathbb{T}_{2\pi},\ -1<z<1\},
\end{equation*}
where  $\mathbb{T}_{2\pi}$ denotes the one-dimensional torus with period $2\pi$.
\subsection{Background supersonic flows}\noindent
\par   Fix $b$ to be a constant $b_0>0$, the background flow is described by smooth functions of the form
\begin{equation*}
{\bf u}({\textbf x})=\bar U_1(r)\mathbf e_r+\bar U_2(r)\mathbf e_{\th},
 \ \ \rho({\textbf x})=\bar\rho(r),\ \  S({\textbf x})=\bar S(r),\ \  \Phi({\textbf x})=\bar \Phi(r), \ \  r\in[r_0, r_1],
\end{equation*}
which solves the following system:
\begin{equation}\label{1-3}
	\begin{cases}
\begin{aligned}
		&(r \bar \rho  \bar U_1)'(r)=0, \ \ & r\in[r_0, r_1],\\
     	 &\bigg(\bar U_1 \bar U_1'+\frac{1}{\bar\rho}\bar P'\bigg)(r)-\frac{1}{r} \bar U_2^2(r)= \bar E(r),\ \ & r\in[r_0, r_1],\\
    	 &(\bar U_1 \bar U_2')(r)+\frac{1}{r} (\bar U_1\bar U_2)(r)=0,\ \ & r\in[r_0, r_1],\\
     	 &(\bar U_1 \bar S')(r)=0,\ \ & r\in[r_0, r_1],\\
    	 &\bar E'(r)+\frac{1}{r}\bar E(r)=\bar \rho(r)-b_0,\ \ & r\in[r_0, r_1],
    \end{aligned}
	\end{cases}
\end{equation}
with the boundary conditions at the entrance $ r = r_0 $:
\begin{equation}\label{1-4}
\bar\rho(r_0)=\rho_0,\quad
\bar U_{1}(r_0)=U_{1,0},\quad
\bar U_{2}(r_0)=U_{2,0},\quad \bar S(r_0)=S_0, \quad
\bar E(r_0)=E_0.
\end{equation}
Here
$$ \bar P(r)=(e^{\bar S}\bar \rho^\gamma)(r) \quad {\rm{and}} \quad \bar E(r)=\bar\Phi'(r), \ \  r\in[r_0, r_1].$$
 Define $J_1=r_0\rho_0U_{1,0}$ and $J_2=r_0U_{2,0}$. Then
\begin{equation*}
(\bar\rho\bar U_1)(r)=\frac{J_1}{r},\quad \bar U_2(r)=\frac{J_2}{r},\quad
\bar S(r) = S_0, \ \  r\in[r_0, r_1].
\end{equation*}
  Hence \eqref{1-3} can be reduced to the following ODE system for $(\bar\rho, \bar E)$:
\begin{equation}\label{1-5}
	\begin{cases}
\begin{aligned}
		&\bar\rho'(r)=\bar\rho(r)\frac{\bar U_1^2(r)+\bar U_2^2(r)+ r\bar E(r) }{r( \bar c^2(r)-\bar U_1^2(r))}, \ \ & r\in[r_0, r_1],\\
		&(r\bar E)'(r)=r(\bar\rho(r)-b_0),  \ \ & r\in[r_0, r_1],\\
\end{aligned}
	\end{cases}
\end{equation}
with
\begin{equation*}
\bar U_1(r)=\frac{J_1}{r\bar\rho(r)},\ \  \bar c^2(r)=\gamma e^{ S_0}\bar\rho^{\gamma-1}(r), \ \ r\in[r_0, r_1].
 \end{equation*}
\par We have the following well-posedness.
\begin{proposition}\label{pro1}
    Assume that $ \gamma\geq 3 $.  Fix  $ r_0>0 $ and $ b_0>0 $. For given constants    $\rho_0>0$, $ U_{1,0}>0$, $U_{2,0}\neq 0$, $S_0>0$ and $E_0>0$ satisfying
    \begin{equation*}
   \rho_0>b_0 \ \ {\rm{and}} \ \ U_{1,0}^2<\gamma e^{S_0}\rho_0^{\gamma-1}<U_{2,0}^2<2r_0^2\rho_0,
    \end{equation*}
     there exists a constant $  R_0$ depending only on $(\gamma,r_0,\rho_0,U_{1,0},U_{2,0},S_0,E_0) $  such that for any  $ r_1\in(r_0,r_0+R_0 )$,
  the initial value problem \eqref{1-3} and \eqref{1-4} has a unique smooth  solution $(\bar\rho,\bar  U_{1}, \bar U_{2},  \bar S, \bar E)$ satisfying
  \begin{equation}\label{1-6}
   \bar U_1^2(r)<\bar c^2(r)<\bar U_2^2(r), \ \  r\in[r_0, r_1].
  \end{equation}
\end{proposition}
\begin{proof}
Firstly, the existence and uniqueness of local solution to \eqref{1-5} follow
directly from the standard theory of ODE system.
It follows from   $\rho_0>b_0$ that
  $ (r\bar E)'(r_0)>0 $. Then one can  deduce that for some constant $R_1>0$,
\begin{equation}\label{1-6-1}
 r\bar E(r)>r_0E_0>0, \quad r\in[r_0, r_0+R_1].
\end{equation}
 Denote the Mach numbers by
\begin{equation*}
\bar M_1^2(r)=\frac{\bar U_1^2(r)}{ \bar c^2(r)}, \ \
\bar M_2^2(r)=\frac{\bar U_2^2(r)}{\bar c^2(r)}, \  \ \bar  {\bf M}(r)=(\bar M_1(r), \bar M_2(r))^T.
\end{equation*}
Then a direct computation shows that
\begin{eqnarray}\label{1-7}
	&&\frac{d}{d r}\bar M_1^2(r)=-\frac{\bar M_1^2(r)}{r(1-\bar M_1^2(r))}
\left(
(\gamma-1)\bar M_1^2(r)+(\gamma+1)\bar M_2^2(r)+(\gamma+1)\frac{r\bar E(r) }{\bar c^2(r)}+2
\right),\\\label{1-m}
&&\frac{d}{d r}\bar M_2^2(r)=-\frac{\bar M_2^2(r)}{r(1-\bar M_1^2(r))}
\left(
(\gamma-3)\bar M_1^2(r)+(\gamma-1)\bar M_2^2(r)+(\gamma-1)\frac{r\bar E(r) }{\bar c^2(r)}+2
\right).
\end{eqnarray}
	Since $\bar M_1^2(r_0)<1$,     \eqref{1-7},  together with \eqref{1-6-1}, yields that
 \begin{equation}\label{1-7-1}
 \bar M_1^2(r)\leq \bar M_1^2(r_0)<1, \ \  r\in [r_0,r_0+R_1].
 \end{equation}
 Thus the flow is  subsonic  in the $r$-direction.  Note that $ \gamma\geq 3 $. Then   it follows from \eqref{1-6-1}, \eqref{1-m} and \eqref{1-7-1}  that
  \begin{equation*}
 \frac{d}{d r}\bar M_2^2(r)<0, \ \  r\in [r_0,r_0+R_1].
 \end{equation*}
 This, together with $\bar M_2^2(r_0)>1$, yields that for some constant $ R_0\in(0,R_1)$, one has
  \begin{equation*}
 \bar M_2^2(r_0)\geq\bar M_2^2(r)>1, \ \  r\in [r_0,r_0+R_0],
 \end{equation*}
    from which one obtains
     \begin{equation*}
 |\bar{\bf M}|^2(r)>1, \ \  r\in [r_0,r_0+R_0].
 \end{equation*}
    Therefore, the flow is  supersonic. Furthermore,  it follows from the first equation in \eqref{1-5} that
$$
(\ln\bar\rho)'(r)>0, \ \  r\in[r_0, r_0+R_0],
$$
 which implies
  $$ \bar \rho(r)\geq\rho_0>0, \ \ r\in[r_0, r_0+R_0]. $$
 By choosing $ r_1\in (r_0,r_0+R_0) $, we complete the proof of Proposition \ref{pro1}.
\end{proof}
 Define
\begin{eqnarray}\label{1-q}
 && \bar P(r)=e^{S_0} \bar\rho^{\gamma}(r) \ \ {\rm{and}} \ \
 \bar\Phi(r)= \int_{r_0}^{r} \bar E(s) \de s, \ \  r\in[r_0, r_1].
\end{eqnarray}
 Furthermore, the  Bernoulli's function is given by
\begin{eqnarray}\label{1-9}
\bar K(r)=\frac{1}{2}(U_1^2+U_2^2)
+\frac{\gamma e^{S_0}\bar\rho^{\gamma-1}(r)}{\gamma-1}-\bar\Phi(r), \ \  r\in[r_0, r_1]. \end{eqnarray}
It follows from the second equation in \eqref{1-3} that
 \begin{eqnarray*}
 \bar K(r)=K_0, \ \ K_0=\frac{1}{2}(U_{1,0}^2+U_{2,0}^2)+\frac{\gamma e^{S_0}\rho_0^{\gamma-1}}{\gamma-1}, \ \  r\in[r_0, r_1].
 \end{eqnarray*}
 \par The  cylindrically
symmetric  solution       $(\bar\rho,\bar  U_1,\bar  U_2, \bar P, \bar \Phi)$   with nonzero angular veloicty  is called the  background supersonic  solution  associated with the entrance data $(b_0,\rho_0,U_{1,0},U_{2,0},S_0,E_0) $.
  The main purpose of this work is to investigate the structural stability of the
background solution under perturbations of suitable boundary conditions. However, due to the effect of nonzero angular velocity, it seems quite difficult to analyze the structural stability of the background supersonic flow under generic three-dimensional perturbations even in the case of isentropic irrotational solutions. Therefore, we consider here two classes of perturbations with special symmetries.
\subsection{Smooth cylindrical  supersonic spiral
flows}\noindent
\par \par Firstly, we investigate the structural stability of the background
solution under  cylindrical
perturbations of suitable boundary conditions. More precisely, we find the solution to \eqref{1-2} with the form $(U_1(r,\theta),  U_2(r,\theta),U_3(r,\theta)\equiv 0, \rho(r,\theta), P(r,\theta),\Phi(r,\theta))$ satisfying
\begin{eqnarray}\label{1-2-c}
\begin{cases}
\begin{aligned}
&\partial_r(r\rho U_1)+\partial_{\theta}(\rho U_2)=0,\\
&\rho\left(U_1\partial_r +\frac{U_2}{r}\partial_{\theta}\right)U_1+\partial_r P-\frac{\rho U_2^2}r=\rho\p_r\Phi,\\
&\rho\left(U_1\partial_r +\frac{U_2}{r}\partial_{\theta}\right)U_2
+\frac{\partial_{\theta}P}{r}+\frac{\rho U_1U_2}{r}=\frac{\rho\p_{\th}\Phi} r,\\
&\rho\left(U_1\partial_r +\frac{U_2}{r}\partial_{\theta}\right)S=0,\\
&\left(\p_r^2+\frac 1 r\p_r+\frac{1}{r^2}\p_{\th}^2\right)\Phi=\rho-b.
\end{aligned}
\end{cases}
\end{eqnarray}
The  domain $\mathbb{D}$ is  simplified as
\begin{equation*}
\Omega=\{(r,\th):r_0<r<r_1,\ \th\in\mathbb{T}_{2\pi}\}.
\end{equation*}
The entrance and exit  are denoted by
\begin{equation*}
\begin{aligned}
\Gamma_{en}=\{(r,\th):r=r_0,\  \th\in\mathbb{T}_{2\pi}\},\quad
		\Gamma_{ex}=\{(r,\th):r=r_1, \  \th\in\mathbb{T}_{2\pi}\}.\\
\end{aligned}
\end{equation*}
\par We start with cylindrical  irrotational  flows.  For  the isentropic irrotational flows, $ S $ is a constant and  $ \text{curl }{\bf u}=0 $.  Without loss of  generality, we assume $ S=S_0 $. In the Euclidean coordinates,
by the vector identity ${\bf u}\cdot\nabla {\bf u}= \nabla \frac12 |{\bf u}|^2- {\bf u}\times \text{curl }{\bf u}$, the momentum equations in \eqref{1-1-1} imply
\begin{eqnarray}\label{momentum}
\nabla \bigg(\frac{1}2|{\bf u}|^2 +\frac{\gamma e^{S_0}\rho^{\gamma-1}}{\gamma-1}-\Phi\bigg)={\bf u}\times \text{curl }{\bf u}=0, \ \  \text{in}\  \ \mathbb{D},
\end{eqnarray}
 from which one obtains  that $ \frac{1}2|{\bf u}|^2 +\frac{\gamma e^{S_0}\rho^{\gamma-1}}{\gamma-1}-\Phi$ is a constant. For simplicity, we assume
 that
\begin{eqnarray}\label{momentum-1}
 \frac{1}2|{\bf u}|^2 +\frac{\gamma e^{S_0}\rho^{\gamma-1}}{\gamma-1}-\Phi=K_0,  \ \  \text{in}\  \ \mathbb{D}.
\end{eqnarray}
In terms of the  cylindrical  coordinates, it follows from \eqref{momentum-1} that\begin{eqnarray*}
\rho=
\bigg(\frac{\gamma-1}{\gamma e^{S_0}}\bigg(K_0+\Phi-\frac{1}{2}\left(U_1^2+U_2^2\right)\bigg)\bigg)
^{\frac{1}{\gamma-1}}, \ \ \text{in}\ \ \Omega.
\end{eqnarray*}
Then the system  \eqref{1-2-c} in $\m$ can be reduced to   the the following system, which is called the potential flow
model of the steady Euler-Poisson system:
\begin{equation}\label{3-1}
\begin{cases}
\partial_r(r\rho U_1)+\partial_\theta(\rho U_2)=0,\\
\partial_r(r U_2)-\partial_\theta U_1=0,\\
\left(\p_r^2+\frac 1 r\p_r+\frac{1}{r^2}\p_{\th}^2\right)\Phi=\rho-b.\\
\end{cases}
\end{equation}
\par For cylindrical  irrotational  flows, we investigate the following problem.
\begin{problem}\label{probl1}
Given functions $(b, U_{1, en}, U_{2, en},E_{ en}, \Phi_{ ex})$ sufficiently close to $(b_0,U_{1,0},U_{2,0},E_0,\bar\Phi(r_1))$,
find  a  solution
$( U_1,U_2, \Phi) $ to the nonlinear system \eqref{3-1} in $\m $ satisfying  the following properties.
\begin{enumerate}[ \rm (1)]
 \item
 $\rho>0$ and $   U_1>0 $ in $ \overline{\m}$.
 \item $( U_1,U_2, \Phi) $ satisfies the boundary conditions:
\begin{equation}\label{1-c}
\begin{cases}
(U_1, U_2,   \p_r\Phi)(r_0,\th)=(U_{1,en}, U_{2,en},  E_{en})(\th),\ \ &{\rm{on}}\ \ \Gamma_{en},\\
\Phi(r_1,\th)=\Phi_{ex}(\th), \ \ &{\rm{on}}\ \ \Gamma_{ex}.\\
\end{cases}
\end{equation}
\item $U_1^2+U_2^2>c^2(K_0, U_1,  U_2,\Phi)$,  i.e., the flow corresponding to $(U_1,U_2,\Phi)$ is supersonic in $\overline{\m}$. Here the
 the  sound speed $ c(K_0, U_1,  U_2,\Phi) $ is given by
\begin{equation}\label{1-c-z}
c(K_0, U_1,  U_2,\Phi)=\sqrt{P_\rho(\rho, S_0)}=\sqrt{(\gamma-1)\bigg(K_0+\Phi-\frac{1}{2}\left( U_1^2+ U_2^2\right)\bigg)}.
\end{equation}
\end{enumerate}
\end{problem}

\par The first main result in this paper  states the structural stability of  the background supersonic solution under    cylindrical
perturbations of suitable boundary conditions
 for the potential flow model, which also yields the unique existence  of smooth  supersonic  irrotational flows  to the  system \eqref{3-1}.
\begin{theorem}\label{th1}
 For given functions $b\in C^2(\overline{\m})$, $U_{ 1,en}\in C^3(\mathbb{T}_{2\pi})$ with  $\displaystyle{\min_{\mathbb{T}_{2\pi} }U_{1, en}>0 }$ and $(U_{2, en},E_{ en},\Phi_{ ex})$\\$\in \left(C^4(\mathbb{T}_{2\pi})\right)^3$, define
\begin{equation}\label{1-9-a}
\begin{aligned}
 \omega_1(b,U_{ 1,en},U_{2, en},E_{ en},\Phi_{ ex})&:=\|b-b_0\|_{C^2(\overline{\m})}+\|U_{1, en} -U_{1,0}\|_{C^3(\mathbb{T}_{2\pi})}\\
 & \ \ \quad +\|(U_{2, en},E_{ en},\Phi_{ ex})-(U_{2,0},E_0,\bar\Phi(r_1))\|_{C^4(\mathbb{T}_{2\pi})}.\\
  \end{aligned}
  \end{equation}
  For each $\epsilon_0\in(0,R_0)$ with $R_0>0$ given in Proposition \ref{pro1}, there exists a positive  constant $\bar{r}_1\in(r_0,r_0+R_0-\epsilon_0]$ so  that for $r_1\in(r_0, \bar{r}_1)$,
 if
$(b, U_{1, en}, U_{2, en},  E_{ en}, \Phi_{ ex})$ satisfies
\begin{equation}\label{1-t-1}
\omega_1(b,U_{ 1,en},U_{2, en},E_{ en},\Phi_{ ex})\leq  \sigma_1^\ast
\end{equation}
for some constant $\sigma_1^\ast$  depending only on  $(\gamma,b_0,\rho_0,U_{1,0},U_{2,0},S_0,E_0,r_0,r_1,\epsilon_0)$,
  then Problem \ref{probl1} has a  unique smooth    supersonic irrotational solution    $(U_1,U_2, \Phi)\in \left(H^3(\m)\right)^2\times H^4(\m)\subset \left( C^{1,\alpha}(\overline{\Omega})\right)^2\times C^{2,\alpha}(\overline{\Omega})$ with $\alpha\in(0,1)$, which satisfies
 \begin{eqnarray}\label{1-t-3}
\|(U_1,U_2)-(\bar U_1,\bar U_2)\|_{H^3(\m)}+\| \Phi- \bar\Phi\|_{H^4(\m)}\leq  \mc_1^\ast \omega_1(b,U_{ 1,en},U_{2, en},E_{ en},\Phi_{ ex}),
 \end{eqnarray}
 where  $ \mc_1^\ast $ is a positive constant depending only on  $(\gamma,b_0,\rho_0,U_{1,0},U_{2,0},S_0,E_0,r_0,r_1,\epsilon_0)$.
 \end{theorem}
\par Next, we turn to cylindrical    flows with  nonzero vorticity.
\begin{problem}\label{probl2}
Given functions $(b, U_{1, en}, U_{2, en},  K_{ en},S_{ en},E_{ en}, \Phi_{ ex})$ sufficiently close to $(b_0,U_{1,0},U_{2,0},K_0,$\\ $S_0,E_0,\bar\Phi(r_1))$,
find  a solution
$(U_1,U_2,K,S,\Phi) $ to the nonlinear system \eqref{1-2-c} in $\m $ satisfying  the following properties.
\begin{enumerate}[ \rm (1)]
 \item
 $\rho>0$ and $   U_1>0 $ in $ \overline{\m}$.
 \item $( U_1,U_2,K,S, \Phi) $ satisfies the boundary conditions:
\begin{equation}\label{1-c-n}
\begin{cases}
(U_1, U_2,   K,S,\p_r\Phi)(r_0,\th)=(U_{1,en}, U_{2,en},  K_{en},S_{en},E_{en})(\th),\ \ &{\rm{on}}\ \ \Gamma_{en},\\
\Phi(r_1,\th)=\Phi_{ex}(\th), \ \ &{\rm{on}}\ \ \Gamma_{ex}.\\
\end{cases}
\end{equation}
\item
$U_1^2+U_2^2>c^2(K, U_1,  U_2,\Phi)$, i.e, the flow corresponding to $(U_1,U_2, K,S,\Phi)$ is supersonic in $\overline{\m}$.
Here
the  sound speed $ c(K, U_1,  U_2,\Phi) $ is defined by
\begin{equation}\label{1-c-z-n}
c(K, U_1,  U_2,\Phi)=\sqrt{P_\rho(\rho, S)}=\sqrt{(\gamma-1)\bigg(K+\Phi-\frac{1}{2}\left( U_1^2+ U_2^2\right)\bigg)},
\end{equation}

\end{enumerate}
\end{problem}
\par The second main result in this paper yields  the existence and uniqueness of smooth supersonic flows with nonzero vorticity to the  system \eqref{1-2-c}.
\begin{theorem}\label{th2}
 For given functions $b\in C^2(\overline{\m})$, $U_{ 1,en}\in C^3(\mathbb{T}_{2\pi})$ with  $\displaystyle{\min_{\mathbb{T}_{2\pi} }U_{1, en}>0 }$ and $(U_{2, en},K_{ en},S_{ en},$\\$E_{ en},\Phi_{ ex})\in \left(C^4(\mathbb{T}_{2\pi})\right)^5$,
define
\begin{equation}\label{1-9-bb}
\begin{aligned}
\omega_2(K_{ en},S_{ en}):=\|(K_{ en},S_{ en})-(K_0,S_0)\|_{C^4(\mathbb{T}_{2\pi})},
  \\
   \end{aligned}
  \end{equation}
  and set
\begin{equation*}
\sigma_p:=\omega_1(b,U_{ 1,en},U_{2, en},E_{ en},\Phi_{ ex})+\omega_2(K_{ en},S_{ en}),
  \end{equation*}
  where $ \omega_1(b,U_{ 1,en},U_{2, en},E_{ en},\Phi_{ ex}) $ is defined in \eqref{1-9-a}.
  For each $\epsilon_0\in(0,R_0)$ with $R_0>0$ given in Proposition \ref{pro1}, there exists a constant $\bar{r}_1\in(r_0,r_0+R_0-\epsilon_0]$  so  that for $r_1\in(r_0, \bar{r}_1)$,
 if
$(b, U_{1, en}, U_{2, en},  K_{ en},S_{ en},E_{ en}, \Phi_{ ex})$ satisfies
\begin{equation}\label{1-t-1-n}
\sigma_p\leq  \sigma_2^\ast
\end{equation}
   for some constant  $\sigma_2^\ast>0$ depending only on  $(\gamma,b_0,\rho_0,U_{1,0},U_{2,0},S_0,E_0,r_0,r_1,\epsilon_0)$,
 then Problem \ref{probl2} has a  unique smooth supersonic  solution $(U_1,U_2, K,S,\Phi)\in \left(H^3(\m)\right)^2\times \left(H^4(\m)\right)^3$ with nonzero vorticity, which satisfies
 \begin{eqnarray}\label{1-t-3-n}
&&\|(U_1,U_2)-(\bar U_1,\bar U_2)\|_{H^3(\m)}+\| \Phi- \bar\Phi\|_{H^4(\m)}\leq  \mc_2^\ast\sigma_p,
  \\\label{1-t-4-n}
  &&\|(K,S)-(K_0,S_0)\|_{H^4(\m)}\leq \mc_2^\ast \omega_2(K_{ en},S_{ en}),
  \end{eqnarray}
  where
 $C_2^\ast>0$ is a constant  depending only on  $(\gamma,b_0, \rho_0,U_{1,0},U_{2,0}, S_{0}, E_{0},r_0,r_1,\epsilon_0)$.
Furthermore,  for each $\alpha\in(0,1)$, it follows from \eqref{1-t-3-n} and \eqref{1-t-4-n} that
\begin{eqnarray}\label{1-t-7}
  &&\|(U_1,U_2)-(\bar U_1,\bar U_2)\|_{C^{1,\alpha}(\overline{\m})} + \|\Phi-\bar\Phi\|_{C^{2,\alpha}(\overline{\m})}
  \leq  \mc_3^\ast \sigma_p,\\\label{1-t-8}
  &&\|K- K_0\|_{C^{2,\alpha}(\overline{\m})}+
  \|S-S_0\|_{C^{2,\alpha}(\overline{\m})}\leq \mc_3^\ast\omega_2(K_{ en},S_{ en})
   \end{eqnarray}
for some constant   $\mc_3^\ast>0$ depending only on $(\gamma,b_0,\rho_0,U_{1,0},U_{2,0},S_0,E_0, r_0, r_1,\epsilon_0,\alpha)$.
\end{theorem}
\begin{remark}
  {\it We restrict  $r_1\in(r_0, \bar{r}_1)$  in Theorems \ref{th1} and \ref{th2}  to
 establish  a priori $H^1 $
energy estimate
  of the linearized
system consisting of a second order hyperbolic equation and a second order elliptic equation weakly
coupled together.  It should be emphasized that  a priori $ H^1 $ estimate  is the key ingredient to get the well-posedness  of  the  linear second order hyperbolic-elliptic coupled system.}
\end{remark}
\begin{remark}
 {\it For smooth  cylindrical supersonic  rotational flows, an effective decomposition of elliptic and hyperbolic modes is crucial for the solvability of the nonlinear
boundary problem.  The authors in \cite{BDXJ21}  utilized the Helmholtz decomposition of the velocity field  to establish the structural stability of two-dimensional supersonic Euler-Poisson flows with nonzero vorticity in flat nozzles. This decomposition   requires   careful decomposition of boundary conditions so that the vector potential and scalar potential
can be solved uniquely. However, due to the nonzero
angular velocity,  the fluid in  concentric cylinders exhibits
 quite different wave phenomena compared with the fluid moving in straight nozzles.
Here we will utilize  the deformation-curl-Poisson  decomposition introduced  in \cite{WS19}
to reformulate the steady Euler-Poisson system  as a deformation-curl-Poisson system  together with two transport equations for the Bernoulli's quantity and the entropy. In the deformation-curl-Poisson system, the vorticity is resolved by an algebraic
equation for  the Bernoulli's function and the entropy and there is a loss of one derivative in the equation for the vorticity due to the type  of the deformation-curl-Poisson  system in the supersonic region.  To  overcome it, we will design
an elaborate two-layer iteration scheme by choosing some appropriate function spaces. By utilizing the
one order higher regularity of the stream function and the fact that the Bernoulli¡¯s quantity and the
entropy can be represented as functions of the stream function, we  gain one more  order derivatives
estimates for  the Bernoulli's function and the entropy than  the velocity filed. This is crucial for
us to close the energy estimates.}
\end{remark}
\subsection{Smooth axisymmetric  supersonic spiral
flows}\noindent
\par In the following, we investigate the structural stability of the background solution under axisymmetric
perturbations of suitable boundary conditions.  Using the   cylindrical coordinates \eqref{coor}, any function $f({\bf x})$ can be represented as $f({\bf x})=f(r,\theta,z)$, and a vector-valued function ${\bf F}({\bf x})$ can be represented as ${\bf F}({\bf x})=F_1(r,\theta,z){\bf e}_r+ F_2(r,\theta,z){\bf e}_\th+ F_{3}(r,\theta,z){\bf e}_{3}$. We say that a function $f({\bf x})$ is  axisymmetric if its value is independent of $\theta$ and that a vector-valued function ${\bf F}= (F_1, F_2, F_{3})$ is axisymmetric if each of functions $F_1({\bf x}), F_{2}({\bf x})$ and $F_{3}({\bf x})$ is axisymmetric.
\par Assume that the velocity, the density, the pressure and the eletrostatic
potential are of the form
\begin{equation*}
\begin{aligned}
&{\bf u}({\textbf x})= U_{1}(r,z)\mathbf{e}_r+U_{2}(r,z)\mathbf{e}_\theta
+U_3(r,z)\mathbf{e}_z,\\
&\rho({\textbf x})=\rho(r,z),\ \ P({\textbf x})=P(r,z),\ \ \Phi({\textbf x})=\Phi(r,z).\\
\end{aligned}
\end{equation*}
Then the system \eqref{1-2} reduces to
\begin{eqnarray}\label{2-10}
\begin{cases}
\begin{aligned}
&\partial_r(r\rho U_1)+\p_{z}(r\rho U_3)=0,\\
&\rho\left(U_1\partial_r +U_3\p_{z}\right)U_1+\partial_r P-\frac{\rho U_2^2}{r}=\rho\p_r\Phi,\\
&\rho\left(U_1\partial_r +U_3\p_{z}\right)U_2+\frac{\rho U_1U_2}{r}=0,\\
&\rho\left(U_1\partial_r +U_3\p_{z}\right)U_3+\p_{z}P=\rho\p_{z}\Phi,\\
&\rho\left(U_1\partial_r+U_3\p_{z}\right)S=0,\\
&\bigg(\p_r^2+\frac 1 r\p_r+\p_{z}^2\bigg)\Phi=\rho-b.
\end{aligned}
\end{cases}
\end{eqnarray}
The  domain $\mathbb{D}$ is  simplified as
 \begin{equation*}
\mn=\{(r,z): r_0<r<r_1, -1<z<1\}.
\end{equation*}
The entrance, exit and boundaries of the cylinder are denoted by
\begin{equation*}
\begin{aligned}
&\Sigma_{en}=\{(r,z):r=r_0,\  -1<z<1\},\\
		&\Sigma_{ex}=\{(r,z):r=r_1,\  -1<z<1\},\\
&\Sigma_{w}^\pm=
		\{(r,z):r_0<r<r_1,\ z=\pm 1\}.
	\end{aligned}
\end{equation*}
\par  For  axisymmetric   flows, we investigate  the following problem.
\begin{problem}\label{probl3}
Given functions $(b,  U_{2, en}, U_{3, en},K_{en},S_{en} ,\Phi_{ en}, U_{1, ex},\Phi_{ ex})$  sufficiently close to $(b_0,U_{2,0},$\\$0,K_0,S_0,0,\bar U_1(r_1),\bar\Phi(r_1))$,
find  a solution
$( U_1,U_2, U_3, K,S,\Phi) $  to the   system \eqref{2-10} in $\mn$ satisfying  the following properties.
\begin{enumerate}[ \rm (1)]
 \item $\rho>0$ and $   U_1>0 $ in $ \overline{\mn}$.
\item $( U_1,U_2, U_3, K,S,\Phi) $ satisfies the boundary conditions:
\begin{equation}\label{1-c-c-r-A}
\begin{cases}
(U_2,   U_3,K,S,\Phi)(r_0,z)=(U_{2,en}, U_{3,en},  K_{en},S_{en},\Phi_{en})(z),\ \ &{\rm{on}}\ \ \Sigma_{en},\\
(U_1,\Phi)(r_1,z)=(U_{1, ex},\Phi_{ ex})(z), \ \ &{\rm{on}}\ \ \Sigma_{ex},\\
U_3(r,\pm 1)=\p_{z}\Phi(r,\pm 1)=0,\ \ &{\rm{on}}\ \ \Sigma_{w}^\pm.\\
\end{cases}
\end{equation}
\item
$U_1^2+U_2^2+U_3^2>c^2(K, U_1,  U_2,U_3,\Phi)$, i.e, the flow corresponding to $(U_1,U_2,U_3, K,S,\Phi)$ is supersonic in $\overline{\mn}$. Here
the  sound speed $ c(K, U_1,  U_2,U_3,\Phi) $ is given by
\begin{equation}\label{1-c-z-n-A}
c(K, U_1,  U_2,U_3,\Phi)=\sqrt{P_\rho(\rho, S)}=\sqrt{(\gamma-1)\bigg(K+\Phi-\frac{1}{2}\left( U_1^2+ U_2^2+U_3^2\right)\bigg)}.
\end{equation}
\end{enumerate}
\end{problem}
\par The third main result in this paper states  the structural stability of  the background supersonic solution under    axisymmetric
perturbations of suitable boundary conditions, which also yields the the existence and uniqueness
of smooth supersonic flows with nonzero vorticity to the system \eqref{2-10}.
\begin{theorem}\label{th3}
Fix $ \alpha\in(0,1) $. For given functions $b\in C^{0,\alpha}(\overline{\mn})$ and $(  U_{2, en}, U_{3, en},K_{en},S_{en} ,\Phi_{ en}, U_{1, ex},$\\$\Phi_{ ex})\in \left(C^{2,\alpha}([-1,1])\right)^7$ , define
\begin{equation}\label{1-9-a-r-a}
\begin{aligned}
 &\omega_3(b)
 :=\|b-b_0\|_{C^{0,\alpha}(\overline{\mn})},\\
  &\omega_4(U_{2, en}, ,K_{en},S_{en} )
  :=\|  (U_{2, en},
  K_{en}, S_{en} )-(U_{2,0},K_0,S_0)\|_{C^{2,\alpha}([-1,1])},
  \\
  &\omega_5( U_{3, en},\Phi_{ en}, U_{1, ex},\Phi_{ ex})
  :=\|  ( U_{3, en},
  \Phi_{ en}, U_{1, ex},\Phi_{ ex})-(0,0,\bar U_1(r_1),\bar\Phi(r_1))\|_{C^{2,\alpha}([-1,1])},
  \\
   \end{aligned}
  \end{equation}
 and set
\begin{equation*}
\sigma_v:=\omega_3(b)+\omega_4(U_{2, en}, K_{en},S_{en})+\omega_5( U_{3, en},\Phi_{ en}, U_{1, ex},\Phi_{ ex}).
  \end{equation*}
   If
$(b,  U_{2, en}, U_{3, en},K_{en},S_{en} ,\Phi_{ en}, U_{1, ex},\Phi_{ ex})$ satisfies
\begin{equation}\label{1-t-1-r-a}
\sigma_v\leq \sigma_1^\star
\end{equation}
 for  some constant  $\sigma_1^\star>0$ depending only on  $(\gamma,b_0,\rho_0,U_{1,0},U_{2,0},S_0,E_0,r_0,r_1)$, and
 the compatibility conditions
 \begin{equation}\label{1-t-1-r-a-a}
 \begin{aligned}
 &U_{2, en}'(\pm 1)= U_{3, en}(\pm 1)=U_{3, en}''(\pm 1)=K_{en}'(\pm 1)=S_{en}'(\pm 1)\\
 &=\Phi_{ en}'(\pm 1)= U_{1, ex}'(\pm 1)=\Phi_{ ex}'(\pm 1)=0,
 \end{aligned}
 \end{equation}
 then Problem \ref{probl3} has a  unique smooth axisymmetric supersonic  solution $(U_1,U_2,U_3,   K,S,\Phi)\in \left(C^{2,\alpha}(\overline\mn)\right)^6$ with nonzero vorticity, which satisfies
  \begin{eqnarray}\label{1-t-3-r-a}
\|(U_1,U_2,U_3,   K,S,\Phi)-(\bar U_1,\bar U_2,0,K_0,S_0,\bar \Phi)\|_{C^{2,\alpha}(\overline\mn)}\leq  \mc_1^\star\sigma_v,
\end{eqnarray}
where
 $\mc_1^\star>0$ is a  constant depending only on  $(\gamma,b_0, \rho_0,U_{1,0},U_{2,0}, S_{0}, E_{0},r_0,r_1)$. Furthermore,  the solution $(U_1,U_2,U_3, K,S,\Phi  )$ satisfies the compatibility conditions
 \begin{equation}
\label{comp-cond-nlbvp-full-1}
\begin{aligned}
\p_{z} U_1=\p_{z} U_2=U_3=\p_{z}^2U_3=\p_{z} K
=\p_{z} S=\p_{z}\Phi=0, \ \ {\rm{on}} \ \Sigma_{w}^\pm.
\end{aligned}
\end{equation}
\end{theorem}
\begin{remark}
 {\it The analysis of axisymmetric supersonic spiral  flows is simpler than those of the cylindrical supersonic spiral flows.  By the deformation-curl-Poisson decomposition developed in \cite{WS19}, it will be shown that the  deformation-curl-Poisson  system for $ (U_1,U_3,\Phi) $ is an elliptic
system when linearized around the background supersonic flow. The quantities $(r U_2, B,S)$ are conserved along the particle trajectory.
We discover  a
special structure of the associated  elliptic system,  which  yields  an $H^1$
estimate. It should be emphasized that the $H^1$ estimate of the weak solution
for the linearized elliptic system plays a key role in our analysis. Based on the $H^1(\mn)$ estimate,  we  establish the $C^{1,\alpha}(\overline\mn) $ estimate of the weak solution  and the   $ C^{2,\alpha}(\overline\mn) $ estimate can be derived  by  the symmetric extension technique and  Schauder type estimate.}
\end{remark}
\begin{remark}
{\it If one consider the case with $ \gamma>1 $, for any given $\rho_0>0$, $ U_{1,0}>0$, $U_{2,0}\neq 0$, $S_0>0$ and $E_0>0$ satisfying
    $\rho_0>b_0 $ and $ U_{1,0}^2<\gamma e^{S_0}\rho_0^{\gamma-1}<U_{1,0}^2+U_{2,0}^2 $, there exists also a class of smooth cylindrically symmetric supersonic spiral flows to the problem \eqref{1-3} and \eqref{1-4} satisfying $ |\bar{\bf M}|^2(r)>1>\bar M_1^2(r) $.
     Such a supersonic flow is also structurally stable under the same perturbations as in \eqref{1-c-c-r-A} within the class of axisymmetric flows. However,
the structural stability of this background solution under cylindrical perturbations is much more complicated than the case in Theorems \ref{th1}  and \ref{th2}. The main difficulty is to establish  a priori $H^1 $
energy estimate  for the associated linearized hyperbolic-elliptic coupled system. As we shall see later in Lemma \ref{pro4}, The positivity of $ \e A_{22} $ plays an important role in searching for an appropriate multiplier for the  linearized problem.  In this case, we  cannot determine the sign of $ \e A_{22} $. How to use the multiplier method to obtain a priori $H^1 $
energy estimate will be investigated in the future. }
\end{remark}
\par The rest of this paper will be arranged as follows. In Section 2, we   establish the basic and higher order energy estimates to the linearized second order hyperbolic-elliptic mixed system and construct approximated solutions by the Galerkin method to prove the structural stability of the background supersonic
flow within the class of cylindrical  irrotational flows. In Section 3, we employ the deformation-curl decomposition-Poisson for the  steady Euler
 Poisson system and design a two-layer iteration to demonstrate the existence of smooth cylindrical supersonic rotational flows. In Section 4, we prove  the structural stability of the background supersonic flow within the class of axisymmetric flows.
\section{The stability analysis within cylindrical  irrotational flows}\label{irrotational}\noindent

\par   In this section, we establish   the structural stability of the background supersonic flow within cylindrical supersonic irrotational flows.   Firstly, the potential function corresponding to the background supersonic flow is
$$\bar\psi(r,\theta)=\int_{r_0}^r \bar U_{1}(s)ds + J_2 \theta, \ \  (r,\th)\in \Omega.$$
 It is easy to see that $ \bar\psi $  has a nonzero circulation, which is not periodic in $\theta$. To avoid the trouble, we denote the difference between the flow and the background flow by
\begin{eqnarray}\label{3-5}
V_1(r,\th)= U_1(r,\th)- \bar U_{1}(r),\ \ V_2(r,\th)= U_2(r,\th)- \bar U_{2}(r),\ \ \Psi(r,\th)= \Phi(r,\th)-\bar\Phi(r), \ \  (r,\th)\in \Omega.
\end{eqnarray}
Then the density  can be rewritten as
\begin{eqnarray}\label{3-6}
\rho=
\bigg(\frac{\gamma-1}{\gamma e^{S_0}}\bigg(K_0+\Psi+\bar\Phi-\frac{1}{2}\left((V_1+\bar U_1)^2+(V_2+\bar U_2)^2\right)\bigg)\bigg)
^{\frac{1}{\gamma-1}}, \ \ \text{in}\ \ \Omega.
\end{eqnarray}
By substituting \eqref{3-5} and \eqref{3-6} into  \eqref{3-1}, it is easy to see that $(V_1,V_2,\Psi)$ satisfies
  \begin{equation}\label{3-7}
  \begin{cases}
 \partial_r\left(r\rho (V_1+\bar U_1)\right)+\partial_\theta\left(\rho (V_2+\bar U_2)\right)=0,\ \ &{\rm{in}}\ \ \Omega,\\
\partial_r(r V_2)-\partial_\theta V_1=0,\ \ &{\rm{in}}\ \ \Omega,\\
\left(\p_r^2+\frac 1 r\p_r+\frac{1}{r^2}\p_{\th}^2\right)(\Psi+\bar \Phi)=\rho-b,\ \ &{\rm{in}}\ \ \Omega,\\
(V_1, V_2,   \p_r\Psi)(r_0,\th)=(U_{1,en}, U_{2,en},E_{en})(\th)-(U_{1,0}, U_{2,0},E_0),\ \ &{\rm{on}}\ \ \Gamma_{en},\\
\Psi(r_1,\th)=\Phi_{ex}(\th)-\bar\Phi(r_1),\ \ &{\rm{on}}\ \ \Gamma_{ex}.\\
\end{cases}
\end{equation}
 Define the potential function
\begin{equation*}
\begin{aligned}
\psi(r,\theta)&=\int_{r_0}^rV_1(s,\theta)\de s+\int_0^\theta\bigg(r_0V_2(r_0,
s)+d_0\bigg)\de s\\
&=\int_{r_0}^rV_1(s,\theta)\de s+\int_0^{\theta}\bigg(r_0(U_{2,en}(s)-U_{2,0})+ d_0\bigg) \de s,\ \  (r,\th)\in \Omega,
\end{aligned}
\end{equation*}
where $d_0$ is introduced so that $\psi(r,\theta)=\psi(r,\theta+2\pi)$. Indeed, $$d_0=- \frac{r_0}{2\pi}\int_0^{2\pi}\bigg(U_{2,en}(s)-U_{2,0}\bigg)\de s.$$  Then $\psi$ is periodic in $\theta$ with period $2\pi$ and satisfies
\begin{align}\label{3-8}
\partial_r\psi=V_1,\ \ \partial_\theta\psi=r V_2+d_0, \ \ \psi(r_0,0)=0, \ \ \text{in}\ \ \Omega.
\end{align}
Substituting \eqref{3-8} into \eqref{3-7} obtains
\begin{equation}\label{3-9}
  \begin{cases}
  L_1(\psi,\Psi)=\p_r^2\psi
  -A_{22}(r,\th,V_1,V_2,\Psi)\p_\th^2\psi
   +\bigg(A_{12}(r,\th,V_1,V_2,\Psi)
   +A_{21}(r,\th,V_1,V_2,\Psi)\bigg)
  \p_{r\th}^2\psi\\
  \ \ +\bar a_1(r)\p_r\psi
 +\bar a_2(r)\p_\th\psi   +\bar b_1(r)\p_r\Psi
    +\bar b_2(r)\p_\th\Psi+\bar b_3(r)\Psi
    =F_1(r,\th,V_1,V_2,\Psi,\n \Psi),\ \ &\text{in}\ \ \Omega,\\
 L_2(\psi,\Psi)=\left(\p_r^2+\frac 1 r\p_r+\frac{1}{r^2}\p_{\th}^2\right)\Psi+ \bar a_3(r)\p_r\psi+\bar a_4(r)\p_\th\psi-\bar b_4(r)\Psi =F_2(r,\th,V_1,V_2,\Psi),\ \ &\text{in}\ \ \Omega,\\
\p_r\psi(r_0,\th)=U_{1,en}(\th)-U_{1,0}, \ \  \psi(r_0,\th)=\int_0^{\theta}\bigg(r_0(U_{2,en}(s)-U_{2,0})+ d_0\bigg) \de s  ,\ \ &\text{on}\ \ \Gamma_{en},\\
 \p_r\Psi(r_0,\th)=E_{en}(\th)-E_0, \ \ &\text{on}\ \ \Gamma_{en},\\ \Psi(r_1,\th)=\Phi_{ex}(\th)-\bar\Phi(r_1),\ \ &\text{on}\ \ \Gamma_{ex},\\
\end{cases}
\end{equation}
where
\begin{equation*}
  \begin{aligned}
 &A_{22}(r,\th, V_1,V_2,\Psi)=\frac{(\bar U_2+V_2)^2-c^2(K_0,\bar U_1+V_1, \bar U_2+V_2,\bar\Phi+\Psi)}{r^2\left(c^2(K_0,\bar U_1+V_1, \bar U_2+V_2,\bar\Phi+\Psi)-(\bar U_1+V_1)^2\right)},\\
 & A_{12}(r,\th, V_1,V_2,\Psi)= A_{21}(r,\th, V_1,V_2,\Psi)=-\frac{(\bar U_1+V_1) (\bar U_2+V_2)}{r\left(c^2(K_0,\bar U_1+V_1, \bar U_2+V_2,\bar\Phi+\Psi)-(\bar U_1+V_1)^2\right)},\\
 & \bar a_1(r)=\frac1{\bar c^2-{\bar U_1^2}}\bigg(-(\gamma+1)\bar U_{1}\bar U_{1}'+\frac{\bar c^2}r-\bar U_{2}\bar U_{2}'-(\gamma-1)\frac{\bar U_{1}^2}r +\bar E\bigg)\\
 &\qquad=\frac1{\bar c^2-{\bar U_1^2}}\bigg(\frac{(\gamma+1)(1+\bar M_{2}^2)}{r(1-\bar M_{1}^2)}\bar U_{1}^2 +\frac{(\gamma+1)\bar E}{(1-\bar M_{1}^2)\bar c^2}\bar U_{1}^2+\frac{\bar c^2+\bar U_{2}^2}r -\frac{(\gamma-1)\bar U_{1}^2}r
+\bar E\bigg), \\
& \bar a_2(r)=\frac1{\bar c^2-{\bar U_1^2}}\bigg(-(\gamma-1)\frac1r\bar U_{2}\bar U_{1}'-\frac{\bar U_{1}\bar U_{2}'}{r}-(\gamma-2)\frac{\bar U_{1}\bar U_{2}}{r^2} \bigg)\\
  &\qquad =\frac1{\bar c^2-{\bar U_1^2}}\bigg(\frac{2(1-\bar M_{1}^2)+(\gamma-1)(\bar M_{1}^2+\bar M_{2}^2)}{1-\bar M_{1}^2}+ \frac{(\gamma-1) r\bar E}{(1-\bar M_{1}^2)\bar c^2}\bigg)\frac{\bar U_{1}\bar U_{2}}{r^2},\\
  &
 \bar b_1(r)=\frac1{\bar c^2-{\bar U_1^2}}\bar U_{1}, \ \
     \bar b_2(r)=\frac1{\bar c^2-{\bar U_1^2}}\frac{\bar U_{2}}{r}, \ \ \bar b_3(r)=\frac1{\bar c^2-{\bar U_1^2}}\bigg((\gamma-1)\bar U_1'+(\gamma-1)\frac{\bar U_1}{r}\bigg),\\
 &\bar a_3(r)=\frac{\bar\rho \bar U_{1}}{\bar c^2}, \quad \quad  \bar a_4(r)=\frac{\bar\rho \bar U_{2}}{r\bar c^2}, \quad \quad   \bar  b_4(r)=\frac{\bar\rho}{ \bar c^2},\\
  &F_1(r,\th,V_1,V_2,\Psi,\n \Psi)=\frac1{c^2(K_0,\bar U_1+V_1, \bar U_2+V_2,\bar\Phi+\Psi)-(\bar U_1+V_1)^2}\bigg(\bar a_2d_0
    +\bar U_1'\bigg(\frac{\gamma+1}{2} V_1^2\\
    &\quad +\frac{\gamma-1}{2} V_2^2\bigg)+\bar U_{2}' V_1 V_2
  + \frac{\bar U_1}{r}\frac{\gamma-1}{2} (V_1^2+V_2^2)
   + \frac{V_1}{r}\bigg((\gamma-1)(-\Psi+\bar U_{1}V_1+ \bar U_{2} V_2)\\
  &\quad + \frac{\gamma-1}{2} (V_1^2+V_2^2)\bigg)-(U_1U_2-\bar U_{1}\bar U_{2})\frac{1}{r^2}(r V_2+d_0)  -V_1\p_r\Psi-\frac{1}rV_2\p_\th\Psi\bigg) \\
& \quad -\bigg(\frac1{c^2(K_0,\bar U_1+V_1, \bar U_2+V_2,\bar\Phi+\Psi)-(\bar U_1+V_1)^2}-\frac1{\bar c^2-\bar U_1^2}\bigg)\bigg(\bigg(-(\gamma+1)\bar U_{1}\bar U_{1}'+\frac{\bar c^2}r\\
&\quad -\bar U_{2}\bar U_{2}'-(\gamma-1)\frac{\bar U_{1}^2}r +\bar E\bigg)V_1+\bigg(-(\gamma-1)\frac1r\bar U_{2}\bar U_{1}'-\frac{\bar U_{1}\bar U_{2}'}{r}-(\gamma-2)\frac{\bar U_{1}\bar U_{2}}{r^2} \bigg)(r V_2+d_0)  \\
 &\quad+\bar U_{1}\p_r\Psi
    +\frac{\bar U_{2}}{r}\p_\th\Psi+\bigg((\gamma-1)\bar U_1'+(\gamma-1)\frac{\bar U_1}{r}\bigg)\Psi\bigg), \ \
\n=\left(\p_r,\frac{\p_\th}r\right),\\
&F_2(r,\th,V_1,V_2,\Psi)=\bar a_4d_0+ \rho-\bar\rho+\bar a_3 V_1+ r\bar a_4V_2-\bar b_4\Psi-(b-b_0).\\
  \end{aligned}
\end{equation*}
\par  Define an iteration set
\begin{equation}\label{3-9-s}
\mj_{\delta,r_1}=\bigg\{(\psi,\Psi)\in \left(H^4(\m)\right)^2: \|(\psi,\Psi)\|_{ H^4(\Omega)}\leq \delta\bigg\},
\end{equation}
where the positive constant  $\delta$ and $ r_1$   will be determined later. For any function $(\e\psi,\e\Psi)\in \mj_{\delta,r_1}$, set
\begin{eqnarray*}
\e V_1=\partial_r\e\psi,\ \  \e V_2=\frac{\partial_\theta\e\psi-d_0}{r}, \ \ \text{in}\ \ \Omega.
\end{eqnarray*}
 Then we  consider the following linearized problem:
 \begin{equation}\label{3-10-1}
  \begin{cases}
   L_1(\psi,\Psi)=\p_r^2\psi
  -A_{22}(r,\th,\e V_1,\e V_2,\e \Psi)\p_\th^2\psi
   +2A_{21}(r,\th,\e V_1,\e V_2,\e \Psi)
  \p_{r\th}^2\psi\\
  \qquad\qquad \ \ +\bar a_1(r)\p_r\psi
 +\bar a_2(r)\p_\th\psi+\bar b_1(r)\p_r\Psi
    +\bar b_2(r)\p_\th\Psi+\bar b_3(r)\Psi \\
 \qquad\qquad \ \    =F_1(r,\th,\e V_1,\e V_2,\e\Psi,\n \e\Psi),\ \ &\text{in}\ \ \Omega,\\
 L_2(\psi,\Psi)=\left(\p_r^2+\frac 1 r\p_r+\frac{1}{r^2}\p_{\th}^2\right)\Psi+ \bar a_3(r)\p_r\psi+\bar a_4(r)\p_\th\psi-\bar b_4(r)\Psi \\
 \qquad\qquad \ \ =F_2(r,\th,\e V_1,\e V_2,\e \Psi),\ \ &\text{in}\ \ \Omega,\\
\p_r\psi(r_0,\th)=U_{1,en}(\th)-U_{1,0}, \ \ \p_\th \psi(r_0,\th)=\int_0^{\theta}\bigg(r_0(U_{2,en}(s)-U_{2,0})+ d_0\bigg) \de s  ,\ \ &\text{on}\ \ \Gamma_{en},\\
 \p_r\Psi(r_0,\th)=E_{en}(\th)-E_0, \ \ &\text{on}\ \ \Gamma_{en},\\ \Psi(r_1,\th)=\Phi_{ex}(\th)-\bar\Phi(r_1),\ \ &\text{on}\ \ \Gamma_{ex}.\\
\end{cases}
\end{equation}
 Set
\begin{equation*}
\begin{aligned}
&\h\psi(r,\th)=\psi(r,\th)-\int_{0}^{\th} \bigg(r_0(U_{2,en}(s)-U_{2,0})+d_0\bigg)\de s, \ \  &(r,\th)\in \Omega, \\
&\h \Psi(r,\th)=\Psi(r,\th)-(r-r_1)(E_{en}(\th)-E_0)-(\Phi_{ex}(\th)-\bar\Phi(r_1)), \ \  &(r,\th)\in \Omega.
\end{aligned}
\end{equation*}
Then \eqref{3-10-1} can be transformed into
\begin{equation}\label{3-10}
  \begin{cases}
  L_1(\h\psi,\h\Psi)=\p_r^2\h\psi
  -A_{22}(r,\th, \e V_1, \e V_2,\e \Psi)\p_\th^2\h\psi +2A_{12}(r,\th,\e V_1, \e  V_2,\e \Psi)
   \p_{r\th}^2\h\psi\\
  \qquad\qquad \ \ +\bar a_1(r)\p_r\h\psi
   +\bar a_2(r)\p_\th\h\psi +\bar b_1(r)\p_r\h\Psi
    +\bar b_2(r)\p_\th\h\Psi+\bar b_3(r)\h\Psi\\
   \qquad\qquad \ \ =F_3(r,\th, \e V_1, \e V_2,\e \Psi, \n\e\Psi),\ \ &\text{in}\ \ \Omega,\\
 L_2(\h\psi,\h\Psi)=\left(\p_r^2+\frac 1 r\p_r+\frac{1}{r^2}\p_{\th}^2\right)\h\Psi+ \bar a_3(r)\p_r\h\psi+\bar a_4(r)\p_\th\h\psi-\bar b_4(r)\h\Psi\\
\qquad\qquad \ \ =F_4(r,\th, \e V_1, \e V_2,\e \Psi),\ \ &\text{in}\ \ \Omega,\\
\p_r\h\psi(r_0,\th)=F_5(\th),  \ \  \h\psi(r_0,\th)=
 \p_r\h\Psi(r_0,\th)=0, \ \ &\text{on}\ \ \Gamma_{en},\\ \h\Psi(r_1,\th)=0,\ \ &\text{on}\ \ \Gamma_{ex},\\
\end{cases}
\end{equation}
where
\begin{equation*}
  \begin{aligned}
 & F_3(r,\th, \e V_1,\e V_2,\e\Psi, \n\e\Psi)=F_1+A_{22}r_0 U_{2,en}'
 -\x a_2(r_0(U_{2,en}-U_{2,0})+d_0)-\x b_1(E_{en}-E_0)\\
 &\quad -\x b_2\bigg((r-r_1)E_{en}'+\Phi_{ex}'\bigg) -\x b_3\bigg((r-r_1)(E_{en}-E_0)+(\Phi_{ex}-\bar\Phi(r_1))\bigg), \\
 & F_4(r,\th, \e V_1,\e V_2,\e\Psi)=F_2-\bigg(\frac1r (E_{en}-E_0 ) +\frac{1}{r^2}\bigg((r-r_1)E_{en}''+\Phi_{ex}''\bigg)\bigg)-\x a_4(r_0(U_{2,en}-U_{2,0})+d_0)\\
 &\quad
  +\x b_4\bigg((r-r_1)(E_{en}-E_0)+(\Phi_{ex}-\bar\Phi(r_1))\bigg),  \\
  &F_5(\th)=  U_{1,en}(\th)-U_{1,0}.
  \end{aligned}
\end{equation*}
It   can be directly checked that there exists a constant $C>0$ depending only on $(\gamma,b_0, \rho_0, U_{1,0},U_{2,0},$\\$S_{0},E_{0})$ such that
 \begin{equation}\label{3-45-e}
 \begin{aligned}
 &\|  F_3(\cdot,
 \e V_1,\e V_2,\e\Psi, \n\e\Psi)\|_{H^3(\m)}+\|  F_4(\cdot,
 \e V_1,\e V_2,\e\Psi)\|_{H^2(\m)}+\|F_5\|_{H^3(\mathbb{T}_{2\pi})}\\
 &\leq C
  \bigg(\|  F_1(\cdot,
 \e V_1,\e V_2,\e\Psi, \n\e\Psi)\|_{H^3(\m)}+\|  F_2(\cdot,
 \e V_1,\e V_2,\e\Psi)\|_{H^2(\m)}+\|F_5\|_{H^3(\mathbb{T}_{2\pi})}\\
 &\quad\ \ \ \ +\omega_1(b,U_{ 1,en},U_{2, en},E_{ en},\Phi_{ ex})\bigg)\leq
 C\left(\delta^2+\omega_1(b,U_{ 1,en},U_{2, en},E_{ en},\Phi_{ ex})\right).
 \end{aligned}
 \end{equation}
  \par For a fixed  $ (\e\psi,\e\Psi)\in \mj_{\delta,r_1} $, denote
  \begin{equation*}
  \begin{aligned}
&\e A_{i2}(r,\th):=A_{i2}(r,\th,\e V_1, \e V_2,\e \Psi),\quad \ \  (r,\th)\in \Omega,\ i=1,2,\\
&\e F_j(r,\th):=F_j (r,\th,\e V_1,\e V_2,\e\Psi, \n\e\Psi),\   (r,\th)\in \Omega, \ j=1,3, \\
&\e F_k(r,\th):=F_k (r,\th,\e V_1,\e V_2,\e\Psi),\qquad    (r,\th)\in \Omega, \ k=2,4. \\
\end{aligned}
 \end{equation*}
  Then we have the following proposition.
  \begin{proposition}\label{pro2}
 Let  $R_0$ be given by  Proposition \ref{pro1}.
\begin{enumerate}[ \rm (i)]
\item For any $r_1\in(r_0,r_0+R_0)$, set
\begin{eqnarray*}
\begin{aligned}
\bar A_{22}(r)=\frac{\bar U_2^2-\bar c^2}{r^2(\bar c^2-{\bar U_1^2})}, \  \
\bar A_{12}(r)=
-\frac{\bar U_1\bar U_2}{r(\bar c^2-{\bar U_1^2})}, \ \    r\in [r_0, r_1].
\end{aligned}
\end{eqnarray*}
 Then there exists a constant $\bar\mu_0\in(0,1)$ depending only on $(\gamma,b_0, \rho_0,U_{1,0},U_{2,0}, S_{0},E_{0})$ such that
\begin{eqnarray}\label{2-7-a}
 {\bar\mu_0}\leq
\bar A_{22}(r)\leq \frac{1}{\bar\mu_0}, \
\ \  r\in [r_0, r_1] .
\end{eqnarray}
Furthermore, the coefficients $\bar A_{i2}$,   $\x a_j $ and $\x b_j  $ for $ i=1,2 $ and  $j=1,2,3,4$  are smooth functions. More precisely, for each $k\in\mathbb{Z}^+$, there exists a constant $\bar C_k>0$ depending only on $(\gamma,b_0, \rho_0,U_{1,0},U_{2,0}, S_{0}, E_{0},r_0,r_1,k)$ such that
\begin{equation}\label{2-7-b}
\Vert (\bar A_{12},\bar A_{22}, \x a_1,\x a_2,\x a_3,\x a_4, \x b_1, \x b_2,\x b_3,\x b_4)\Vert_{C^k([r_0,r_1])}\leq\bar C_k.
\end{equation}
\item For each $\epsilon_0\in(0, R_0)$,
there exists a positive constant $\delta_0$ depending only on  $(\gamma,b_0,   \rho_0, U_{1,0},U_{2,0},$\\$S_{0},E_{0},\epsilon_0)$ such that   for $r_1\in(r_0,r_0+R_0-\epsilon_0]$ and $\delta\leq \delta_0$, the coefficients $ ( \e A_{12},\e A_{22}) $ for $ (\e\psi,\e\Psi)\in \mj_{\delta, r_1}  $ with $\mj_{\delta, r_1} $    given by \eqref{3-9-s}  satisfying
\begin{eqnarray}\label{2-7-f}
 \|( \e A_{12},\e A_{22})-( \bar A_{12},\bar A_{22})\|_{H^3(\m)}\leq C\delta,
\end{eqnarray}
 and
 \begin{equation}\label{2-7-e}
\mu_0\leq \e A_{22}(r,\th)\leq \frac{1}{\mu_0},
\   \ (r,\th)\in\overline{\m}.
\end{equation}
Here  the positive constants  $C$ and $ \mu_0\in(0,1) $ depend only on $(\gamma,b_0, \rho_0, U_{1,0},U_{2,0},S_{0},E_{0},\epsilon_0)$.
\end{enumerate}
\end{proposition}
\subsection{Energy estimates  for the linearized problem }\noindent
\par In this subsection, we  derive the energy estimates for the problem   \eqref{3-10} under the assumptions that $(\e A_{12},\e A_{22})\in \left(C^{\infty}(\overline{\Omega})\right)^2$ and   \eqref{2-7-f}-\eqref{2-7-e}  hold.
\begin{lemma}\label{pro4}
 Let   $R_0 $ and  $ \delta_0$  be from  Propositions \ref{pro1} and  \ref{pro2}, respectively.
  For each $\epsilon_0\in(0,R_0)$, there exists a constant $\bar r_1\in(r_0,r_0+R_0-\epsilon_0]$ depending only $(\gamma,b_0,   \rho_0, U_{1,0},U_{2,0},S_{0},E_{0},\epsilon_0)$
such that for  $r_1\in(r_0,\bar r_1)$ and $\delta\leq \delta_0$, the classical  solution $ (\h\psi,\h\Psi) $ to  \eqref{3-10} satisfies  the energy estimate
\begin{equation}\label{3-11}
\|(\h\psi,\h\Psi)\|_{H^1(\m)}\leq C\left(\|\e F_3\|_{L^2(\m)}+\|\e F_4\|_{L^2(\m)}+\|F_5\|_{L^2(\mathbb{T}_{2\pi})}\right),
\end{equation}
 where    $C$  is a positive  constant depending only on $(\gamma,b_0,   \rho_0, U_{1,0},U_{2,0},S_{0},E_{0},r_0,r_1,\epsilon_0)$.
\end{lemma}
\begin{proof}
The basic  energy estimate \eqref{3-11} can be obtained   by  the multiplier method. More precisely, let  $ Q $ be  a smooth function of  $ r $ in $[r_0,r_1]$ to be determined, set
\begin{equation*}
\begin{aligned}
I_1(\h\psi,\h\Psi,Q)=\iint_{\m}QL_1(\h\psi,\h\Psi)\p_r\h\psi\de r \de \th, \ \ I_2(\h\psi,\h\Psi)=\iint_{\m}L_2(\h\psi,\h\Psi)\h\Psi\de r \de \th.
\end{aligned}
\end{equation*}
For $ I_1 $,  integration by parts leads to
$$ I_1(\h\psi,\h\Psi, Q)=\sum_{i=1}^4 I_{1i}(\h\psi,\h\Psi,Q), $$
where
\begin{equation*}
\begin{aligned}
I_{11}(\h\psi,\h\Psi,Q)&=\iint_{\m}Q\p_r^2 \h\psi \p_r\h\psi\de r \de \th \\
&=\int_{0}^{2\pi}\frac12\bigg(Q(\p_r \h\psi)^2 \bigg)(r_1,\th) \de \th-\int_{0}^{2\pi}\frac12\bigg(Q(\p_r \h\psi)^2 \bigg)(r_0,\th) \de \th\\
&\quad -\iint_{\m}\frac12Q'(\p_r \h\psi)^2\de r \de \th,\\
I_{12}(\h\psi,\h\Psi,Q)&=\iint_{\m}2Q\e A_{12}\p_{r\th}^2 \h\psi \h\p_r\h\psi\de r \de \th
=-\iint_{\m}\p_{\th}(Q\e A_{12})(\p_r \h\psi)^2\de r \de \th,\\
I_{13}(\h\psi,\h\Psi,Q)&=\iint_{\m}-Q\e A_{22}\p_{\th}^2 \h\psi \p_r\h\psi\de r \de \th \\
&=\iint_{\m}Q\p_\th \e A_{22}\p_r\h\psi\p_\th\h\psi \de r\de \th+\int_{0}^{2\pi}\frac12\bigg(Q\e A_{22}(\p_\th \h\psi)^2 \bigg)(r_1,\th) \de \th\\
&\quad-
\iint_{\m}\frac12\p_r(Q\e A_{22})(\p_\th \h\psi)^2\de r \de \th,\\
I_{14}(\h\psi,\h\Psi,Q)&=\iint_{\m}Q(\x a_1\p_r\h\psi+\x a_2\p_\th\h\psi+\x b_1\p_r\h\Psi+\x b_2\p_\th\h\Psi+\x b_3 \h\Psi)\p_r\h\psi\de r \de \th.\\
\end{aligned}
\end{equation*}
For $I_2 $,   integration by parts gives
\begin{equation*}
\begin{aligned}
I_2(\h\psi,\h\Psi)&=\iint_{\m}\left(\bigg(\p_r^2+\frac 1 r\p_r+\frac{1}{r^2}\p_{\th}^2\bigg)\h\Psi+ \bar a_3\p_r\h\psi+\bar a_4\p_\th\h\psi-\bar b_4\h\Psi\right)\h\Psi\de r\de \th\\
&=-\iint_{\m}(\p_r\h\Psi)^2\de r\de \th-\iint_{\m}\frac1{r^2}(\p_\th\h\Psi)^2\de r\de \th\\
&\quad-\int_{0}^{2\pi}\left(\frac{\h\Psi^2}{2r}\right)(r_0,\th)\de \th+\iint_{\m}\frac{\h\Psi^2}{2r^2}\de r\de \th+\iint_{\m}\x a_3\p_r\h\psi\h\Psi\de r\de \th\\
&\quad+\iint_{\m}\x a_4\p_\th\h\psi\h\Psi\de r\de \th-\iint_{\m}\x b_4\h\Psi^2\de r\de \th.
\end{aligned}
\end{equation*}
 Then one obtains
\begin{equation*}
\begin{aligned}
I_1(\h\psi,\h\Psi,Q)-I_2(\h\psi,\h\Psi)=J_1(\h\psi,\h\Psi,Q)+J_2(\h\psi,\h\Psi,Q)
+J_3(\h\psi,\h\Psi,Q),
\end{aligned}
\end{equation*}
where
\begin{equation*}
\begin{aligned}
J_1(\h\psi,\h\Psi,Q)&=\iint_{\m} \bigg(\bigg(-\frac12Q'-Q\p_\th\e A_{12}+\bar a_1Q\bigg)(\p_r\h\psi)^2
- \frac12\p_{r}(Q\e A_{22})(\p_\th \h\psi)^2
 +\bigg(\bigg(\x b_4-\frac1{2r^2}\bigg)\h\Psi^2\\
 &\qquad\quad\ \ +(\p_r\h\Psi)^2+\frac1{r^2}(\p_\th\h\Psi)^2\bigg)\bigg)\de r\de \th,
\\
\end{aligned}
\end{equation*}
\begin{equation*}
\begin{aligned}
J_2(\h\psi,\h\Psi,Q)&=\iint_{\m}\bigg(Q(\x a_2\p_\th\h\psi+\x b_1\p_r\h\Psi+\x b_2\p_\th\h\Psi+\x b_3 \h\Psi)\p_r\h\psi-\x a_3\p_r\h\psi\h\Psi-\x a_4\p_\th\h\psi\h\Psi\bigg)\de r\de \th,\\
J_3(\h\psi,\h\Psi,Q)&=\int_{0}^{2\pi}\left(\frac12\bigg(Q(\p_r \h\psi)^2 \bigg)(r_1,\th) -\frac12\bigg(Q(\p_r \h\psi)^2 \bigg)(r_0,\th)+\frac12\bigg(Q\e A_{22}(\p_\th \h\psi)^2 \bigg)(r_1,\th)\right.\\
&\qquad\quad\ \ \left.+\bigg(\frac{\h\Psi^2}{2r}\bigg)(r_0,\th)\right)\de \th.
 \end{aligned}
\end{equation*}
\par Note that $ U_{2,0}^2<2r_0^2\rho_0 $. Thus it follows from Proposition \ref{pro1} that
\begin{equation*}
\x b_4(r)-\frac1{2r^2}=\frac{\bar \rho(r)}{\bar c^2(r)}-\frac1{2r^2}>\frac{\bar \rho(r)}{\bar U_2^2(r)}-\frac1{2r^2}
>
\frac{\rho_0}{U_{2,0}^2}-\frac1{2r^2}>0, \ \  r\in [r_0,r_0+R_0).
\end{equation*}
For each $ r_1\in (r_0,r_0+R_0) $, set
\begin{equation}\label{2-19-2}
 \vartheta_{r_1}=\min_{r\in[r_0,r_1]}\bigg(\x b_4(r)-\frac1{2r^2}\bigg),
\end{equation}
 which is a strictly positive constant depending only on $(\gamma,b_0, \rho_0, U_{1,0},U_{2,0}, S_{0},E_{0},r_0,r_1)$. Next, one can use the Cauchy inequality to obtain that
  \begin{equation}\label{3-12}
\begin{split}
\vert J_2(\h\psi,\h\Psi,Q)\vert
&\leq\iint_{\Omega}\bigg(\frac{1}{8}(\p_r{\h\Psi})^2+\frac{1}{8r^2}(\p_\th{\h\Psi})^2
+\frac{3\vartheta_{r_1}}{8}{\h\Psi}^2
+h_1(Q,r)(\p_r\h\psi)^2+h_2(Q,r)(\p_\th\h\psi)^2\bigg)
\mathrm{d}r\mathrm{d}\theta,
\end{split}
\end{equation}
where
 \begin{equation*}
\begin{aligned}
h_1(Q,r)=\frac12\x a_2^2+Q^2\bigg(2\x b_1^2+2r^2\x b_2^2+2\frac{\bar b_3^2}{\vartheta_{r_1}}\bigg)+2\frac{\bar a_3^2}{\vartheta_{r_1}},\  \ \
h_2(Q,r)=\frac12Q^2+2\frac{\bar a_4^2}{\vartheta_{r_1}}.
\end{aligned}
\end{equation*}
If the function $Q $ satisfies $ Q(r)>0 $ in $[r_0,r_1]$, by using \eqref{2-7-e} and \eqref{3-12}, there holds
\begin{equation}\label{3-13}
\begin{aligned}
&I_1(\h\psi,\h\Psi,Q)-I_2(\h\psi,\h\Psi)\\
&\geq \iint_{\m}\bigg(\bigg(-\frac12Q'-Q\p_\th\e A_{12}+\bar a_1Q-h_1(Q,r)\bigg)(\p_r\h\psi)^2-\bigg(\frac12\p_{r}(Q\e A_{22})+h_2(Q,r)\bigg)(\p_\th \h\psi)^2\\
&\qquad\qquad +\bigg(\frac{5\vartheta_{r_1}}{8}{\h\Psi}^2+\frac{7}{8}(\p_{r}\h{\Psi})^2
+\frac{7}{8r^2}(\p_{\th}{\h\Psi})^2\bigg)\bigg)\mathrm{d}r\mathrm{d}\theta\\
& \quad
+\int_{0}^{2\pi}\bigg(\mu_0\bigg(\frac{Q}{2}\left((\p_r\h\psi)^2+(\p_\th\h\psi)^2\right)
\bigg)
(r_1,\th)
 -\frac{1}{2}\left(QF_5^2\right)(r_0,\th)\bigg)\mathrm{d}\theta.
\end{aligned}
\end{equation}
Therefore, we need to chose $ Q$ and $ r_1 $  such that for $ r\in[r_0,r_1] $, the function $ Q $ satisfies
\begin{equation}\label{3-14}
 \min_{[r_0,r_1]} Q(r)>0,
 \end{equation}
 \begin{equation}\label{3-15}
2\bigg( -\frac12Q'- Q\p_\th\e A_{12}+\bar a_1Q-h_1(Q,r)\bigg)\geq\lambda_0,
 \end{equation}
 and
 \begin{equation}\label{3-16}
 \frac{2}{\e A_{22}}\bigg(-\frac12\p_{r}(Q\e A_{22})-h_2(Q,r)\bigg)\geq\lambda_0
\end{equation}
    for some positive   constant $ \lambda_0 $.
    \par Set
\begin{equation*}
\vartheta_{1}=\min_{[r_0,r_0+R_0-\epsilon_0]}\bigg(\x b_4(r)-\frac1{2r^2}\bigg).
\end{equation*}
Then  for $ r_1\in(r_0,r_0+R_0-\epsilon_0]$, one has
\begin{equation}\label{mu_1}
\vartheta_{r_1}\geq \vartheta_{1}>0.
\end{equation}
By  \eqref{2-7-e} and \eqref{mu_1}, for  any $\delta\in(0,\delta_0]$, $r_1\in(r_0,r_0+R_0-\epsilon_0]$, it holds that
\begin{equation}\label{3-17}
\begin{cases}
\begin{aligned}
&2\bigg(-\frac12Q'- Q\p_\th\e A_{12}+\bar a_1Q-h_1(Q,r)\bigg)\\
&\geq -Q'+2(\bar a_1-\p_\th\e A_{12})Q-4Q^2\bigg(\x b_1^2+r^2\x b_2^2+\frac{\bar b_3^2}{\vartheta_{1}}\bigg)-\bigg(\x a_2^2+4\frac{\bar a_3^2}{\vartheta_{1}}\bigg)\\
&\frac{2}{\e A_{22}}\bigg(-\frac12\p_{r}(Q\e A_{22})-h_2(Q,r)\bigg)
\geq -Q'-2\frac{\p_r\e A_{22}}{2\e A_{22}}Q-\frac{Q^2}{\mu_0}-\frac{4\bar a_4^2}{\mu_0\vartheta_{1}}.
\end{aligned}
\end{cases}
\end{equation}
Define
\begin{equation*}
\begin{aligned}
&\mathfrak {a}_0:=
\max_{(r_0,r_0+R_0-\epsilon_0]}\bigg(\x a_2^2+4\frac{\bar a_3^2}{\vartheta_{1}}+\frac{4\bar a_4^2}{\mu_0\vartheta_{1}}\bigg),\quad
\mathfrak {a}_2:=\max_{(r_0,r_0+R_0-\epsilon_0]}\bigg(4\x b_1^2+4r^2\x b_2^2+4\frac{\bar b_3^2}{\vartheta_{1}}\bigg)+\frac{1}{\mu_0},\\
&\mathfrak {a}_{1,1}:=\min_{(r_0,r_0+R_0-\epsilon_0]}
\x a_1,\ \
\mathfrak {a}_{1,2}:=\min_{(r_0,r_0+R_0-\epsilon_0]}
\bigg(-\frac{\bar A_{22}'}{2\bar A_{22}}\bigg),\ \
\mathfrak {a}_1:=\min\{\mathfrak {a}_{1,1},\mathfrak {a}_{1,2}\}.
\end{aligned}
\end{equation*}
It follows from Proposition  \ref{pro1} that there exists a constant $C_b>0$ depending only on $(\gamma,b_0, \rho_0, U_{1,0},U_{2,0},$\\$ S_{0},E_{0},\epsilon_0)$ such that
\begin{equation*}
0<\mathfrak {a}_i\leq C_b, \ \ i=0,2, \ \ \hbox{and} \ \ \vert \mathfrak {a}_1\vert\leq C_b.
\end{equation*}
For any  $ (\e\psi,\e\Psi)\in \mj_{\delta, r_1}  $  with $(\delta,r_1)\in (0, \delta_0]\times (0, r_0+R_0-\epsilon_0]$, set
\begin{equation*}
\e{\mathfrak {a}}_{1,1}:=\min_{(r_0,r_0+R_0-\epsilon_0]}
(\bar a_1-\p_\th\e A_{12}),\ \
\e{\mathfrak {a}}_{1,2}:=\min_{(r_0,r_0+R_0-\epsilon_0]}
\bigg(-\frac{\p_r\e A_{22}}{2\e A_{22}}\bigg),\ \
\e{\mathfrak {a}}_1:=\min\{\e{\mathfrak {a}}_{1,1},\e{\mathfrak {a}}_{1,2}\}.
\end{equation*}
By Proposition  \ref{pro2} and Morrey's inequality, there exists a constant $\e C>0$ such that
\begin{equation}\label{3-18}
\e{\mathfrak {a}}_1\geq {\mathfrak {a}}_1-\e C\delta_0:=\h{\mathfrak {a}}_1,
\end{equation}
 where $\h{\mathfrak {a}}_1 $ is a constant  depending only on $(\gamma,b_0, \rho_0, U_{1,0},U_{2,0}, S_{0},E_{0},\epsilon_0)$.  It follows from \eqref{3-17} and \eqref{3-18}
  that the conditions \eqref{3-14}-\eqref{3-16} are satisfied as long as $ Q>0$ and solves the ODE
 \begin{equation}\label{ODE}
-Q'-\mathfrak {a}_2Q^2+2\h{\mathfrak {a}}_1Q-\mathfrak {a}_0=\lambda_0, \ \ r\in[r_0,r_1].
\end{equation}
Set $\h Q(r):=Q(r)-\frac{\h{\mathfrak {a}}_1}{{\mathfrak {a}}_2}$. Since ${\mathfrak {a}}_2>0$, the equation \eqref{ODE} can be rewritten  as
\begin{equation}\label{ODE1}
-\frac{\h Q^{\prime}}{\mathfrak {a}_2}=\h Q^2+\frac{1}{\mathfrak {a}_2}\bigg(\frac{\mathfrak {a}_0\mathfrak {a}_2-\h{\mathfrak {a}}_1^2}{\mathfrak {a}_2}+\lambda_0\bigg).
\end{equation}
By following the argument of Step 3 in the proof of \cite[Proposition 2.4]{BDXJ21}, one can divide two cases $ \mathfrak {a}_0\mathfrak {a}_2-\h{\mathfrak {a}}_1^2>0 $ and $ \mathfrak {a}_0\mathfrak {a}_2-\h{\mathfrak {a}}_1^2<0 $ to  solve
   \eqref{ODE1}. Furthermore,   there exists a  positive constant $\bar r_1\in(r_0,r_0+R_0-\epsilon_0]$  so that for $r_1\in(r_0,\bar r_1)$, one can find a  positive constant  $\lambda_0 $  depending only on $(\gamma,b_0, \rho_0, U_{1,0},U_{2,0}, S_{0}, E_{0},\epsilon_0)$   such  that
 the solution $ Q $ satisfies  the  conditions \eqref{3-14}-\eqref{3-16}.

\par With the aid of \eqref{2-7-e} and \eqref{3-13}-\eqref{3-16}, one derives
\begin{equation}\label{3-21}
\begin{aligned}
&I_1(\h\psi,\h\Psi, Q)-I_2(\h\psi,\h\Psi)\\
&\geq \iint_{\m}\bigg(\frac{\mu_0\lambda_0}2\left((\p_r\h\psi)^2+(\p_\th \h\psi)^2\right)+\bigg(\frac{5\vartheta_{r_1}}{8}{\h\Psi}^2+\frac{7}{8}(\p_{r}\h{\Psi})^2
+\frac{7}{8r^2}(\p_{\th}{\h\Psi})^2\bigg)\bigg)\mathrm{d}r\mathrm{d}\theta\\
& \quad
 -\int_{0}^{2\pi}\frac{ 1}{2}\left(QF_5^2\right)(r_0,\th)\mathrm{d}\theta.
\end{aligned}
\end{equation}
 Furthermore, \eqref{3-10} shows that
 \begin{equation}\label{3-22}
 \begin{aligned}
 &I_1(\h\psi,\h\Psi, Q)-I_2(\h\psi,\h\Psi)=\iint_{\m}\bigg( Q\e F_3\p_r\h\psi-\e F_4\h\Psi\bigg)\de r \de \th\\
 &\leq\iint_{\Omega}\bigg(\frac{\mu_0\lambda_0}{4}(\p_r\h\psi)^2
+\frac{\vert  Q\e F_3\vert^2}{\mu_0\lambda_0}
+\frac{\vartheta_{r_1}}{4}{\h\Psi}^2+\frac{\e F_4^2}{\vartheta_{r_1}}
\bigg)\mathrm{d}r\mathrm{d}\theta.
 \end{aligned}
 \end{equation}
 This, together with \eqref{3-21}, leads to
 \begin{equation}\label{3-23}
\begin{aligned}
&\iint_{\m}\bigg(\frac{\mu_0\lambda_0}4\left((\p_r\h\psi)^2+(\p_\th \h\psi)^2\right)+\bigg(\frac{\vartheta_{r_1}}{8}{\h\Psi}^2+\frac{7}{8}(\p_{r}\h{\Psi})^2
+\frac{7}{8r^2}(\p_{\th}{\h\Psi})^2\bigg)\bigg)\mathrm{d}r\mathrm{d}\theta\\
&\leq\iint_{\Omega}\bigg(
\frac{\vert Q\e F_3\vert^2}{\mu_0\lambda_0}
+\frac{\e F_4^2}{\vartheta_{r_1}}
\bigg)\mathrm{d}r\mathrm{d}\theta
+\int_{0}^{2\pi}\frac{1}{2}\left(QF_5^2\right)(r_0,\th)\mathrm{d}\theta.
\end{aligned}
\end{equation}
Then one can apply   the Poincar\'{e} inequality to  get
 \begin{equation*}
\|(\h\psi,\h\Psi)\|_{H^1(\m)}\leq C\left(\|\e F_3\|_{L^2(\m)}+\|\e F_4\|_{L^2(\m)}+\|F_5\|_{L^2(\mathbb{T}_{2\pi})}\right)
\end{equation*}
  for a constant   $C>0$ depending only on $(\gamma,b_0,   \rho_0, U_{1,0},U_{2,0},S_{0},E_{0},r_0,r_1,\epsilon_0)$.
\end{proof}
\par With the $ H^1 $ energy estimate,  one can  regard the system
 \eqref{3-10} as two individual
boundary value problems for $ \h\psi $ and $\h\Psi $ by leaving only the principal term on the left-hand side of
the equations to  derive the higher order derivative estimates for $ (\h\psi,\h\Psi) $.
\begin{lemma}\label{pro5}
  For fixed $\epsilon_0\in(0,R_0)$, let  $\bar r_1\in(r_0,r_0+R_0-\epsilon_0]$  be from  Lemma \ref{pro4}. There exist a positive constant $C$ depending only on $(\gamma,b_0,   \rho_0, U_{1,0},U_{2,0},S_{0},E_{0},r_0,r_1,\epsilon_0)$ and a sufficiently small constant  $ \delta_1\in(0,\frac{\delta_0}2]$
such that for  $ r_1\in (r_0,\bar r_1) $ and $\delta\leq \delta_1$, the classical   solution $ (\h\psi,\h\Psi) $ to  \eqref{3-10} satisfies
\begin{eqnarray}\label{H^4-1-1-1}
&&\Vert\h\psi\Vert_{H^4(\Omega)}
\leq C\left(\|\e F_3\|_{H^3(\m)}+\|\e F_4\|_{H^2(\m)}+\|F_5\|_{H^3(\mathbb{T}_{2\pi})}\right),\\\label{H^4-2-2}
&&\Vert\h\Psi\Vert_{H^4(\Omega)}
\leq C\left(\|\e F_3\|_{H^2(\m)}+\|\e F_4\|_{H^2(\m)}+\|F_5\|_{H^2(\mathbb{T}_{2\pi})}\right).
\end{eqnarray}
\end{lemma}
\begin{proof}
The proof is divided into three steps.
\par { \bf Step 1. A priori $H^2$-estimate for $ \h\Psi $.}
\par
   $ \h\Psi $ can be regarded as a classical solution to the linear boundary value problem:
 \begin{equation}\label{3-24}
\begin{cases}
\left(\p_r^2+\frac 1 r\p_r+\frac{1}{r^2}\p_{\th}^2\right)\h\Psi-\bar b_4\h\Psi=\e F_4-\bar a_3\p_r\h\psi-\bar a_4\p_\th\h\psi,\ \ &{\rm{in}}\ \ \Omega,\\
\p_r \h\Psi(r_0,\th)=0,\ \ &{\rm{on}}\ \ \Gamma_{en},\\
\h\Psi(r_1,\th)=0, \ \ &{\rm{on}}\ \ \Gamma_{ex}.\\
\end{cases}
\end{equation}
Then we   apply \cite[Theorems 8.8 and 8.12]{GT98} to obtain
 \begin{equation*}
 \begin{aligned}
\Vert\h\Psi\Vert_{H^2(\Omega)}
\leq C\left(\Vert\e F_4\Vert_{L^2(\Omega)}+\Vert\p_r\h\psi\Vert_{L^2(\Omega)}+
\Vert\p_\th\h\psi\Vert_{L^2(\Omega)}
+\Vert\h\Psi\Vert_{L^2(\Omega)}\right).
\end{aligned}
\end{equation*}
Combining this estimate with \eqref{3-11} gives
\begin{equation}\label{3-25}
 \begin{aligned}
\Vert\h\Psi\Vert_{H^2(\Omega)}
\leq C\left(\|\e F_3\|_{L^2(\m)}+\|\e F_4\|_{L^2(\m)}+\| F_5\|_{L^2(\mathbb{T}_{2\pi})}\right),
\end{aligned}
\end{equation}
 where the constant $C>0$ depends only on $(\gamma,b_0, \rho_0, U_{1,0},U_{2,0},S_{0},E_{0},r_0,r_1,\epsilon_0)$.
 \par { \bf Step 2. A priori $H^2$-estimate for $\h\psi$.}
 \par
 The first equation in \eqref{3-10} can be rewritten as
  \begin{equation}\label{3-26}
\p_r^2 \h\psi- \e A_{22}\p_\theta^2\h\psi
+ 2\e A_{12}\p_{r\theta}^2\h\psi
 =-\bar a_1\p_r\h\psi
 -\bar a_2 \p_\th \h\psi  -\bar b_1\p_r\h\Psi
    -\bar b_2\p_\th\h\Psi-\bar b_3\h\Psi
 +\e F_3, \ \ {\rm{in}} \ \ \Omega.
\end{equation}
Let $ q=\p_r\h\psi $. Differentiating \eqref{3-26} with respect to $r$ yields
\begin{equation}\label{3-27}
\begin{aligned}
&\p_r^2 q-\e A_{22}\p_\theta^2q
+ 2\e A_{12}\p_{r\theta}^2q=
 \p_r\e A_{22}\p_\theta^2\h\psi-2\p_r\e A_{12}\p_{\theta}q-\x a_1\p_rq-\bar a_2 \p_\th q \\
 &
 -\x a_1'q-\x a_2'\p_\th \h\psi
 +\p_r(-\bar b_1\p_r\h\Psi
    -\bar b_2\p_\th\h\Psi-\bar b_3\h\Psi
 +\e F_3),\ \ {\rm{in}} \ \ \Omega.
 \end{aligned}
 \end{equation}
 Set
 $$ \Omega_t=\{(r,\theta):r_0<r<t,\, \theta\in \mathbb{T}_{2\pi}\}.$$  By  integrating by parts with  the  boundary conditions in \eqref{3-10}, there holds
 \begin{equation*}
\begin{aligned}
V(q)&=\iint_{\Omega_t}\bigg(\p_r^2 q-\e A_{22}\p_\theta^2q
+ 2\e A_{12}\p_{r\theta}^2q\bigg)\p_r q \de r \de \th\\
 &=\int_{0}^{2\pi}\bigg(\bigg(\frac{1}2(\p_r q)^2+\frac{\e A_{22}}{2}(\p_\th q)^2\bigg)(t,\th)
 -\bigg(\frac{1}2(\p_r q)^2+\frac{\e A_{22}}{2}(\p_\th q)^2\bigg)(r_0,\th)\bigg)\de \th\\
 &\quad+\iint_{\Omega_t}\bigg(-\p_\th\e A_{12}
 (\p_rq)^2+\p_\th\e A_{22}\p_\th q\p_r q-\frac{\p_r\e A_{22}}{2}(\p_\th q)^2\bigg)\de r \de \th.
 \end{aligned}
 \end{equation*}
One can fix a small constant $ \delta_1\in(0, \frac{\delta_0}2] $  so that if  $\delta \leq  \delta_1$,
 it follows from \eqref{2-7-f} that
 \begin{eqnarray}\label{2-7-f-s}
 \|( \e A_{12},\e A_{22})\|_{H^3(\m)}\leq C
\end{eqnarray}
for some constant $C>0$ depending only on  $ (\gamma,b_0, \rho_0, U_{1,0},U_{2,0},S_{0},E_{0},\epsilon_0)$. Combining  \eqref{2-7-e}, \eqref{2-7-f-s}, the Morrey inequality and the Cauchy inequality    yields that \begin{equation}\label{3-28}
\begin{aligned}
V(q)&\geq \int_{0}^{2\pi}\bigg(\frac{\mu_0}2\left((\p_r q)^2+(\p_\th q)^2\right)(t,\th)
 -\frac1{2\mu_0}\left((\p_r q)^2+(\p_\th q)^2\right)(r_0,\th)\bigg)\de \th\\
 &\quad -C\iint_{\Omega_t}\left((\p_r q)^2+(\p_\th q)^2\right)\de r \de \th.
 \end{aligned}
 \end{equation}  On the other hand,
 \begin{equation*}
\begin{aligned}
V(q)=&\iint_{\Omega_t}\bigg(\p_r^2 q-\e A_{22}\p_\theta^2q
+ 2\e A_{12}\p_{r\theta}^2q\bigg)\p_r q \de r \de \th\\
 &=\iint_{\Omega_t}\bigg( \p_r\e A_{22}\p_\theta^2\h\psi-2\p_r\e A_{12}\p_{\theta}q
 -\x a_1\p_rq-\bar a_2 \p_\th q-\x a_1'q-\x a_2'\p_\th \h\psi\\
 &\qquad\qquad +\p_r(-\bar b_1\p_r\h\Psi
    -\bar b_2\p_\th\h\Psi-\bar b_3\h\Psi
 +\e F_3)\bigg)\p_r q\de r \de \th.\\
 \end{aligned}
 \end{equation*}
 Then one can   use the Cauchy inequality  and combine \eqref{3-11},    \eqref{3-25},  and \eqref{2-7-f-s} to  obtain
 \begin{equation}
\label{3-29}
\begin{split}
&\int_{0}^{2\pi}\frac{\mu_0}2\bigg((\p_r q)^2+(\p_\th q)^2\bigg)(t,\th)\de \th
\leq
C\left(\|\e F_3\|_{H^1(\m)}+\|\e F_4\|_{L^2(\m)}+\|F_5\|_{L^2(\mathbb{T}_{2\pi})}\right)^2\\
&+\int_{0}^{2\pi}\frac1{\mu_0}\bigg((\p_r q)^2+(\p_\th q)^2\bigg)(r_0,\th)\de \th
+C\iint_{\Omega_t}\left((\p_r q)^2+(\p_\th q)^2+(\p_{\th}^2\h\psi)^2\right)\mathrm{d}r\mathrm{d}\theta.
\end{split}
\end{equation}
The estimate constant $C>0 $ in \eqref{3-28} and \eqref{3-29} depends only on $(\gamma,b_0,   \rho_0, U_{1,0},U_{2,0},S_{0},E_{0},r_0,r_1,\epsilon_0)$.
\par Owing to $\p_r\h\psi(r_0,\th)=F_5(\th)$, one gets
\begin{equation}\label{3-30}
 \int_{0}^{2\pi}(\p_\th q)^2(r_0,\th)\de \th= \int_{0}^{2\pi}( F_5')^2(\th)\de \th\leq C\| F_5\|_{H^1(\mathbb{T}_{2\pi})}^2.
 \end{equation}
 Furthermore, it follows from \eqref{3-26}  that
   \begin{equation*}
   \begin{aligned}
&\int_{0}^{2\pi}(\p_rq)^2
(r_0,\th) \de \th
 =\int_{0}^{2\pi}\bigg(\bigg(\e A_{22}\p_\theta^2\h\psi
- 2\e A_{12}\p_{\theta}q-\bar a_1q
 -\bar a_2 \p_\th \h\psi  -\bar b_1\p_r\h\Psi\\
&\qquad\qquad\qquad\qquad\qquad\qquad    -\bar b_2\p_\th\h\Psi-\bar b_3\h\Psi
 +\e F_3\bigg)\p_rq\bigg)(r_0,\th) \de \th.
 \end{aligned}
\end{equation*}
Due to $ \h\psi(r_0,\th)=\p_r \h\Psi(r_0,\th)=0$, one can use  the Cauchy inequality and the trace theorem   together with \eqref{3-11},    \eqref{3-25} and  \eqref{3-30} to obtain  that
\begin{equation}\label{3-31}
  \int_{0}^{2\pi}(\p_r q)^2(r_0,\th)\de \th\leq \left(\|\e F_3\|_{H^1(\m)}+\|\e F_4\|_{L^2(\m)}+\|F_5\|_{H^1(\mathbb{T}_{2\pi})}\right)^2.
 \end{equation}
 Note that
 \begin{equation*}
   \begin{aligned}
\iint_{\Omega_t}\e A_{22}(\p_\theta^2\h\psi)^2
\de r \de \th&=
 \iint_{\Omega_t}\bigg(\p_rq
-2\e A_{12}\p_{\theta}q
 +\bar a_1\p_rq
 +\bar a_2 \p_\th \h\psi  +\bar b_1\p_r\h\Psi
    +\bar b_2\p_\th\h\Psi+\bar b_3\h\Psi
 -\e F_3\bigg)\p_\theta^2\h\psi\de r \de \th.
 \end{aligned}
\end{equation*}
 By using the Cauchy inequality and combining  the estimates \eqref{2-7-e}, \eqref{3-11} and \eqref{2-7-f-s}, there holds
\begin{equation}\label{3-32}
\begin{aligned}
\iint_{\Omega_t}(\p_\theta^2\h\psi)^2
\de r \de \theta
\leq
C_1\bigg(\left(\|\e F_3\|_{L^2(\m)}+\|\e F_4\|_{L^2(\m)}+\|F_5\|_{L^2(\mathbb{T}_{2\pi})}\right)^2
+\iint_{\Omega_t}\left((\p_r q)^2+(\p_\th q)^2\right)\de r\mathrm{d}\theta\bigg).
\end{aligned}
\end{equation}
Collecting the estimates \eqref{3-29}-\eqref{3-32} leads to
\begin{equation}\label{3-33}
\begin{aligned}
 &\frac{\de\bigg(\iint_{\Omega_t}\left((\p_r q)^2+(\p_\th q)^2\right)\de r\mathrm{d}\theta\bigg)}
 {\de t} \\
 &\leq C_2\left(\|\e F_3\|_{H^1(\m)}+\|\e F_4\|_{L^2(\m)}+\| F_5\|_{H^1(\mathbb{T}_{2\pi})}\right)^2+C_3\iint_{\Omega_t}\bigg((\p_r q)^2+(\p_\th q)^2\bigg)\de r\mathrm{d}\theta.
 \end{aligned}
\end{equation}
By the Gronwall inequality, one derives
\begin{equation}\label{3-34}
\begin{aligned}
 \iint_{\Omega}\left((\p_r q)^2+(\p_\th q)^2\right)\de r\mathrm{d}\theta
 \leq C_4\left(\|\e F_3\|_{H^1(\m)}+\|\e F_4\|_{L^2(\m)}+\| F_5\|_{H^1(\mathbb{T}_{2\pi})}\right)^2.
 \end{aligned}
\end{equation}
The estimate constants  $C_i>0 \ (i=1,2,3,4) $ depend only on $(\gamma,b_0, \rho_0, U_{1,0},U_{2,0},S_{0},E_{0},r_0,r_1,\epsilon_0)$. Therefore, \eqref{3-34}, together with \eqref{3-32}, yields that
\begin{equation}\label{H2-estimate-vm-final}
  \|\h\psi\|_{H^2(\m)}\le C\left(\|\e F_3\|_{H^1(\m)}+\|\e F_4\|_{L^2(\m)}+\| F_5\|_{H^1(\mathbb{T}_{2\pi})}\right)^2,
\end{equation}
where the constant $C>0 $ depends only on $(\gamma,b_0, \rho_0, U_{1,0},U_{2,0},S_{0},E_{0},r_0,r_1,\epsilon_0)$.
\par { \bf Step 3. Estimate for higher order weak derivatives of $(\h\psi,\hat\Psi)$.}
\par For the higher order weak derivatives of $(\h\psi,\hat\Psi)$, we continue the bootstrap argument. Adapting the argument  in Steps 1 and 2 and combining \eqref{3-25} and \eqref{H2-estimate-vm-final} yield
\begin{eqnarray}\label{H^4-1-1-1-1}
&&\Vert\h\psi\Vert_{H^4(\Omega)}
\leq C\left(\|\e F_3\|_{H^3(\m)}+\|\e F_4\|_{H^2(\m)}+\|F_5\|_{H^3(\mathbb{T}_{2\pi})}\right),\\\label{H^4-2-2-2}
&&\Vert\h\Psi\Vert_{H^4(\Omega)}
\leq C\left(\|\e F_3\|_{H^2(\m)}+\|\e F_4\|_{H^2(\m)}+\|F_5\|_{H^2(\mathbb{T}_{2\pi})}\right).
\end{eqnarray}
The above constant $ C>0 $   depends only on $(\gamma,b_0, \rho_0, U_{1,0},U_{2,0},S_{0},E_{0},r_0,r_1,\epsilon_0)$. Therefore,
the proof of Lemma \ref{pro5} is completed.
\end{proof}
\subsection{The well-posedness of the linearized  problem}\noindent
 \par With the aid of  Lemmas \ref{pro4} and \ref{pro5}, the following well-posedness of the linearized  problem \eqref{3-10} can be obtained.
\begin{proposition}\label{pro6}

  Fix  $\epsilon_0\in(0,R_0)$,   let  $\bar r_1\in(r_0,r_0+R_0-\epsilon_0]$ and $ \delta_1\in(0, \frac{\delta_0}2] $ be from Lemmas \ref{pro4} and \ref{pro5}, respectively.   For  $r_1\in(r_0,\bar r_1)$ and $\delta\leq \delta_1$,  the  linear boundary value problem \eqref{3-10}  associated with $(\e\psi,\e\Psi)\in \mj_{\delta,r_1} $  has a unique solution  $ (\h\psi,\h\Psi)\in \left(H^4(\m)\right)^2$ that satisfies
\begin{eqnarray}\label{H^4-1}
&&\Vert\h\psi\Vert_{H^4(\Omega)}
\leq C\left(\|\e F_3\|_{H^3(\m)}+\|\e F_4\|_{H^2(\m)}+\|F_5\|_{H^3(\mathbb{T}_{2\pi})}\right),\\\label{H^4-2}
&&\Vert\h\Psi\Vert_{H^4(\Omega)}
\leq C\left(\|\e F_3\|_{H^2(\m)}+\|\e F_4\|_{H^2(\m)}+\|F_5\|_{H^2(\mathbb{T}_{2\pi})}\right)
\end{eqnarray}
 for some  constant $C>0$ depending only on $(\gamma,b_0, \rho_0,U_{1,0},U_{2,0}, S_{0},E_{0},r_0,r_1,\epsilon_0)$.

\end{proposition}
\begin{proof}
 Since the solution is periodic in $\th$, it is natural to use the Fourier series to construct the approximate solution to the linearized problem. Note that the $H^4(\Omega)$ energy estimate is obtained only for the linearized problem with smooth coefficients, one needs to mollify the coefficients in \eqref{3-10}.
 \par Note that $(\e A_{12},\e A_{22},\e F_3,\e F_4)\in \left(H^3(\m)\right)^3\times H^2(\m)$, there exists sequences of smooth functions
 $(\e A_{12}^\eta,\e A_{22}^\eta,\e F_3^\eta,\e F_4^\eta) $
 such that  as $ \eta\rightarrow0 $, one has
 \begin{equation*}
 \|\e A_{12}^\eta-\e A_{12}\|_{H^3(\m)}\rightarrow0, \ \  \|\e A_{22}^\eta-\e A_{22}\|_{H^3(\m)}\rightarrow0, \ \ \|\e F_3^\eta-\e F_3\|_{H^3(\m)}\rightarrow0, \ \
 \|\e F_4^\eta-\e F_4\|_{H^2(\m)}\rightarrow0.
 \end{equation*}
Then a  approximation problem with smooth coefficients of \eqref{3-10} is obtained as follows:
  \begin{equation}\label{3-35}
  \begin{cases}
  L_1^\eta(\h\psi,\h\Psi)=\p_r^2\h\psi
  -\e A_{22}^\eta\p_\th^2\h\psi +2\e A_{12}^\eta
   \p_{r\th}^2\h\psi
   +\bar a_1\p_r\h\psi
   +\bar a_2\p_\th\h\psi \\
   \qquad\qquad \ \ +\bar b_1\p_r\h\Psi
    +\bar b_2\p_\th\h\Psi+\bar b_3\h\Psi
    =\e F_3^\eta,\ \ &\text{in}\ \ \Omega,\\
 L_2(\h\psi,\h\Psi)=\left(\p_r^2+\frac 1 r\p_r+\frac{1}{r^2}\p_{\th}^2\right)\h\Psi+ \bar a_3\p_r\h\psi+\bar a_4\p_\th\h\psi-\bar b_4\h\Psi =\e F_4^\eta,\ \ &\text{in}\ \ \Omega,\\
\p_r\h\psi(r_0,\th)=F_5(\th),  \ \  \h\psi(r_0,\th)=
 \p_r\h\Psi(r_0,\th)=0, \ \ &\text{on}\ \ \Gamma_{en},\\ \h\Psi(r_1,\th)=0,\ \ &\text{on}\ \ \Gamma_{ex}.\\
\end{cases}
\end{equation}
Choose the standard orthonormal basis $\{\beta_j(\th)\}_{j=1}^{\infty}$ of $L^2(\mathbb{T}_{2\pi})$, where for each positive integer $k\in\mathbb{N}$:
\begin{equation*}
\beta_0(\th)=\frac{1}{\sqrt{2\pi}},\ \ \beta_{2k-1}(\th)=\frac{1}{\sqrt{\pi}}\sin (k\th),\ \ \beta_{2k}(\th)= \frac{1}{\sqrt{\pi}}\cos (k \th), \ \ k=1,2,\cdots.
\end{equation*}
For any fixed $\eta>0$, we consider the linear boundary value problem \eqref{3-35}.
 For each $m=1,2,\cdots$, let $(\h\psi_m^\eta, \h\Psi_m^\eta)$ be of the form
\begin{equation}\label{m}
\h\psi_m^\eta(r,\theta)=\sum_{j=0}^{2m}\my_j^\eta(r)\beta_j(\theta),\quad \h\Psi_m^\eta(r,\theta)=\sum_{j=0}^{2m}\ml_j^\eta(r)\beta_j(\theta), \ \  (r,\th)\in \m.
\end{equation}
 Then we find a solution $(\h\psi_m^\eta,\h\Psi_m^\eta)$ satisfying
\begin{equation}\label{Leqs4}
\begin{cases}
\begin{aligned}
&\int_0^{2\pi} {L}^{\eta}_1(\h\psi_m^\eta, \h\Psi_m^\eta)(r,\theta)\beta_{k}(\theta)\de \th = \int_0^{2\pi}\e F^{\eta}_3(r,\theta)\beta_{k}(\theta)\de \th ,  \ \ & r\in (r_0, r_1),\\
&\int_0^{2\pi} {L}_2(\h\psi_m^\eta, \h\Psi_m^\eta)(r,\theta)\beta_{k}(\theta)\de \th = \int_0^{2\pi}\e F^{\eta}_4(r,\theta)\beta_{k}(\theta)\de \th ,\ \   & r\in (r_0, r_1),\\
&\h\psi_m^\eta=\partial_r\h\Psi_m^\eta=0,  \
\partial_r\h\psi_m^\eta=\sum_{j=0}^{2m}\beta_j(\theta)\int_0^{2\pi} F_5(\theta)\beta_j(\theta)\de \th,  \ &\hbox{on} \ \ \Gamma_{en},\\
&\h\Psi_m^\eta=0,\ \  &\hbox{on}\ \ \Gamma_{ex},\\
\end{aligned}
\end{cases}
\end{equation}
for  all $k=0,1,...,2m$. If $\h\psi_m^\eta  $ and $\h\Psi_m^\eta$ are smooth functions and solve \eqref{Leqs4},  one obtains
\begin{equation*}
\begin{split}
&\iint_{\Omega}{L}^{\eta}_1(\h\psi_m^\eta,\h{\Psi}_m^\eta)Q\partial_r\h\psi_m^\eta
\mathrm{d}r\mathrm{d}\theta
-\iint_{\Omega}{L}_2(\h\psi_m^\eta,\h{\Psi}_m^\eta)\h\Psi_m^\eta\mathrm{d}r\mathrm{d}\theta\\
&=\int_{r_0}^{r_1}Q\int_0^{2\pi} \sum_{j=0}^{2m}(\my_j^\eta)'{L}^{\eta}_1(\h\psi_m^\eta, \h\Psi_m^\eta)\beta_{j}\de \th\mathrm{d}r
-\int_{r_0}^{r_1}\int_0^{2\pi}\sum_{j=0}^{2m}\ml_j^\eta
{L}_2(\h\psi_m^\eta, \h\Psi_m^\eta)\beta_{j}\de \th\mathrm{d}r\\
&=\iint_{\Omega}Q\e F^{\eta}_3\partial_r\h\psi_m^\eta\mathrm{d}r\de \th
-\iint_{\Omega}\e F^{\eta}_4\h\Psi_m^\eta\mathrm{d}r \de \th,
\end{split}
\end{equation*}
where   the function $ Q $  is constructed in Lemma \ref{pro4}. Then it holds that
\begin{equation}\label{H1}
\Vert(\h\psi_m^\eta,\h\Psi_m^\eta)\Vert_{H^1(\Omega)}
\leq C\left(\Vert\e F_3^\eta\Vert_{L^2(\Omega)}+\Vert\e F_4^\eta\Vert_{L^2(\Omega)}
+\Vert  F_5\Vert_{L^2(\mathbb{T}_{2\pi})}\right),
\end{equation}
where the constant $ C>0 $  is independent of $\eta$ and $m$.
\par Define
\begin{equation*}
\begin{aligned}
&\mathcal{S}_{1,jk}^\eta(r)= \int_0^{2\pi} \left(2\e A_{12}^{\eta}(r,\theta) \beta_j'(\theta) +\x a_1(r) \beta_j(\theta)\right)\beta_k(\theta) \de \th, \ \ \mathcal{S}_{3,jk}(r)=\int_0^{2\pi}\x b_1(r)\beta_j(\theta) \beta_k(\theta) \de \th,\\
&\mathcal{S}_{2,jk}^\eta(r)= \int_0^{2\pi} \left(\e A_{22}^{\eta}(r,\theta) \beta_j''(\theta)
+\x a_2(r)\beta_j'(\theta)\right) \beta_k(\theta) \de \th,\ \ \mathcal{S}_{4,jk}(r)=\int_0^{2\pi}\x b_2(r)\beta_j'(\theta) \beta_k(\theta) \de \th,\\
&  \mathcal{S}_{5,jk}=\int_0^{2\pi}\beta_j''(\theta) \beta_k (\theta)\de \th,\ \ \mathcal{S}_{6,jk}(r)=\int_0^{2\pi}\x a_3(r)\beta_j \beta_k \de \th,\ \
\mathcal{S}_{7,jk}(r)=\int_0^{2\pi}\x a_4(r)\beta_j'(\theta) \beta_k(\theta) \de \th,\\
&\e F^{\eta}_{3,k}(r)=\int_0^{2\pi}\e F^{\eta}_3(r,\theta)\beta_{k}(\theta)\de \th, \ \ \  \e F^{\eta}_{4,k}(r)=\int_0^{2\pi}\e F^{\eta}_4(r,\theta)\beta_{k}(\theta)\de \th.
\end{aligned}
\end{equation*}
Then the system \eqref{Leqs4} can be rewritten as
\begin{equation}\label{Leqs5}
\begin{cases}
\begin{aligned}
  &(\my_k^\eta)''(r)+\sum_{j=0}^{2m}\bigg(\mathcal{S}_{1,jk}^\eta(\my_j^\eta)'
  +\mathcal{S}_{2,jk}^\eta\my_j^\eta+\mathcal{S}_{3,jk}(\ml_j^\eta)'
  +\mathcal{S}_{4,jk}\ml_j^\eta\bigg)(r)\\
 &\quad  +\x b_3(r)\ml_k^\eta(r)=\e F^{\eta}_{3,k}(r),\ \ r\in (r_0, r_1),\\
&(\ml_k^\eta)''(r)+\frac1r(\ml_k^\eta)'(r)+\sum_{j=0}^{2m}\bigg(
\frac{\mathcal{S}_{5,jk}}{r^2}
\ml_j^\eta(r)+\left(\mathcal{S}_{6,jk}(\my_j^\eta)'+
\mathcal{S}_{7,jk}\my_j^\eta\right)(r)\bigg)\\
&\quad  -\x b_4(r)\my_k^\eta(r)=\e F^{\eta}_{4,k}(r),\ \ r\in (r_0, r_1),\\
&\my_k^\eta(r_0)=(\ml_k^\eta)'(r_0)=0,\ \
(\my_k^\eta)'(r_0)=\int_0^{2\pi} F_5(\th)\beta_k(\th)\de \th,\ \
\ml_k^\eta(r_1)=0.
 \end{aligned}
\end{cases}
\end{equation}
 For each $m\in\mathbb{N}$, set
\begin{equation*}
\begin{split}
\mathbf{Y}_1&:=(\my_{0}^\eta,...,\my_{2m}^\eta), \ \ \mathbf{Y}_2:=((\my_{0}^\eta)',...,(\my_{2m}^\eta)'),\\
\mathbf{Y}_3&:=(\ml_{0}^\eta,...,\ml_{2m}^\eta), \ \ \mathbf{Y}_4:=((\ml_{0}^\eta)',...,(\ml_{2m}^\eta)'),\\
\mathbf{Y}&:=(\mathbf{Y}_1,\mathbf{Y}_2,\mathbf{Y}_3, \mathbf{Y}_4)^{T},
\end{split}
\end{equation*}
and  define a projection mapping $\Pi_{i}\ (i=1,2)$ by
\begin{equation*}
\Pi_1(\mathbf{Y})=(\mathbf{Y}_1,\mathbf{Y}_2,\mathbf{0},\mathbf{Y}_4)^{T} \quad \hbox{and} \quad
\Pi_2(\mathbf{Y})=(\mathbf{0},\mathbf{0},\mathbf{Y}_3,\mathbf{0})^{T}.
\end{equation*}
Then \eqref{Leqs5} reduces to a first order ODE system:
\begin{equation}\label{3-36}
\begin{cases}
\begin{aligned}
&\mathbf{Y}^{\prime}=\mathbb{S}\mathbf{Y}+\mathbf{F},\\
&\Pi_1(\mathbf{Y})(r_0)=\bigg(\mathbf{0},\int_0^{2\pi} F_5(\th)\beta_0(\th)\de \th ,\cdots,\int_0^{2\pi} F_5(\th)\beta_{2m}(\th)\de \th, \mathbf{0},\mathbf{0}\bigg)^{T}:=\mathbf{F}_0,\\
 &\Pi_2(\mathbf{Y})(r_1)=\mathbf{0},
 \end{aligned}
\end{cases}
\end{equation}
 where $\mathbb{S}:(r_0,r_1)\rightarrow \mathbb{R}^{4(m+1)^2\times 4(m+1)^2}$ and $\mathbf{F}:(r_0,r_1)\rightarrow \mathbb{R}^{4(m+1)^2}$ are smooth functions with respect to the coefficients and right terms of \eqref{Leqs5}. Therefore, $\mathbf{Y}$ can be given by solving the following integral equations:
\begin{equation}\label{X}
\begin{cases}
\begin{aligned}
&\Pi_1(\mathbf{Y})(r)-\int_{r_0}^r\Pi_1\mathbb{S}(\mathbf{Y})(s)\mathrm{d}s
=\mathbf{F}_0+\int_{r_0}^r\Pi_1\mathbf{F}(s)\mathrm{d}s,\\
&\Pi_2(\mathbf{Y})(r)-\int_{r_1}^r\Pi_2\mathbb{S}(\mathbf{Y})(s)\mathrm{d}s
=\int_{r_1}^r\Pi_2\mathbf{F}(s)\mathrm{d}s.
\end{aligned}
\end{cases}
\end{equation}
 We define a linear operator $\mathfrak{E}:C^1([r_0,r_1];\mathbb{R}^{4(m+1)^2})\rightarrow C^1([r_0,r_1];\mathbb{R}^{4(m+1)^2})$ by
\begin{equation*}
\mathfrak{E}\mathbf{Y}(r):=\Pi_1\int_{r_0}^r\mathbb{S}(\mathbf{Y})(s)\mathrm{d}s
+\Pi_2\int_{r_1}^r\mathbb{S}(\mathbf{Y})(s)\mathrm{d}s.
\end{equation*}
 Note that $\mathbb{S}:(r_0,r_1)\rightarrow \mathbb{R}^{4(m+1)^2\times4(m+1)^2}$ is smooth. Then  there exists a constant
$C_0 > 0$ such that
\begin{equation*}
\|\mathfrak{E}\mathbf{Y}\|_{C^2([r_0, r_1])}\leq C_0\|\mathbf{Y}\|_{C^1([r_0, r_1])}, \ \ \rm{{for \ all}} \  \mathbf{Y}\in \mathfrak{E}.
\end{equation*}
Thus it follows  from the Arzel\'{a}-Ascoli theorem  that $\mathfrak{E}$ is compact.  Furthermore, one  can rewrite   \eqref{X} as
\begin{equation}\label{R}
(\mathbf{I}-\mathfrak{E})\mathbf{Y}(r)=\mathbf{F}_0+\Pi_1\int_{r_0}^r
\mathbf{F}(s)\mathrm{d}s
+\Pi_2\int_{r_1}^r\mathbf{F}(s)\mathrm{d}s.
\end{equation}
Suppose that $(\mathbf{I}-\mathfrak{E})\mathbf{Y}_*=0$. Then $ \mathbf{Y}_* $ solves \eqref{Leqs5}   with $(\mathbf{F},\mathbf{F}_0)=(\mathbf{0},\mathbf{0})$ and  the  corresponding $(\h\psi_m^\eta,\h\Psi_m^\eta)$ solves \eqref{Leqs4} with $\e F^\eta_3=\e F^\eta_4=F_5=0$. The estimate \eqref{H1} implies $\h\psi_m^\eta=\h\Psi_m^\eta=0$,  from which one gets $\mathbf{Y}_*=\mathbf{0}$ on $[r_0,r_1]$. Hence the  Fredholm alternative theorem  shows  that  \eqref{R} has a unique solution $\mathbf{Y}\in C^1([r_0,r_1];\mathbb{R}^{4(m+1)})$. Furthermore, the bootstrap argument for \eqref{X} gives the smoothness of the solution $\mathbf{Y}$ on $[r_0,r_1]$. Therefore, $\h\psi_m^\eta  $ and $\h\Psi_m^\eta$ are smooth functions and solve \eqref{Leqs4}. Then
the following higher order derivatives estimate can be derived similarly as in Lemma \ref{pro5}:
\begin{eqnarray}\label{H^4-1-1-a}
&&\Vert\h\psi_m^\eta\Vert_{H^4(\Omega)}
\leq C\left(\|\e F_3^\eta\|_{H^3(\m)}+\|\e F_4^\eta\|_{H^2(\m)}+\| F_5\|_{H^3(\mathbb{T}_{2\pi})}\right),\\\label{H^4-2-2-b}
&&\Vert\h\Psi_m^\eta\Vert_{H^4(\Omega)}
\leq C\left(\|\e F_3^\eta\|_{H^2(\m)}+\|\e F_4^\eta\|_{H^2(\m)}+\|F_5\|_{H^2(\mathbb{T}_{2\pi})}\right).
\end{eqnarray}
The above estimate constant $C>0$  is independent of $\eta$ and $m$.
\par For  any fixed $\eta>0$, let  $ (\h\psi_m^\eta,\h\Psi_m^\eta) $ given in the form \eqref{m} be the solution to the problem \eqref{Leqs4}. It follows from \eqref{H^4-1-1-a} and \eqref{H^4-2-2-b} that the sequence $\{(\h\psi_m^\eta,\h\Psi_m^\eta)\}_{m=1}^{\infty}$ is bounded in $[H^4(\Omega)]^2$. By  the weak compactness property of $H^4(\Omega)$,  there exists a subsequence $\{(\h\psi_{m_k}^\eta,\h\Psi_{m_k}^\eta)\}_{k=1}^{\infty}$ with  $ m_k\rightarrow \infty$ as $ k\rightarrow \infty$ converging weakly to  a limit $(\h\psi_*^\eta,\h\Psi_*^\eta)$  in $H^4(\Omega)$. Furthermore,  the Morrey inequality implies that the sequence $\{(\h\psi_m^\eta,\h\Psi_m^\eta)\}_{m\in\mathbb{N}}$ is bounded in $[C^{2,\frac{2}{3}}(\overline{\Omega})]^2$. Then it follows the Arzel$\mathrm{\grave{a}}$-Ascoli theorem that the subsequence $(\h\psi_{m_k}^\eta,\h\Psi_{m_k}^\eta)$   converges to $(\h\psi_*^\eta,\h\Psi_*^\eta) $  in $ [C^{2,\frac{1}{2}}(\overline{\Omega})]^2 $.
  This shows that $(\h\psi_*^\eta,\h\Psi_*^\eta)$ is a classical solution to \eqref{3-35}.
Then it follows from
 \eqref{H^4-1-1-a}-\eqref{H^4-2-2-b} and the $H^4$ convergence of $\{(\h\psi_{m_k}^\eta,\h\Psi_{m_k}^\eta)\}$ that one obtains
\begin{eqnarray}\label{*H^4_1}
&&\Vert\h\psi_*^\eta \Vert_{H^4(\Omega)}
\leq C\left(\Vert\e F_3^\eta\Vert_{H^3(\Omega)}+\Vert\e F_4^\eta\Vert_{H^2(\Omega)}
+\Vert F_5\Vert_{H^3(\mathbb{T}_{2\pi})}\right),\\
&&\Vert\h\Psi_*^\eta\Vert_{H^4(\Omega)}\label{*H^4_2}
\leq C\left(\Vert\e F_3^\eta\Vert_{H^2(\Omega)}+\Vert\e F_4^\eta\Vert_{H^2(\Omega)}
+\Vert F_5\Vert_{H^2(\mathbb{T}_{2\pi})}\right),
\end{eqnarray}
where the positive constant $ C$  is independent of $\eta$.
\par  For   any fixed $\eta>0$, let $(\h\psi_\ast^\eta,\h\Psi_\ast^\eta)\in[H^4(\Omega)]^2$ be the solution to  \eqref{3-35}. It follows from  \eqref{*H^4_1}-\eqref{*H^4_2} and the Morrey inequality that the sequence $\{(\h\psi_\ast^\eta,\h\Psi_\ast^\eta)\}_{n\in\mathbb{N}}$ is bounded in $[C^{2,\frac{2}{3}}(\overline{\Omega})]^2$. By adjusting the above limiting argument, we extract a subsequence $\{(\h\psi_\ast^{\eta_k},\h\Psi_\ast^{\eta_k})\}_{k=1}^{\infty}$ with $ \eta_k\rightarrow 0$ as $ k\rightarrow \infty$ such that its limit $(\h\psi_\ast,\h\Psi_\ast)\in[H^4(\Omega)\cap C^{2,\frac{1}{2}}(\overline{\Omega})]^2$ is a classical solution to  \eqref{3-10}  and satisfies
 \eqref{H^4-1}-\eqref{H^4-2}.
This completes the proof of Proposition \ref{pro6}.
\end{proof}
\par From Proposition \ref{pro6}, the well-posedness of \eqref{3-10-1} directly follows.
\begin{corollary}\label{cor1}
Fix  $\epsilon_0\in(0,R_0)$,   let  $\bar r_1\in(r_0,r_0+R_0-\epsilon_0]$ and $ \delta_1\in(0, \frac{\delta_0}2] $ be from Lemmas \ref{pro4} and  \ref{pro5}, respectively.    For  $r_1\in(r_0,\bar r_1)$ and $\delta\leq \delta_1$, let the iteration set $\mj_{\delta,r_1} $ be given by \eqref{3-9-s}. Then for each $(\e\psi,\e\Psi)\in \mj_{\delta,r_1} $, the  linear boundary value problem \eqref{3-10-1}    has a unique solution  $ (\psi,\Psi)\in \left(H^4(\m)\right)^2$ that satisfies
\begin{equation}\label{H^4-1-a-a}
\begin{aligned}
\Vert\psi\Vert_{H^4(\Omega)}
&\leq C\left(\|\e F_1\|_{H^3(\m)}+\|\e F_2\|_{H^2(\m)}+\|U_{1, en}-U_{1,0}\|_{C^3(\mathbb{T}_{2\pi})}\right.\\
&\qquad\left.+\|(U_{2, en},E_{ en},\Phi_{ ex})-(U_{2,0},E_0,\bar\Phi(r_1))\|_{C^4(\mathbb{T}_{2\pi})}\right),\\
\end{aligned}
\end{equation}
\begin{equation}\label{H^4-2-b-b}
\begin{aligned}
\Vert\Psi\Vert_{H^4(\Omega)}
&\leq C\left(\|\e F_1\|_{H^2(\m)}+\|\e F_2\|_{H^2(\m)}+\|U_{1, en}-U_{1,0}\|_{C^3(\mathbb{T}_{2\pi})}\right.\\
&\qquad\left.+\|(U_{2, en},E_{ en},\Phi_{ ex})-(U_{2,0},E_0,\bar\Phi(r_1))\|_{C^4(\mathbb{T}_{2\pi})}\right)
\end{aligned}
\end{equation}
 for some constant $C>0$ depending only on $(\gamma,b_0,   \rho_0, U_{1,0},U_{2,0},S_{0},E_{0},r_0,r_1,\epsilon_0)$.

\end{corollary}
\subsection{Proof of Theorem 1.3}\noindent
\par We are now ready to prove Theorem \ref{th1}.  Fix $\epsilon_0\in(0,R_0)$ and $ r_1\in (r_0,\bar r_1) $ for  $\bar r_1$ from Lemma \ref{pro4}.  Assume that
\begin{equation}\label{3-44}
 \delta\leq \delta_1,
 \end{equation}
 where $ \delta_1 $ is given by Lemma \ref{pro5}.
For any fixed  $ (\e\psi,\e\Psi)\in \mj_{\delta, r_1} $,  one can define a mapping
\begin{equation*}
\mt(\e\psi,\e\Psi)=(\psi,\Psi),
\end{equation*}
where $(\psi, \Psi)$ is the solution to the linear boundary value problem \eqref{3-10-1}. Then the estimates \eqref{H^4-1-a-a} and \eqref{H^4-2-b-b} in Corollary \ref{cor1} together with  \eqref{3-45-e} yields
  \begin{equation}\label{3-45-z-e-z}
 \Vert(\psi,\Psi)\Vert_{H^4(\Omega)}
\leq \mc_1^\flat \left(\delta^2+\omega_1(b,U_{ 1,en},U_{2, en},E_{ en},\Phi_{ ex})\right)
 \end{equation}
 for a constant $\mc_1^\flat>0 $  depending only on $(\gamma,b_0,   \rho_0, U_{1,0},U_{2,0},S_{0},E_{0},r_0,r_1,\epsilon_0)$.
 Set
 \begin{equation}\label{3-46}
  \sigma_1^\flat=\frac{1}{16((\mc_1^\flat)^2+\mc_1^\flat)}  \ \ {\rm{and}} \ \ \delta=4\mc_1^\flat\omega_1(b,U_{ 1,en},U_{2, en},E_{ en},\Phi_{ ex}).
\end{equation}
Then if \begin{equation}\label{3-47}
\omega_1(b,U_{ 1,en},U_{2, en},E_{ en},\Phi_{ ex})\leq \sigma_1^\flat,
\end{equation}
one can follow from \eqref{3-45-z-e-z} to obtain that
\begin{equation*}
\begin{aligned}
\Vert(\psi,\Psi)\Vert_{H^4(\Omega)}
\leq \mc_1^\flat\bigg(\delta^2+\omega_1(b,U_{ 1,en},U_{2, en},E_{ en},\Phi_{ ex})\bigg)
\leq \frac12\delta.
\end{aligned}
\end{equation*}
This implies that the mapping $\mt$ maps $\mj_{\delta, r_1}$ into itself.
\par It remains to show that the mapping $\mt$ is contractive in a low order norm. Let $ (\e\psi^{(i)},\e\Psi^{(i)})\in \mj_{\delta, r_1} $, $ i=1,2 $, one has $\mt(\e\psi^{(i)},\e\Psi^{(i)})=(\psi^{(i)},\Psi^{(i)}).$ Set
   \begin{equation*}
  (\e Y_1,\e Y_2) =(\e\psi^{(1)},\e\Psi^{(1)})-(\e\psi^{(2)},\e\Psi^{(2)}), \quad  ( Y_1, Y_2) =(\psi^{(1)},\Psi^{(1)})-(\psi^{(2)},\Psi^{(2)}).
  \end{equation*}
  Then it follows from \eqref{3-10-1} that $   ( Y_1, Y_2) $ satisfies
  \begin{equation}\label{3-48}
  \begin{cases}
  \p_r^2Y_1
  -A_{22}(r,\th,\e V_1^{(1)},\e V_2^{(1)},\e\Psi^{(1)})\p_\th^2Y_1 +2A_{12}(r,\th,\e V_1^{(1)},\e V_2^{(1)},\e\Psi^{(1)})
   \p_{r\th}^2Y_1 \\
  +\bar a_1(r)\p_rY_1
   +\bar a_2(r)\p_\th Y_1  +\bar b_1(r)\p_rY_2
    +\bar b_2(r)\p_\th Y_2+\bar b_3(r)Y_2\\
   =\mf_1\left(r,\th,\e V_1^{(1)},\e V_2^{(1)},\e\Psi^{(1)}, \n\e\Psi^{(1)},\e V_1^{(2)},\e V_2^{(2)},\e\Psi^{(2)}, \n\e\Psi^{(2)}\right),\ \ &\text{in}\ \ \Omega,\\
 \left(\p_r^2+\frac 1 r\p_r+\frac{1}{r^2}\p_{\th}^2\right)Y_2+ \bar a_3(r)\p_rY_1+\bar a_4(r)\p_\th Y_1-\bar b_4(r)Y_2\\
=\mf_2\left(r,\th,\e V_1^{(1)},\e V_2^{(1)}, \e\Psi^{(1)},\e V_1^{(2)},\e V_2^{(2)},\e\Psi^{(2)}\right),\ \ &\text{in}\ \ \Omega,\\
 Y_1(r_0,\th)=\p_rY_1(r_0,\th)=
 \p_rY_2(r_0,\th)=0, \ \ &\text{on}\ \ \Gamma_{en},\\
 Y_2(r_1,\th)=0,\ \ &\text{on}\ \ \Gamma_{ex},\\
\end{cases}
\end{equation}
where
 \begin{equation*}
 \begin{aligned}
 &\e V_1^{(i)}=\partial_r\e\psi^{(i)},\ \  \e V_2^{(i)}=\frac{\partial_\theta\e\psi^{(i)}-d_0}{r},\ i=1,2,\\
 &\mf_1\left(r,\th,\e V_1^{(1)},\e V_2^{(1)},\e\Psi^{(1)}, \n\e\Psi^{(1)},\e V_1^{(2)},\e V_2^{(2)},\e\Psi^{(2)}, \n\e\Psi^{(2)}\right)\\
 &=F_1(r,\th,\e V_1^{(1)},\e V_2^{(1)},\e\Psi^{(1)}, \n\e\Psi^{(1)})-F_1(r,\th,\e V_1^{(2)},\e V_2^{(2)},\e\Psi^{(2)}, \n\e\Psi^{(2)})\\
 &\quad+\bigg(A_{22}(r,\th,\e V_1^{(1)},\e V_2^{(1)},\e\Psi^{(1)})-A_{22}(r,\th,\e V_1^{(2)},\e V_2^{(2)},\e\Psi^{(2)})\bigg)\p_\th^2\psi^{(2)}\\
 &\quad-\bigg(2A_{12}(r,\th,\e V_1^{(1)},\e V_2^{(1)},\e\Psi^{(1)})-2A_{12}(r,\th,\e V_1^{(2)},\e V_2^{(2)},\e\Psi^{(2)})\bigg)\p_{r\th}^2\psi^{(2)},\\
 \end{aligned}
 \end{equation*}
 \begin{equation*}
 \begin{aligned}
 &\mf_2\left(r,\th,\e V_1^{(1)},\e V_2^{(1)},\e\Psi^{(1)}, \e V_1^{(2)},\e V_2^{(2)},\e\Psi^{(2)}\right)\\
 &=F_2(r,\th,\e V_1^{(1)},\e V_2^{(1)},\e\Psi^{(1)})-F_2(r,\th,\e V_1^{(2)},\e V_2^{(2)},\e\Psi^{(2)}).\\
 \end{aligned}
 \end{equation*}
 Since $(\e\psi^{(i)},\e\Psi^{(i)})$, $(\psi^{(i)},\Psi^{(i)}) \in \mj_{\delta, r_1}$ for $i= 1, 2$, the $H^1$ estimate in Lemma \ref{pro4} shows that
 \begin{equation}\label{3-48-1}
 \begin{aligned}
\|(Y_1,Y_2)\|_{H^1(\m)}
&\leq C\bigg(
 \| \mf_1(\cdot, \e V_1^{(1)},\e V_2^{(1)},\e\Psi^{(1)}, \n\e\Psi^{(1)},\e V_1^{(2)},\e V_2^{(2)},\e\Psi^{(2)}, \n\e\Psi^{(2)})\|_{L^2(\m)}\\
 &\qquad+\|\mf_2(\cdot,
 \e V_1^{(1)},\e V_2^{(1)},\e\Psi^{(1)}, \e V_1^{(2)},\e V_2^{(2)},\e\Psi^{(2)})\|_{L^2(\m)}\bigg)\\
 &\leq \mc_2^\flat\delta\|(\e Y_1,\e Y_2)\|_{H^1(\m)}
  \leq 4\mc_1^\flat\mc_2^\flat\omega_1(b,U_{ 1,en},U_{2, en},E_{ en},\Phi_{ ex}),
 \end{aligned}
\end{equation}
where  $\mc_2^\flat $ is a  positive constant depending only on $(\gamma,b_0,   \rho_0, U_{1,0},U_{2,0},S_{0},E_{0},r_0,r_1,\epsilon_0)$. Set
 \begin{equation}\label{3-49}
\sigma_2^\flat=\frac{1}{8\mc_1^\flat\mc_2^\flat}.
\end{equation}
Then if
$$ \omega_1(b,U_{ 1,en},U_{2, en},E_{ en},\Phi_{ ex})\leq \sigma_2^\flat,$$ $\mt$ is a contractive mapping in $H^1(\m)$ norm and there exists a unique $(\psi,\Psi) \in \mj_{\delta, r_1}$ such that $\mt (\psi,\Psi) = (\psi,\Psi)$.
\par Finally, we choose
\begin{equation*}
\sigma_1^\ast=\min\left\{\frac{\delta_1}{4\mc_1^\flat},\sigma_1^\flat,\sigma_2^\flat\right
\}.
\end{equation*}
Then if $ \omega_1(b,U_{ 1,en},U_{2, en},E_{ en},\Phi_{ ex})\leq \sigma_1^\ast$,  the nonlinear problem \eqref{3-7} has a unique solution $(V_1,V_2, \Psi)\in \left(H^3(\m)\right)^2\times H^4(\m)$  satisfying the estimate
\begin{equation}\label{3-49-1}
\Vert(V_1,V_2)\Vert_{H^3(\Omega)}+\Vert\Psi\Vert_{H^4(\Omega)}\leq \mc_3^\flat\omega_1(b,U_{ 1,en},U_{2, en},E_{ en},\Phi_{ ex}),
\end{equation}
where  $\mc_3^\flat>0 $ is a  constant  depending only on $(\gamma,b_0,   \rho_0, U_{1,0},U_{2,0},S_{0},E_{0},r_0,r_1,\epsilon_0)$. That is, the
background supersonic flow is structurally stable within  cylindrical   irrotational flows under perturbations
of  suitable boundary
conditions.
\section{The stability analysis within cylindrical  rotational flows}\label{rotation}\noindent
\par Now, we turn to the case of rotational flows and prove Theorem \ref{th2}. It is well-known that the steady Euler-Poisson system is hyperbolic-elliptic mixed in the supersonic region.
  The solvability of nonlinear boundary value problems for such mixed system is is extremely subtle and difficult. Some of the key difficulties lie in identifying suitable hyperbolic and elliptic modes with proper boundary conditions.   Here
  we  utilize the deformation-curl-Poisson decomposition developed in \cite{WS19}  to decompose the hyperbolic
and elliptic modes  in the system \eqref{1-2-c}.
  \par Using the Bernoulli's function,    the density  can be represented  as
\begin{equation}\label{5-3}
\rho=\mh(K,S,U_1,  U_2,\Phi)=
\bigg(\frac{\gamma-1}{\gamma e^S}\bigg(K+\Phi-\frac{1}{2}(U_1^2+U_2^2)\bigg)\bigg)
^{\frac{1}{\gamma-1}}.
\end{equation}
Substituting \eqref{5-3}   into the first equation in \eqref{1-2}, if a smooth flow does not contain the vacuum and the stagnation points, then the system \eqref{1-2} is equivalent to the following system:
\begin{equation}\label{5-5}
\begin{cases}
\begin{aligned}
&\bigg(c^2(K,U_1,  U_2,\Phi)-U_1^2\bigg)\p_r U_1+\bigg({c^2(K,U_1,  U_2,\Phi)-U_2^2}\bigg)\frac{\p_\theta U_2}{r}
+\frac {c^2(K,U_1,  U_2,\Phi)U_1}{r}\\
& -{U_1U_2}\bigg(\p_rU_2+\frac{\p_\theta U_1}{r}\bigg)+{\bigg(U_1\p_r+\frac{U_2}{r}\p_\theta\bigg)\Phi}
=0, \\
&\frac{U_2}{r}(\p_\th U_1-\p_r(rU_2))= \frac{e^S\mh^{\gamma-1}(K,S,U_1,  U_2,\Phi) }{\gamma-1}{\p_r S}-\p_r K,\\
&\bigg(\p_r^2+\frac 1 r\p_r+\frac{1}{r^2}\p_{\th}^2\bigg)\Phi=\mh(K,S,U_1,  U_2,\Phi)-b,\\
&\bigg(U_1\partial_r +\frac{U_2}{r}\partial_{\theta}\bigg) K=0,\\
&\bigg(U_1\partial_r +\frac{U_2}{r}\partial_{\theta}\bigg) S=0.\\
\end{aligned}
\end{cases}
\end{equation}
 Set
\begin{equation*}
\begin{aligned}
&(W_1,W_2,W_3)(r,\th)=(U_1,U_2,\Phi)(r,\th)-(\bar U_1,\bar U_2,\bar \Phi)(r), \ \  (r,\th)\in \Omega, \\
&(N_1,N_2)(r,\th)=(K,S)(r,\th)-( K_0,S_0), \qquad \qquad \qquad\ \ \ \ (r,\th)\in \Omega,\\
&{\bf W}=(W_1,W_2,W_3), \ \ \bar{\bf W}=(\bar U_1,\bar U_2,\bar \Phi), \ \ {\bf N}=(N_1,N_2) , \ \ {\bf N}_0=( K_0,S_0).\\
\end{aligned}
\end{equation*}
Then  one can use   ${\bf W}$ and $ {\bf N} $ to rewrite the system \eqref{5-5}  in  $ \Omega $  as follows
\begin{equation}\label{5-6}
  \begin{cases}
  \begin{aligned}
  &\p_rW_1
  -r B_{22}(r,\th,{\bf W},N_1)\p_\th W_2
   +r B_{12}(r,\th,{\bf W},N_1)\p_rW_2
   +B_{21}(r,\th,{\bf W},N_1)
  \p_{\th}W_1\\
 & +\bar a_1(r)W_1
  +r\e a_2(r)W_2
   +\bar b_1(r)\p_rW_3
    +\bar b_2(r)\p_\th W_3+\bar b_3(r)W_3
    =G_1(r,\th,{\bf W},{\bf N}),\\
    &\frac{1}{r}(\p_\th W_1-\p_r(rW_2))=G_2(r,\th,{\bf W},{\bf N}),\\
 &\left(\p_r^2+\frac 1 r\p_r+\frac{1}{r^2}\p_{\th}^2\right)W_3+ \bar a_3(r)W_1+r\bar a_4(r) W_2-\bar b_4(r)W_3
 =G_3(r,\th,{\bf W},{\bf N}),\\
 &\bigg((W_1+\x U_1)\partial_r +\frac{W_2+\x U_2}{r}\partial_{\theta}\bigg) N_1=0,\\
&\bigg((W_1+\x U_1)\partial_r +\frac{W_2+\x U_2}{r}\partial_{\theta}\bigg) N_2=0,
\end{aligned}
\end{cases}
\end{equation}
where
\begin{equation*}
  \begin{aligned}
 &B_{22}(r,\th,{\bf W},N_1)=\frac{(W_2+\bar U_2)^2-c^2(N_1+ K_0,{\bf W}+\bar{\bf W})}{r^2\left(c^2(N_1+ K_0,{\bf W}+\bar{\bf W})-(W_1+\bar U_1)^2\right)},\\
 & B_{12}(r,\th, {\bf W},N_1)= B_{21}(r,\th, {\bf W},N_1)=-\frac{(W_1+\bar U_1) (W_2+\bar U_2)}{r\left(c^2(N_1+ K_0,{\bf W}+\bar{\bf W})-(W_1+\bar U_1)^2\right)},\\
 & \bar a_1(r)=\frac1{\bar c^2-{\bar U_1^2}}\bigg(-(\gamma+1)\bar U_{1}\bar U_{1}'+\frac{\bar c^2}r-\bar U_{2}\bar U_{2}'-(\gamma-1)\frac{\bar U_{1}^2}r +\bar E\bigg)\\
 &\qquad=\frac1{\bar c^2-{\bar U_1^2}}\bigg(\frac{(\gamma+1)(1+\bar M_{2}^2)}{r(1-\bar M_{1}^2)}\bar U_{1}^2 +\frac{(\gamma+1)\bar E}{(1-\bar M_{1}^2)\bar c^2}\bar U_{1}^2+\frac{\bar c^2+\bar U_{2}^2}r -\frac{(\gamma-1)\bar U_{1}^2}r
+\bar E\bigg), \\
  \end{aligned}
\end{equation*}
\begin{equation*}
\begin{aligned}
 & \e a_2(r)=\frac1{\bar c^2-{\bar U_1^2}}\bigg(-(\gamma-1)\frac{\bar U_{2}\bar U_{1}'}{r}-\frac{{\bar U_{1}\bar U_{2}'}}{r}-(\gamma-1)\frac{\bar U_{1}\bar U_{2}}{r^2} \bigg)\\
  &\qquad =\frac1{\bar c^2-{\bar U_1^2}}\bigg(\frac{1-\bar M_{1}^2+(\gamma-1)(\bar M_{1}^2+\bar M_{2}^2)}{1-\bar M_{1}^2}+ \frac{(\gamma-1) r\bar E}{(1-\bar M_{1}^2)\bar c^2}\bigg)\frac{\bar U_{1}\bar U_{2}}{r^2},\\
 &
 \bar b_1(r)=\frac1{\bar c^2-{\bar U_1^2}}\bar U_{1}, \ \
     \bar b_2(r)=\frac1{\bar c^2-{\bar U_1^2}}\frac{\bar U_{2}}{r}, \ \ \bar b_3(r)=\frac1{\bar c^2-{\bar U_1^2}}\bigg((\gamma-1)\bar U_1'+(\gamma-1)\frac{\bar U_1}{r}\bigg),\\
      &\bar a_3(r)=\frac{\bar\rho \bar U_{1}}{\bar c^2}, \quad \quad  \bar a_4(r)=\frac{\bar\rho \bar U_{2}}{r\bar c^2}, \quad \quad   \bar  b_4(r)=\frac{\bar\rho}{ \bar c^2},\\
&G_1(r,\th,{\bf W},{\bf N})=\frac1{c^2(N_1+ K_0,{\bf W}+\bar{\bf W})-(\bar U_1+W_1)^2}\bigg(
    \bar U_1'\bigg(-(\gamma-1)N_1+\frac{\gamma+1}{2} W_1^2\\
    &\quad +\frac{\gamma-1}{2} W_2^2\bigg)+\bar U_{2}' W_1 W_2
  + \frac{\bar U_1}{r}\bigg(-(\gamma-1)N_1+\frac{\gamma-1}{2} (W_1^2+W_2^2)\\
  &\quad   + \frac{W_1}{r}\bigg((\gamma-1)(-N_1-W_3+\bar U_{1}W_1+ \bar U_{2} W_2)
   + \frac{\gamma-1}{2} (W_1^2+W_2^2)\bigg)  -W_1\p_rW_3-\frac{1}rW_2\p_\th W_3\bigg) \\
   & \quad -\bigg(\frac1{c^2(N_1+ K_0,{\bf W}+\bar{\bf W})-(W_1+\bar U_1)^2}-\frac1{\bar c^2-\bar U_1^2}\bigg)\bigg(\bigg(-(\gamma+1)\bar U_{1}\bar U_{1}'+\frac{\bar c^2}r\\
    &\quad -\bar U_{2}\bar U_{2}'-(\gamma-1)\frac{\bar U_{1}^2}r +\bar E\bigg)W_1+\bigg(-(\gamma-1)\bar U_{2}\bar U_{1}'-{\bar U_{1}\bar U_{2}'}-(\gamma-1)\frac{\bar U_{1}\bar U_{2}}{r} \bigg)W_2  \\
  &\quad+\bar U_{1}\p_r W_3
    +\frac{\bar U_{2}}{r}\p_\th W_3+\bigg((\gamma-1)\bar U_1'+(\gamma-1)\frac{\bar U_1}{r}\bigg)W_3\bigg), \\
    &G_2(r,\th,{\bf W},{\bf N})=\frac{1}{W_2+\bar U_2}\bigg(\frac{e^{(N_2+S_0)}\mh^{\gamma-1}({\bf N}+{\bf N}_0,{\bf W}+\bar{\bf W}) }{\gamma-1}{\p_r N_2}-\p_r N_1\bigg),\\
&G_3(r,\th,{\bf W},{\bf N})=\mh({\bf N}+{\bf N}_0,{\bf W}+\bar{\bf W})-\mh({\bf N}_0,\bar{\bf W}) +\bar a_3 W_1+ r\bar a_4W_2-\bar b_4W_3-(b-b_0).\\
  \end{aligned}
\end{equation*}
 Obviously, $(\x a_1,\x a_3,\x a_4,\bar b_1,\bar b_2,\bar b_3,\bar b_4)(r)$ are same as the one defined in Section \ref{irrotational}  and $\tilde{a}_2(r)= a_2(r)-(\bar U_{1}\bar U_{2})(r)$. Furthermore, the system \eqref{5-6} is supplemented with the following boundary
conditions:
\begin{equation}\label{5-7}
\begin{cases}
(W_1, W_2,\p_r W_3)(r_0,\th)=(U_{1,en}, U_{2,en},E_{en})(\th)-(U_{1,0},U_{2,0},E_0),\ \ &{\rm{on}}\ \ \Gamma_{en},\\
( N_1,N_2)(r_0,\th)=(  K_{en},S_{en})(\th)-(K_0,S_0),\ \ &{\rm{on}}\ \ \Gamma_{en},\\
W_3(r_1,\th)=\Phi_{ex}(\th)-\bar\Phi(r_1), \ \ &{\rm{on}}\ \ \Gamma_{ex}.\\
\end{cases}
\end{equation}
\par Note that $G_2({\bf W},{\bf N})$ only belongs to $H^2(\m)$ if one looks for the solution $({\bf W},{\bf N})$ in $\left(H^3(\m)\right)^2\times H^4(\m)\times \left(H^3(\m)\right)^2$.   The first two equations in \eqref{5-6} can be regarded as a first order system for $ (W_1, W_2)$,  the energy estimates obtained in the previous section for irrotational flows
indicate that the regularity of the solutions $W_1$ and $ W_2$ would be at best same as the source terms on the right hand sides in general. Hence  it seems that only $H^2(\m)$ regularity for $(W_1,W_2)$ is possible and there appears a loss of derivatives. To overcome this difficulty,  we use the continuity equation to introduce the
stream function which has the advantage of one order higher regularity than the velocity field. The
Bernoulli's quantity and the entropy can be represented as functions of the stream function. This
will enable us to overcome the possibility of losing derivatives. To achieve this, we will choose some
appropriate function spaces and design an elaborate two-layer iteration scheme to prove Theorem  \ref{th2}.
\subsection{The linearized problem }\noindent
\par  To solve the nonlinear boundary value problem \eqref{5-6} with \eqref{5-7} by the method of iteration, we define the following the iteration sets:
\begin{equation}\label{5-8}
\begin{aligned}
  \ma_{\delta_{e}, r_1}&=\bigg\{{\bf N}=(N_1,N_2)\in
  \left(H^4(\m)\right)^2: \|(N_1,N_2)\|_{H^4(\m)}\leq \delta_e\},\\
   \end{aligned}
  \end{equation}
  and
  \begin{equation}
  \label{5-9}
  \begin{aligned}
\ma_{\delta_v, r_1}&=\bigg\{{\bf W}=(W_1,W_2,W_3)\in \left(H^3(\m)\right)^2\times H^4(\m):\|(W_1,W_2)\|_{H^3(\m)}+\|W_3\|_{H^4(\m)}\leq \delta_v\bigg\},
\end{aligned}
\end{equation}
where the positive constants $\delta_e$, $\delta_v$ and $r_1$ will be determined later.
\par Next, for   fixed  $ \e {\bf N}\in \ma_{\delta_e, r_1} $ and for any function $\e{\bf W}\in \ma_{\delta_v, r_1} $, we first construct an
operator $ \mathfrak{F}_1^{\e {\bf N}}:\e{\bf W}\in \ma_{\delta_v, r_1}\mapsto{\bf W}\in \ma_{\delta_v, r_1}$  by resolving the following boundary value problem:
\begin{equation}\label{5-10}
\begin{cases}
\p_rW_1
  -r B_{22}(r,\th,\e{\bf W},\e N_1)\p_\th W_2
   +r B_{12}(r,\th,\e{\bf W},\e N_1)\p_rW_2
   +B_{21}(r,\th,\e{\bf W},N_1)
  \p_{\th}W_1\\
  +\bar a_1(r)W_1
  +r\e a_2(r)W_2
   +\bar b_1(r)\p_rW_3
    +\bar b_2(r)\p_\th W_3+\bar b_3(r)W_3
    =G_1(r,\th,\e{\bf W},\e{\bf N}), \ \ &{\rm{in}} \ \ \m,\\
\frac{1}{r}(\p_\th W_1-\p_r(rW_2))=G_2(r,\th,\e{\bf W},\e{\bf N}), \ \ &{\rm{in}} \ \ \m,\\
\left(\p_r^2+\frac 1 r\p_r+\frac{1}{r^2}\p_{\th}^2\right)W_3+ \bar a_3(r)W_1+r\bar a_4(r) W_2-\bar b_4(r)W_3
 =G_3(r,\th,\e{\bf W},\e{\bf N}), \ \ &{\rm{in}} \ \ \m,\\
(W_1, W_2,\p_r W_3)(r_0,\th)=(U_{1,en}, U_{2,en},E_{en})(\th)-(U_{1,0},U_{2,0},E_0),\ \ &{\rm{on}}\ \ \Gamma_{en},\\
W_3(r_1,\th)=\Phi_{ex}(\th)-\bar\Phi(r_1), \ \ &{\rm{on}}\ \ \Gamma_{ex}.
\end{cases}
\end{equation}
Then it is easy to verify  that
\begin{equation}\label{5-11}
\begin{cases}
\begin{aligned}
&\|G_1(\cdot,\e{\bf W},\e{\bf N})\|_{H^3(\m)}\leq C(\delta_e+\delta_v^2),\\
&\|G_2(\cdot,\e{\bf W},\e{\bf N})\|_{H^3(\m)}\leq C\delta_e,\\
&\|G_3(\cdot,\e{\bf W},\e{\bf N})\|_{H^2(\m)}\leq C(\delta_e+\delta_v^2),\\
\end{aligned}
\end{cases}
\end{equation}
where the constant $ C>0 $ depends only on $(\gamma,b_0, \rho_0,U_{1,0},U_{2,0}, S_{0},E_{0})$.
\par
In the following, we consider  the following problem:
\begin{align}\label{5-12}
\begin{cases}
(\p_r^2+\frac{1}{r}\p_r+ \frac{1}{r^2}\p_\th^2)\phi_1= G_2(r,\th,\e{\bf W},\e{\bf N}),\ \ &{\rm{in}} \ \ \m,\\
\phi_1(r_0,\theta)=0,\ \ &{\rm{on}}\ \ \Gamma_{en},\\
\phi_1(r_1,\theta)=0, \ \ &{\rm{on}}\ \ \Gamma_{ex}.\\
\end{cases}
\end{align}
The standard  elliptic theory in \cite{GT98} shows that  $ \phi_1\in H^{5}(\m) $ satisfies the estimate
\begin{equation}\label{5-13}
\|\phi_1\|_{H^5(\m)} \leq C\|G_2(\cdot,\e{\bf W},\e{\bf N})\|_{H^3(\m)}.
 \end{equation}
 Define
 \begin{equation*}
  \h W_1= W_1-\frac{1}{r}\p_{\theta}\phi_1, \ \  \h W_2= W_2+\p_{r}\phi_1, \  \ {\rm{in}} \ \ \m.\\
  \end{equation*}
  Then
  \eqref{5-10} becomes
  \begin{equation}\label{5-15}
\begin{cases}
\p_r\h W_1
  -r B_{22}(r,\th,\e{\bf W},\e N_1)\p_\th \h W_2
   +r B_{12}(r,\th,\e{\bf W},\e N_1)\p_r\h W_2
   +B_{21}(r,\th,\e{\bf W},\e N_1)
  \p_{\th}\h W_1\\
  +\bar a_1(r)\h W_1
  +r\e a_2(r)\h W_2
   +\bar b_1(r)\p_rW_3
    +\bar b_2(r)\p_\th W_3+\bar b_3(r)W_3
    =G_4(r,\th,\e{\bf W},\e{\bf N}), \ \ &{\rm{in}} \ \ \m,\\
\frac{1}{r}(\p_\th W_1-\p_r(rW_2))=0, \ \ &{\rm{in}} \ \ \m,\\
\left(\p_r^2+\frac 1 r\p_r+\frac{1}{r^2}\p_{\th}^2\right)W_3+ \bar a_3(r)\h W_1+r\bar a_4(r) \h W_2-\bar b_4(r)W_3
 =G_5(r,\th,\e{\bf W},\e{\bf N}), \ \ &{\rm{in}} \ \ \m,\\
 (\h W_1, \h W_2)(r_0,\th)=\bigg(U_{1,en}-\frac{1}{r}\p_{\theta}\phi_1, U_{2,en}+\p_{r}\phi_1\bigg)(r_0,\th)-(U_{1,0},U_{2,0}),\ \ &{\rm{on}}\ \ \Gamma_{en},\\
 \p_r W_3(r_0,\th)=E_{en}(\th)-E_0,\ \ &{\rm{on}}\ \ \Gamma_{en},\\
W_3(r_1,\th)=\Phi_{ex}(\th)-\bar\Phi(r_1), \ \ &{\rm{on}}\ \ \Gamma_{ex},
\end{cases}
\end{equation}
where
 \begin{equation*}
\begin{aligned}
&G_4(r,\th,\e{\bf W},\e{\bf N})=G_1(r,\th,\e{\bf W},\e{\bf N})-\p_r\bigg(\frac{1}{r}\p_{\theta}\phi_1\bigg)-r B_{22}(r,\th,\e{\bf W},\e N_1)\p_{r\th}^2\phi_1+r B_{12}(r,\th,\e{\bf W},\e N_1)\p_r^2\phi_1\\
&\quad-B_{21}(r,\th,\e{\bf W},\e N_1)
  \frac{\p_{\th}^2\phi_1}{r}-\bar a_1(r)\frac{\p_{\theta}\phi_1}{r}
  +r\e a_2(r)\p_{r}\phi_1,\\
  &G_5(r,\th,\e{\bf W},\e{\bf N})=G_3(r,\th,\e{\bf W},\e{\bf N})-\bar a_3(r)\p_r\bigg(\frac{1}{r}\p_{\theta}\phi_1\bigg)+r\bar a_4(r)\p_{r\th}^2\phi_1.\\
  \end{aligned}
\end{equation*}
Introduce the potential function
\begin{equation*}
\phi_2(r,\theta)= \int_{r_0}^r \h  W_1(s,\theta)\de s + \int_0^{\theta}\bigg( r_0 \h  W_2(r_0,s) +\tilde{d}_0\bigg)\de s, \ \  (r,\th)\in \Omega,
\end{equation*}
where $$ \e d_0=- \frac{r_0}{2\pi}\int_0^{2\pi}\bigg(U_{2,en}(s)+\p_{r}\phi_1(r_0,s)-U_{2,0}\bigg)\de s $$ is introduced to guarantee that $\phi_2(r,\theta)=\phi_2(r,\theta+2\pi)$. Then $\phi_2$ is periodic in $\theta$ with period $2\pi$ and satisfies
\begin{align}\label{5-16}
\partial_r\phi_2=\h W_1,\ \ \partial_\theta\phi_2=r \h  W_2+\e d_0, \ \ \phi_2(r_0,0)=0, \ \ \text{in}\ \ \Omega.
\end{align}
Substituting \eqref{5-16} into \eqref{5-15} yields  the following boundary value problem:
\begin{equation}\label{5-17}
\begin{cases}
\p_r^2\phi_2
  - B_{22}(r,\th,\e{\bf W},\e N_1)\p_\th^2 \phi_2
   +2B_{12}(r,\th,\e{\bf W},\e N_1)\p_{r\th }^2\phi_2-\frac{B_{12}(r,\th,\e{\bf W},\e N_1)}{r}\p_\th \phi_2\\
   +\bar a_1(r)\partial_r\phi_2
  +\e a_2(r)\partial_\th\phi_2
   +\bar b_1(r)\p_rW_3
    +\bar b_2(r)\p_\th W_3+\bar b_3(r)W_3\\
    =G_4(r,\th,\e{\bf W},\e{\bf N})-B_{12}(r,\th,\e{\bf W},\e N_1)\frac{\e d_0}{r}+\e a_2(r)\e d_0, \ \ &{\rm{in}} \ \ \m,\\
\left(\p_r^2+\frac 1 r\p_r+\frac{1}{r^2}\p_{\th}^2\right)W_3+ \bar a_3(r)\p_r\phi_2+\bar a_4(r)\p_\th \phi_2-\bar b_4(r)W_3\\
 =G_5(r,\th,\e{\bf W},\e{\bf N})+\bar a_4(r)\e d_0, \ \ &{\rm{in}} \ \ \m,\\
 \partial_r\phi_2(r_0,\th)=U_{1,en}(\th)-\frac{1}{r_0}\p_{\theta}\phi_1(r_0,\th)
 -U_{1,0}, \ \p_r W_3(r_0,\th)= E_{en}(\th)-E_0,\ \ &{\rm{on}}\ \ \Gamma_{en},\\
 \partial_\th\phi_2(r_0,\th)=r_0U_{2,en}(\th)+r_0\p_{r}\phi_1(r_0,\th)+\e d_0-U_{2,0},\ \ &{\rm{on}}\ \ \Gamma_{en},\\
W_3(r_1,\th)=\Phi_{ex}(\th)-\bar\Phi(r_1), \ \ &{\rm{on}}\ \ \Gamma_{ex}.
\end{cases}
\end{equation}

\par Set
\begin{equation*}
\begin{aligned}
&\psi(r,\th)=\phi_2(r,\th)-\int_{0}^{\th} \bigg(r_0U_{2,en}(s)+r_0\p_{r}\phi_1(r_0,s)+\e d_0-U_{2,0}\bigg)\de s, \ \  &(r,\th)\in \Omega, \\
& \Psi(r,\th)=W_3(r,\th)-(r-r_1)(E_{en}(\th)-E_0)-(\Phi_{ex}(\th)-\bar\Phi(r_1)), \ \  &(r,\th)\in \Omega.
\end{aligned}
\end{equation*}
Then \eqref{5-17} can be rewritten as
\begin{equation}\label{5-18}
\begin{cases}
\p_r^2\psi
  - B_{22}(r,\th,\e{\bf W},\e N_1)\p_\th^2 \psi
   +2B_{12}(r,\th,\e{\bf W},\e N_1)\p_{r\th }^2\psi-\frac{B_{12}(r,\th,\e{\bf W},\e N_1)}{r}\p_\th \psi\\
   +\bar a_1(r)\partial_r\psi
  +\e a_2(r)\partial_\th\psi
   +\bar b_1(r)\p_r\Psi
    +\bar b_2(r)\p_\th \Psi+\bar b_3(r)\Psi
    =G_6(r,\th,\e{\bf W},\e{\bf N}), \ \ &{\rm{in}} \ \ \m,\\
\left(\p_r^2+\frac 1 r\p_r+\frac{1}{r^2}\p_{\th}^2\right)\Psi+ \bar a_3(r)\p_r\psi+\bar a_4(r)\p_\th \psi-\bar b_4(r)\Psi
 =G_7(r,\th,\e{\bf W},\e{\bf N}), \ \ &{\rm{in}} \ \ \m,\\
 \partial_r\psi(r_0,\th)=G_8(\th), \ \ \psi(r_0,\th)= \Psi(r_0,\th)=0,\ \ &{\rm{on}}\ \ \Gamma_{en},\\
 \Psi(r_1,\th)=0, \ \ &{\rm{on}}\ \ \Gamma_{ex},
\end{cases}
\end{equation}
where
\begin{equation*}
\begin{aligned}
&G_6(r,\th,\e{\bf W},\e{\bf N})=G_4(r,\th,\e{\bf W},\e{\bf N})-B_{12}(r,\th,\e{\bf W},\e N_1)\frac{\e d_0}{r}
 +\e a_2(r)\e d_0
+\bigg(\frac{B_{12}(r,\th,\e{\bf W},\e N_1)}{r}\\
&\quad-\e a_2(r)\bigg)
\bigg(r_0(U_{2,en}+\p_{r}\psi_1)(r_0,\th)+\e d_0-U_{2,0}\bigg)-\bar b_1(r)(E_{en}(\th)-E_0)\\
&\quad-\bar b_2(r)\bigg((r-r_1)(E_{en}'(\th)+\Psi_{ex}'(\th)\bigg)
 -\x b_3(r)\bigg((r-r_1)(E_{en}(\th)-E_0)+(\Psi_{ex}(\th)-\bar\Psi(r_1))\bigg),\\
 &G_7(r,\th,\e{\bf W},\e{\bf N})=G_5(r,\th,\e{\bf W},\e{\bf N})+\bar a_4(r)\e d_0
-\bigg(\frac1r (E_{en}(\th)-E_0 ) +\frac{1}{r^2}((r-r_1)E_{en}''(\th)+\Psi_{ex}''(\th))\bigg)\\
&\quad -\x a_4(r)\bigg(r_0(U_{2,en}(\th)-U_{2,0})+d_0\bigg)
 +\x b_4(r)\bigg((r-r_1)(E_{en}(\th)-E_0)+(\Psi_{ex}(\th)-\bar\Psi(r_1))\bigg),  \\
  &G_8(\th)=U_{1,en}(\th)-\frac{1}{r_0}\p_{\theta}\phi_1(r_0,\th)-U_{1,0}.
\end{aligned}
\end{equation*}
It   can be directly checked that there exists a constant $C>0$ depending only on $(\gamma,b_0, \rho_0, U_{1,0},U_{2,0},$\\$S_{0},E_{0})$ such that
 \begin{equation}\label{5-19}
 \begin{aligned}
 &\|  G_6(\cdot,
 \e{\bf W},\e{\bf N})\|_{H^3(\m)}+\|  G_7(\cdot,
 \e{\bf W},\e{\bf N})\|_{H^2(\m)}+\|G_8\|_{H^3(\mathbb{T}_{2\pi})}\\
 &\leq
 C\left(\|  G_4(\cdot,
 \e{\bf W},\e{\bf N})\|_{H^3(\m)}+\|  G_5(\cdot,
 \e{\bf W},\e{\bf N})\|_{H^2(\m)}+\|G_8\|_{H^3(\mathbb{T}_{2\pi})}+\sigma_p\right)\\
 &\leq
 C\left(\|  G_1(\cdot,
 \e{\bf W},\e{\bf N})\|_{H^3(\m)}+\|  G_2(\cdot,
 \e{\bf W},\e{\bf N})\|_{H^3(\m)}+\|  G_3(\cdot,
 \e{\bf W},\e{\bf N})\|_{H^3(\m)}+\|G_8\|_{H^3(\mathbb{T}_{2\pi})}+\sigma_p\right).\\
 \end{aligned}
 \end{equation}

 \par For  fixed  $ \e {\bf N}\in \ma_{\delta_e, r_1} $ and  $\e{\bf W}\in \ma_{\delta_v, r_1} $, we denote
  \begin{equation*}
  \begin{aligned}
&\e B_{i2}(r,\th):=B_{i2}(r,\th,\e{\bf W},\e N_1),\quad \ \  (r,\th)\in \Omega,\ i=1,2,\\
&\e G_j(r,\th):=G_j (r,\th,\e{\bf W},\e N_1),\quad\ \ \    (r,\th)\in \Omega, \ j=6,7. \\
\end{aligned}
 \end{equation*}
  Then we have the following proposition.
  \begin{proposition}\label{pro7}
  For each $\epsilon_0\in(0, R_0)$ with $R_0$  given in  Proposition \ref{pro1},
there exists a positive constant $\delta_2$ depending only on  $(\gamma,b_0,   \rho_0, U_{1,0},U_{2,0},S_{0},E_{0},\epsilon_0)$ such that   for $r_1\in(r_0,r_0+R_0-\epsilon_0]$ and $\max\{\delta_e,\delta_v\}\leq \delta_2$, the coefficients $ ( \e B_{12},\e B_{22}) $ for $ (\e {\bf N},\e{\bf W})\in \ma_{\delta_e, r_1}\times \ma_{\delta_v, r_1}$  satisfying
\begin{eqnarray}\label{5-7-f}
 \|( \e B_{12},\e B_{22})-( \bar A_{12},\bar A_{22})\|_{H^3(\m)}\leq C(\delta_e+\delta_v),
\end{eqnarray}
 and
 \begin{equation}\label{5-7-e}
\mu_1\leq \e B_{22}(r,\th)\leq \frac{1}{\mu_1},
\   \ (r,\th)\in\overline{\m}.
\end{equation}
Here $(\bar A_{12},\bar A_{22}) $ is defined in Proposition \ref{pro7} and the positive constants  $C$ and $ \mu_1\in(0,1) $ depend only on $(\gamma,b_0, \rho_0, U_{1,0},U_{2,0},S_{0},E_{0},\epsilon_0)$.
\end{proposition}
 Next, by adapting  the same ideas in  Proposition \ref{pro6},  one can show the
 the following well-posedness of the linearized  problem \eqref{5-18}.
 \begin{proposition}\label{pro8}
 Fix  $\epsilon_0\in(0,R_0)$,  there exist a constant $\bar r_1\in(r_0,r_0+R_0-\epsilon_0]$ depending only $(\gamma,b_0,   \rho_0, U_{1,0},U_{2,0},S_{0},E_{0},\epsilon_0)$  and a sufficiently small constant  $ \delta_3\in(0,\frac{\delta_2}2]$ such that for  $r_1\in(r_0,\bar r_1)$ and $\max\{\delta_e,\delta_v\}\leq \delta_3$,  the  linear boundary value problem \eqref{5-18}  associated with $ (\e {\bf N},\e{\bf W})\in \ma_{\delta_e, r_1}\times \ma_{\delta_v, r_1}$ has a unique solution  $ (\psi,\Psi)\in \left(H^4(\m)\right)^2$ that satisfies
\begin{eqnarray}\label{H^5-1}
&&\Vert\psi\Vert_{H^4(\Omega)}
\leq C\left(\|\e G_6\|_{H^3(\m)}+\|\e G_7\|_{H^2(\m)}+\|G_8\|_{H^3(\mathbb{T}_{2\pi})}\right),\\\label{H^5-2}
&&\Vert\Psi\Vert_{H^4(\Omega)}
\leq C\left(\|\e G_6\|_{H^2(\m)}+\|\e G_7\|_{H^2(\m)}+\|G_8\|_{H^2(\mathbb{T}_{2\pi})}\right)
\end{eqnarray}
for  some constant $C>0$ depending only on $(\gamma,b_0,   \rho_0, U_{1,0},U_{2,0},S_{0},E_{0},r_0,r_1,\epsilon_0)$.

\end{proposition}
The well-posedness of \eqref{5-10} directly follows from Proposition \ref{pro8}.

\begin{corollary}\label{cor2}
Fix  $\epsilon_0\in(0,R_0)$,   let  $\bar r_1\in(r_0,r_0+R_0-\epsilon_0]$ and $ \delta_3\in(0, \frac{\delta_2}2] $ be from Proposition  \ref{pro8}.    For  $r_1\in(r_0,\bar r_1)$ and $\max\{\delta_e,\delta_v\}\leq \delta_3$, let the iteration sets $  \ma_{\delta_e, r_1}$ and $ \ma_{\delta_v, r_1}$  be given by \eqref{5-8} and \eqref{5-9}, respectively. Then for each $ (\e {\bf N},\e{\bf W})\in \ma_{\delta_e, r_1}\times \ma_{\delta_v, r_1}$, the  linear boundary value problem \eqref{5-10}    has a unique solution  $ (W_1,W_2,W_3)\in \left(H^3(\m)\right)^2\times H^4(\m)$ satisfying the estimate
\begin{eqnarray}\label{H^5-1-a-a}
\begin{aligned}
&\Vert(W_1,W_2)\Vert_{H^3(\Omega)}
+\Vert W_3\Vert_{H^4(\Omega)}\\
&\leq C\left(\|G_1(\cdot,\e{\bf W},\e{\bf N})\|_{H^3(\m)}+
\|G_2(\cdot,\e{\bf W},\e{\bf N})\|_{H^3(\m)}+\|G_3(\cdot,\e{\bf W},\e{\bf N})\|_{H^2(\m)}+\sigma_p\right),\\
\end{aligned}
\end{eqnarray}
where $C>0$  is a  constant depending only on $(\gamma,b_0, \rho_0,U_{1,0},U_{2,0}, S_{0},E_{0},r_0,r_1,\epsilon_0)$.

\end{corollary}
\subsection{Proof of Theorem 1.5}\noindent
\par In this subsection, we establish   the existence and uniqueness of smooth supersonic spiral  flows with non-zero vorticity. The proof is divided into three steps.
\par { \bf Step 1. The  existence and uniqueness of a fixed point of $\mathfrak{F}_1^{\e {\bf N}}$.}
\par Fix $\epsilon_0\in(0,R_0)$ and  $ r_1\in (r_0,\bar r_1) $ for  $\bar r_1$ from Proposition  \ref{pro8}.  Assume that
\begin{equation}\label{5-44}
 \max\{\delta_e,\delta_v\}\leq \delta_3,
 \end{equation}
 where $ \delta_3 $ is given by  Proposition  \ref{pro8}. For any fixed $ \e {\bf N}\in \ma_{\delta_e, r_1}$, one can define a mapping
\begin{equation*}
\mathfrak{F}_1^{\e {\bf N}}(\e{\bf W})={\bf W}, \quad{\rm{ for \ each}} \ \ \e{\bf W}\in \ma_{\delta_v, r_1}.
\end{equation*}
 Then it follows from    \eqref{H^5-1-a-a} and \eqref{5-11}  that  $ (W_1,W_2,W_3)\in \left(H^3(\m)\right)^2\times H^4(\m)$ satisfying the estimate
\begin{equation}\label{6-1}
\begin{aligned}
\Vert(W_1,W_2)\Vert_{H^3(\m)}+\Vert W_3\Vert_{H^4(\m)}\leq
\mc_1^\sharp\left(\delta_e+(\delta_v)^2+\sigma_p\right),
\end{aligned}
\end{equation}
where  $\mc_1^\sharp$ is a  positive  constant depending only on $(\gamma,b_0, \rho_0,U_{1,0}, U_{2,0}, S_{0},E_{0},r_0,r_1,\epsilon_0)$.
\par Set
\begin{equation}\label{6-3}
 \delta_v=4\mc_1^\sharp(\delta_e+\sigma_p).
\end{equation}
Then if
\begin{equation}\label{6-4}
\delta_e+\sigma_p\leq \frac{1}{16(\mc_1^\sharp)^2},
\end{equation}
one can follow from \eqref{6-1} to obtain that
\begin{equation*}
\begin{aligned}
\Vert(W_1,W_2)\Vert_{H^3(\m)}+\Vert W_3\Vert_{H^4(\m)}\leq
\mc_1^\sharp\left(\delta_e+(\delta_v)^2+\sigma_p\right)\leq \frac12\delta_v.
\end{aligned}
\end{equation*}
This implies that the mapping $\mathfrak{F}_1^{\e {\bf N}}$ maps $\ma_{\delta_v, r_1}$ into itself.
 Next, for any fixed $ \e {\bf N}\in \ma_{\delta_e, r_1}$, we will show that $\mathfrak{F}_1^{\e {\bf N}}$  is a
a contraction mapping in a low order norm
\begin{equation*}
\|{\bf W}\|_\Sigma:=\Vert(W_1,W_2)\Vert_{L^2(\m)}+\Vert W_3\Vert_{H^1(\m)}
\end{equation*}
so that one can find a unique fixed point by the contraction mapping theorem. Let $ \e{\bf W}^{(i)}\in \ma_{\delta_v, r_1} $, $ i=1,2 $, one has $\mathfrak{F}_1^{\e {\bf N}}(\e{\bf W}^{(i)})={\bf W}^{(i)}.$ Set
   \begin{equation*}
  \e {\bf Z}=\e{\bf W}^{(1)}-\e{\bf W}^{(2)}, \quad  {\bf Z} ={\bf W}^{(1)}-{\bf W}^{(2)}.
  \end{equation*}
  Then $   {\bf Z} $ satisfies
  \begin{equation}\label{6-6}
\begin{cases}
\p_rZ_1
  -r B_{22}(r,\th,\e{\bf W}^{(1)},\e N_1)\p_\th Z_2
   +r B_{12}(r,\th,\e{\bf W}^{(1)},\e N_1)\p_rZ_2\\
   +B_{21}(r,\th,\e{\bf W}^{(1)},N_1)
  \p_{\th}Z_1
  +\bar a_1(r)Z_1
  +r\e a_2(r)Z_2
   +\bar b_1(r)\p_rZ_3\\
    +\bar b_2(r)\p_\th Z_3+\bar b_3(r)Z_3
    =\mg_1(r,\th,\e{\bf W}^{(1)},\e{\bf W}^{(2)},\e{\bf N}), \ \ &{\rm{in}} \ \ \m,\\
\frac{1}{r}(\p_\th Z_1-\p_r(rZ_2))
=\mg_2(r,\th,\e{\bf W}^{(1)},\e{\bf W}^{(2)},\e{\bf N}), \ \ &{\rm{in}} \ \ \m,\\
\left(\p_r^2+\frac 1 r\p_r+\frac{1}{r^2}\p_{\th}^2\right)Z_3+ \bar a_3(r)Z_1+r\bar a_4(r) Z_2-\bar b_4(r)Z_3\\
 =\mg_3(r,\th,\e{\bf W}^{(1)},\e{\bf W}^{(2)},\e{\bf N}), \ \ &{\rm{in}} \ \ \m,\\
Z_1(r_0,\th)= Z_2(r_0,\th)=\p_r Z_3(r_0,\th)=0,\ \ &{\rm{on}}\ \ \Gamma_{en},\\
Z_3(r_1,\th)=0, \ \ &{\rm{on}}\ \ \Gamma_{ex},
\end{cases}
\end{equation}
where
\begin{equation*}
\begin{aligned}
&\mg_1(r,\th,\e{\bf W}^{(1)},\e{\bf W}^{(2)},\e{\bf N})=\bigg(r B_{22}(r,\th,\e{\bf W}^{(1)},\e N_1)-r B_{22}(r,\th,\e{\bf W}^{(2)},\e N_1)\bigg)\p_\th \e W_2^{(2)}\\
&\ -\bigg(r B_{12}(r,\th,\e{\bf W}^{(1)},\e N_1)-r B_{12}(r,\th,\e{\bf W}^{(2)},\e N_1)\bigg)\p_r \e W_2^{(2)}
 -\bigg( B_{21}(r,\th,\e{\bf W}^{(1)},\e N_1)\\
&\ - B_{21}(r,\th,\e{\bf W}^{(2)},\e N_1)\bigg)\p_\th \e W_1^{(2)}
 +G_1(r,\th,\e{\bf W}^{(1)},\e{\bf N})-G_1(r,\th,\e{\bf W}^{(2)},\e{\bf N}),\\
 &\mg_2(r,\th,\e{\bf W}^{(1)},\e{\bf W}^{(2)},\e{\bf N})=G_2(r,\th,\e{\bf W}^{(1)},\e{\bf N})-G_2(r,\th,\e{\bf W}^{(2)},\e{\bf N}),\\
 &\mg_3(r,\th,\e{\bf W}^{(1)},\e{\bf W}^{(2)},\e{\bf N})=G_3(r,\th,\e{\bf W}^{(1)},\e{\bf N})-G_3(r,\th,\e{\bf W}^{(2)},\e{\bf N}).\\
\end{aligned}
\end{equation*}
\par  To simplicity the notation, denote
\begin{equation*}
\begin{aligned}
&\e B_{i2}^{(1)}(r,\th):=B_{i2}(r,\th,\e{\bf W}^{(1)},\e N_1),\ \ \ \ \  \ \ \ \   \ \ (r,\th)\in \m,\ \ i=1,2,\\
&\e\mg_j(r,\th):=\mg_j(r,\th,\e{\bf W}^{(1)},\e{\bf W}^{(2)},\e{\bf N}),\quad \ \ (r,\th)\in \m,\ \  j=1,2,3.\\
\end{aligned}
\end{equation*}
In order to estimate  $   {\bf Z} $, one can decompose $Z_1$ and $Z_2$ as
\begin{equation*}
\begin{aligned}
Z_1=\frac{1}{r}\p_{\theta} \phi_3+ \p_r \phi_4,\
 Z_2= -\p_r\phi_3) + \frac{1}{r}\p_{\theta}\phi_4-\frac{d_1}{r}, \
  d_1= -\frac{r_0}{2\pi}\int_0^{2\pi} \p_r\phi_3(r_1,\theta) \de\theta, \ \ {\rm{in}} \ \ \m.
\end{aligned}
\end{equation*}
where $\phi_3$ and $\phi_4$ solve the following boundary value problems respectively:\begin{align}\label{6-7}
\begin{cases}
(\p_r^2+\frac{1}{r}\p_r+ \frac{1}{r^2}\p_\th^2)\phi_3= \e \mg_2,\ \ &{\rm{in}} \ \ \m,\\
\phi_3(r_0,\theta)=0,\ \ &{\rm{on}}\ \ \Gamma_{en},\\
\phi_3(r_1,\theta)=0, \ \ &{\rm{on}}\ \ \Gamma_{ex},\\
\end{cases}
\end{align}
and
 \begin{equation}\label{6-8}
\begin{cases}
\p_r^2\phi_4
  - \e B_{22}^{(1)}\p_\th^2 \phi_4
   +2 \e B_{12}^{(1)}\p_{r\th }^2\phi_4-\frac{\e B_{12}^{(1)}}{r}\p_\th \phi_4
   +\bar a_1\partial_r\phi_4
  +\e a_2\partial_\th\phi_4
   +\bar b_1\p_rZ_3\\
    +\bar b_2\p_\th Z_3+\bar b_3Z_3
    =\e\mg_1-\p_r\bigg(\frac{1}{r}\p_{\theta}\phi_3\bigg)-r \e B_{22}^{(1)}\p_{r\th}^2\phi_3
    +r \e B_{12}^{(1)}\p_r^2\phi_3
    -\e B_{12}^{(1)}
  \frac{\p_{\th}^2\phi_3}{r}\\
  -\bar a_1\frac{\p_{\theta}\phi_3}{r}
  +r\e a_2\p_{r}\phi_3-\e B_{12}^{(1)}\p_{r}\phi_3-\e B_{12}^{(1)}\frac{ d_1}{r}+\e a_2 d_1, \ \ &{\rm{in}} \ \ \m,\\
\left(\p_r^2+\frac 1 r\p_r+\frac{1}{r^2}\p_{\th}^2\right)Z_3+ \bar a_3\p_r\phi_4+\bar a_4\p_\th \phi_4-\bar b_4Z_3=\e\mg_3-\bar a_3\frac{\p_{\theta} \phi_3}{r}+\bar a_4(r\p_r\phi_3+d_1), \ \ &{\rm{in}} \ \ \m,\\
 \partial_r\phi_4(r_0,\th)=\frac{1}{r_0}\p_{\theta}\phi_3(r_0,\th), \ \ \p_r Z_3(r_0,\th)=0,\
 \partial_\th\phi_4(r_0,\th)=r_0\p_{r}\phi_3(r_0,\th)+d_1,\ \ &{\rm{on}}\ \ \Gamma_{en},\\
Z_3(r_1,\th)=0, \ \ &{\rm{on}}\ \ \Gamma_{ex}.
\end{cases}
\end{equation}
Then  similar arguments  as for \eqref{5-12} and \eqref{5-17}  yield
\begin{equation*}
\begin{aligned}
&\|\phi_3\|_{H^2(\m)}\leq C\|\e \mg_2\|_{L^2(\m)}\leq C(\delta_e+\delta_v)\left(\|(\e Z_1,\e Z_2)\|_{L^2(\m)}+\|\e Z_3\|_{H^1(\m)}\right),
\end{aligned}
\end{equation*}
and
\begin{equation*}
\begin{aligned}
\|(\phi_4,Z_3)\|_{H^1(\m)}&\leq C\left(\|\e\mg_1\|_{L^2(\m)}+\|\e\mg_3\|_{L^2(\m)}
+\|\phi_3\|_{H^2(\m)}
\right)\\
&\leq C(\delta_e+\delta_v) \left(\|(\e Z_1,\e Z_2)\|_{L^2(\m)}+\|\e Z_3\|_{H^1(\m)}\right).
\end{aligned}
\end{equation*}
 Collecting the above estimates leads to
\begin{equation}\label{6-9}
\begin{aligned}
\|(Z_1,Z_2)\|_{L^2(\m)}+\|Z_3\|_{H^1(\m)}\leq \mc_2^\sharp(\delta_e+\delta_v) \left(\|(\e Z_1,\e Z_2)\|_{L^2(\m)}+\|\e Z_3\|_{H^1(\m)}\right)
\end{aligned}
\end{equation}
  for a  constant $\mc_2^\sharp>0$ depending only on $(\gamma,b_0, \rho_0,U_{1,0},U_{2,0}, S_{0},E_{0},r_0,r_1,\epsilon_0)$.
 This, together with \eqref{6-3}, yields that
 \begin{equation*}
\begin{aligned}
\|{\bf Z}\|_\Sigma \leq 4\mc_2^\sharp(\mc_1^\sharp+1)(\delta_e+\sigma_p) \|\e{\bf Z}\|_\Sigma.
\end{aligned}
\end{equation*}
 If $\delta_e $ and $\sigma_p$ satisfy
\begin{equation}\label{6-10}
\mc_2^\sharp(\mc_1^\sharp+1)(\delta_e+\sigma_p)\leq \frac18,
\end{equation}
then  one can conclude that  $\mathfrak{F}_1^{\e {\bf N}}$ has a unique fixed point in $\ma_{\delta_v, r_1}$ provided that the conditions \eqref{6-4} and \eqref{6-10} hold.
\par { \bf Step 2. The  existence and uniqueness of a fixed point of $\mathfrak{F}_2$.}
\par For any fixed $ \e {\bf N}\in \ma_{\delta_e, r_1}$,  { \bf Step 1} shows that  $\mathfrak{F}_1^{\e {\bf N}}$ has a unique fixed point in $ \e{\bf W}\in \ma_{\delta_v, r_1}$ provided that the conditions \eqref{6-4} and \eqref{6-10} hold.  Note that
the fixed point $\e {\bf W}$ solves the nonlinear boundary value problem:
\begin{equation}\label{6-11}
\begin{cases}
\p_r\e W_1
  -r B_{22}(r,\th,\e{\bf W},\e N_1)\p_\th \e W_2
   +r B_{12}(r,\th,\e{\bf W},\e N_1)\p_r\e W_2
   +B_{21}(r,\th,\e{\bf W},\e N_1)
  \p_{\th}\e W_1\\
  +\bar a_1(r)\e W_1
  +r\e a_2(r)\e W_2
   +\bar b_1(r)\p_r\e W_3
    +\bar b_2(r)\p_\th\e W_3+\bar b_3(r)\e W_3
    =G_1(r,\th,\e{\bf W},\e{\bf N}), \ \ &{\rm{in}} \ \ \m,\\
\frac{1}{r}(\p_\th \e W_1-\p_r(r\e W_2))=G_2(r,\th,\e{\bf W},\e{\bf N}), \ \ &{\rm{in}} \ \ \m,\\
\left(\p_r^2+\frac 1 r\p_r+\frac{1}{r^2}\p_{\th}^2\right)\e W_3+ \bar a_3(r)\e W_1+r\bar a_4(r)\e W_2-\bar b_4(r)\e W_3
 =G_3(r,\th,\e{\bf W},\e{\bf N}), \ \ &{\rm{in}} \ \ \m,\\
(\e W_1,\e W_2,\p_r\e W_3)(r_0,\th)=(U_{1,en}, U_{2,en},E_{en})(\th)-(U_{1,0},U_{2,0},E_0),\ \ &{\rm{on}}\ \ \Gamma_{en},\\
\e W_3(r_1,\th)=\Phi_{ex}(\th)-\bar\Phi(r_1), \ \ &{\rm{on}}\ \ \Gamma_{ex}.\\
\end{cases}
\end{equation}
Therefore, $( \e U_1, \e U_2, \e \Phi)=(\e W_1, \e W_2, \e W_3)+(\bar U_1, \bar U_2, \bar \Phi)\in (H^3(\m))^2\times H^4(\m) $ solves
\begin{equation}\label{6-12}
\begin{cases}
\bigg(c^2(\e N_1+ K_0,  \e U_1, \e U_2,\e\Phi)- \e U_1^2\big)\p_r  \e U_1+\bigg(c^2(\e  W_1+ K_0,\e U_1, \e U_2,\e\Phi)- \e U_2^2\bigg)\frac{\p_\theta \e U_2}{r}\\
-{\e U_1\e U_2}\bigg(\p_r\e U_2+\frac{\p_\theta \e U_1}{r}\bigg)
+\frac {c^2(\e  N_1+ K_0,\e U_1, \e U_2,\e\Phi)\e U_1}{r}+{\bigg(\e U_1\p_r+\frac{\e U_2}{r}\p_\theta\bigg)\e\Phi}
=0,\ \ &{\rm{in}}\ \ \Omega,\\
\frac{\e U_2}{r}(\p_\th \e U_1-\p_r(r\e U_2))=  \frac{e^{(\e N_2+S_0)} }{\gamma-1}\mh^{\gamma-1}( \e{\bf N}+{\bf N}_0,\e U_1, \e U_2,\e\Phi){\p_r \e N_2}-\p_r \e N_1,\ \ &{\rm{in}}\ \ \Omega,\\
\bigg(\p_r^2+\frac 1 r\p_r+\frac{1}{r^2}\p_{\th}^2\bigg) \e \Phi=\mh(\e{\bf N}+{\bf N}_0,\e U_1, \e U_2,\e\Phi)-b,\ \ &{\rm{in}}\ \ \Omega,\\
(\e U_1, \e U_2,  \p_r\e\Phi)(r_0,\th)=(U_{1,en}, U_{2,en},E_{en})(\th),\ \ &{\rm{on}}\ \ \Gamma_{en},\\
\e\Phi(r_1,\th)=\Phi_{ex}(\th), \ \ &{\rm{on}}\ \ \Gamma_{ex}.\\
\end{cases}
\end{equation}
Furthermore, one can  follow from \eqref{3-1} to derive that $ \mh(\e{\bf N}+{\bf N}_0,\e U_1, \e U_2,\e\Phi)\in H^3(\m) $. Applying the Morrey inequality yields
\begin{equation*}
\|\e\Phi\|_{C^{2,\frac12}(\overline\m)}\leq C\|\e\Phi\|_{H^{4}(\m)} \ {\rm{and}} \
\|\mh\|_{C^{1,\frac12}(\overline\m)}\leq C\|\mh\|_{H^{3}(\m)}.\\
\end{equation*}
Then we consider the following problem:
\begin{equation}\label{6-12-f}
\begin{cases}
\bigg(\p_r^2+\frac 1 r\p_r+\frac{1}{r^2}\p_{\th}^2\bigg) \e \Phi=\mh(\e{\bf N}+{\bf N}_0,\e U_1, \e U_2,\e\Phi)-b,\ \ &{\rm{in}}\ \ \Omega,\\
  \p_r\e\Phi(r_0,\th)=E_{en}(\th),\ \ &{\rm{on}}\ \ \Gamma_{en},\\
  \e\Phi(r_1,\th)=\Phi_{ex}(\th), \ \ &{\rm{on}}\ \ \Gamma_{ex}.\\
\end{cases}
\end{equation}
By the standard Schauder
estimate, we obtain that
\begin{equation}\label{6-12-f-f}
\begin{aligned}
\|\e\Phi\|_{C^{3,\frac12}(\overline\m)}
\leq C(\|\e\Phi\|_{H^{4}(\m)}+\|\mh\|_{H^{3}(\m)}+\|b\|_{C^{2}(\overline\m)}
+\|(E_{en},\Phi_{ex})\|_{C^{4}(\mathbb{T}_{2\pi})}
\bigg).\\
\end{aligned}
\end{equation}
\par Next, for any $\e {\bf N}\in\ma_{\delta_e, r_1} $, we construct a mapping  $\mathfrak{F}_2$: $\e {\bf N}\in\ma_{\delta_e, r_1} \mapsto {\bf N}\in\ma_{\delta_e, r_1}$, where $ {\bf N}=(N_1,N_2)$ solves the following transport equations, respectively:
\begin{equation}\label{6-13}
\begin{cases}
\begin{aligned}
&\mh(\e{\bf N}+{\bf N}_0,\e U_1, \e U_2,\e\Phi)\bigg(\e U_1\partial_r +\frac{\e U_2}{r}\partial_{\theta}\bigg) N_1=0, \ \ {\rm{in}}\ \ \Omega,\\
&N_1(r_0,\th)=  K_{en}(\th)-K_0,\qquad\qquad\qquad\qquad\quad {\rm{on}}\ \ \Gamma_{en},\\
 \end{aligned}
\end{cases}
\end{equation}
and
\begin{equation}\label{6-14}
\begin{cases}
\begin{aligned}
&\mh(\e{\bf N}+{\bf N}_0,\e U_1, \e U_2,\e\Phi)\bigg(\e U_1\partial_r +\frac{\e U_2}{r}\partial_{\theta}\bigg) N_2=0,\ \ {\rm{in}}\ \ \Omega,\\
 &N_2(r_0,\th)=S_{en}(\th)-S_0,\qquad\qquad\qquad\qquad\quad {\rm{on}}\ \ \Gamma_{en}.\\
\end{aligned}
\end{cases}
\end{equation}
It follows from the first equation in \eqref{6-12} that
\begin{equation}\label{6-15}
\partial_r\bigg(r\mh(\e{\bf N}+{\bf N}_0,\e U_1, \e U_2,\e\Phi) \e U_1\bigg)+\partial_{\theta}\bigg(\mh(\e{\bf N}+{\bf N}_0,\e U_1, \e U_2,\e\Phi) \e U_2\bigg)=0,
\end{equation}
from which  one can define a stream function on $[r_0,r_1]\times \mathbb{R}$ as
\begin{equation*}
\mathscr{L}(r,\th)=\int_{0}^{\th}r_0\left(\mh(\e{\bf N}+{\bf N}_0,\e U_1, \e U_2,\e\Phi)\e U_1\right)(r_0,s)\de s
-\int_{r_0}^r\left(\mh(\e{\bf N}+{\bf N}_0,\e U_1, \e U_2,\e\Phi)\e U_2\right)(s,\th) \de s.
\end{equation*}
 Here the function $\mathscr{L}$ defined above may not be periodic in $\theta$. However, it holds true that
\begin{equation*}\begin{cases}
\p_r\mathscr{L}(r,\th)=-\left(\mh(\e{\bf N}+{\bf N}_0,\e U_1, \e U_2,\e\Phi)\e U_2\right)(r,\th)\in H^3(\m),\\
\p_\th\mathscr{L}(r,\th)=r\left(\mh(\e{\bf N}+{\bf N}_0,\e U_1, \e U_2,\e\Phi)\e U_1\right)(r,\th)\in H^3(\m),\\
\mathscr{L}(r,\th)\in  H^4(\m).\\
\end{cases}
\end{equation*}
  Note that $\p_\th\mathscr{L}(r_0,\th)=r_0\bigg(\mh(\e{\bf N}+{\bf N}_0,\e U_1, \e U_2,\e\Phi)\e U_{1}\bigg)(r_0,\th)>0$, from which one obtains $\mathscr{L}(r,\th)$ is an strictly increasing function of $\th$ for each fixed $r\in [r_0, r_1]$. Therefore, the inverse function of $\mathscr{L}(r_0,\cdot)$: $\th\in\mathbb{R}\mapsto t\in\mathbb{R}$ is well-defined and is denoted by $\mathscr{L}_{r_0}^{-1}$: $t\in\mathbb{R}\mapsto\th\in\mathbb{R}$.
  \par  Define
\begin{equation}\label{6-16-1}
\begin{cases}
\begin{aligned}
&N_1(r,\th)=  K_{en}\left(\mathscr{L}_{r_0}^{-1}(\mathscr{L}(r,\th))\right)-K_0,\ \ (r,\th)\in \m,\\
&N_2(r,\th)=  S_{en}\left(\mathscr{L}_{r_0}^{-1}(\mathscr{L}(r,\th))\right)-S_0, \ \ (r,\th)\in \m.
\end{aligned}
\end{cases}
\end{equation}
We claim that $N_1$ and $N_2$ defined in \eqref{6-16-1} are periodic in $\theta$ with period $2\pi$. Indeed, it follows from the definition that $$N_1(r,\th+2\pi)=K_{en}\left(\mathscr{L}_{r_0}^{-1}(\mathscr{L}(r,\th+2\pi))\right)-K_0.$$  Denote $\beta_0=\mathscr{L}_{r_0}^{-1}(\mathscr{L}(r,\th))$ and $\beta_1=\mathscr{L}_{r_0}^{-1}(\mathscr{L}(r,\th+2\pi))$. It suffices to show that $\beta_0+2\pi=\beta_1$. Since $\mh(\e{\bf N}+{\bf N}_0,\e U_1, \e U_2,\e\Phi)(\e{U}_1, \e{U}_2)$ is periodic in $\theta$ with period $2\pi$, one has
\begin{equation*}
\begin{aligned}
\mathscr{L}_{r_0}(\beta_1)&=\mathscr{L}(r,\theta+2\pi)= \mathscr{L}(r,\th)+ \int_{\theta}^{\theta+2\pi} r_0\left(\mh(\e{\bf N}+{\bf N}_0,\e U_1, \e U_2,\e\Phi)\e U_1\right)(r_0,s)\de s\\
&= \mathscr{L}_{r_0}(\beta_0)+ \int_{\beta_1-2\pi}^{\beta_1} r_0\left(\mh(\e{\bf N}+{\bf N}_0,\e U_1, \e U_2,\e\Phi)\e U_1\right)(r_0,s)\de s.
\end{aligned}
\end{equation*}
 Note that
 $$ \mathscr{L}_{r_0}(\beta_1)=\int_{0}^{\beta_1} r_0\left(\mh(\e{\bf N}+{\bf N}_0,\e U_1, \e U_2,\e\Phi)\e U_1\right)(r_0,s)\de s,$$ there holds
\begin{equation*}
\mathscr{L}_{r_0}(\beta_0)= \mathscr{L}_{r_0}(\beta_1)- \int_{\beta_1-2\pi}^{\beta_1}r_0\left(\mh(\e{\bf N}+{\bf N}_0,\e U_1, \e U_2,\e\Phi)\e U_1\right)(r_0,s)\de s=\mathscr{L}_{r_0}(\beta_1-2\pi).
\end{equation*}
By the monotonicity of $\mathscr{L}_{r_0}(\cdot)$, $\beta_1-2\pi=\beta_0$. It is easy to verify that the functions defined in \eqref{6-16-1} yield the unique solution to \eqref{6-13} and \eqref{6-14}, respectively.
  Furthermore,   it follows from  \eqref{1-9-a}, \eqref{1-9-bb} and \eqref{6-12-f-f} that
\begin{equation*}\begin{aligned}
&\left(\mh(\e {\bf N}+ {\bf N}_0,\e\Phi,\e U_1, \e U_2)\e U_1\right)(r_0,\cdot)\\
&=\left(\bigg(\frac{\gamma-1}{\gamma e^{S_{en}}}\bigg)
^{\frac{1}{\gamma-1}}\bigg(K_{en}+\e\Phi-\frac{1}{2}\left(U_{1,en}^2+U_{2,en}^2\right)
\bigg)
^{\frac{1}{\gamma-1}}U_{1,en}\right)(r_0,\cdot)\in C^{3}(\mathbb{T}_{2\pi}).
\end{aligned}
\end{equation*} Then one derives $\mathscr{L}_{r_0}^{-1}(\cdot)\in
C^{4}(\mathbb{R})$ and
\begin{equation}\label{6-17}
\begin{cases}
\begin{aligned}
&\|N_1\|_{H^4(\m)}\leq \mc_3^\sharp \|K_{ en}-K_0\|_{C^4(\mathbb{T}_{2\pi})},\\
&\|N_2\|_{H^4(\m)}\leq \mc_3^\sharp \|S_{ en}-S_0\|_{C^4(\mathbb{T}_{2\pi})},\\
\end{aligned}
\end{cases}
\end{equation}
where  $\mc_3^\sharp>0$ is a constant depending only on $(\gamma,b_0, \rho_0,U_{1,0},U_{2,0}, S_{0},E_{0},r_0,r_1,\epsilon_0)$.
\par  Define another iteration mapping
 \begin{equation*}
\mathfrak{F}_2(\e{\bf N})={\bf N}, \quad{\rm{ for \ each}} \ \ \e{\bf N}\in \ma_{\delta_e, r_1}.
\end{equation*}
Set
\begin{equation}
\label{6-20}
\delta_e= 4\mc_3^\sharp\sigma_p.
\end{equation}
Then \eqref{6-17} yields that
\begin{equation*}
\|{\bf N}\|_{H^4(\m)}\leq 2\mc_3^\sharp\sigma_p\leq \frac12 \delta_e.
\end{equation*}
This shows that the mapping $\mathfrak{F}_2$ maps $\ma_{\delta_e, r_1}$ into itself.
It remains to show that the mapping $\mathfrak{F}_2$ is contractive in a low order norm for suitably small $\sigma_p$.  for
any $\e{\bf N}^{(i)}\in \ma_{\delta_e, r_1}, (i=1,2)$, set  ${\bf N}^{(i)}=\mathfrak{F}_2(\e {\bf N}^{(i)}) $ and denote
$$ {\bf N}^{h}=\e {\bf N}^{(1)}-\e {\bf N}^{(2)}\ \ {\rm{ and }}\ \ {\bf N}^{h}={\bf N}^{(1)}-{\bf N}^{(2)}. $$
Furthermore, set $$( \e U_1^{(i)},\e U_2^{(i)}, \e \Phi^{(i)})=(\e W_1^{(i)}, \e W_2^{(i)}, \e W_3^{(i)})+(\bar U_1, \bar U_2, \bar \Phi),$$ where $ \e {\bf W}^{(i)} $ is the unique fixed point of $\mathfrak{F}_1^{(\e{\bf N}^{(i)})}$ for $i=1,2$.
 Then it follows from \eqref{6-16-1} that
\begin{equation}\label{6-21}
\begin{cases}
\begin{aligned}
&N_1^{(i)}(r,\th)=  K_{en}\left(\mathscr{T}^{(i)}(r,\th)\right)-K_0,\ \ (r,\th)\in \m,\\
&N_2^{(i)}(r,\th)=  S_{en}\left(\mathscr{T}^{(i)}(r,\th)\right)-S_0, \ \ (r,\th)\in \m,
\end{aligned}
\end{cases}
\end{equation}
with $ \mathscr{T}^{(i)}(r,\th)=(\mathscr{L}_{r_0}^{(i)})^{-1}(\mathscr{L}^{(i)}(r,\th))
 $ and
\begin{equation*}
\begin{aligned}
\mathscr{L}^{(i)}(r,\th)&=\int_{0}^{\th}r_0\left(\mh(\e{\bf N}^{(i)}+{\bf N}_0,\e U_1^{(i)}, \e U_2^{(i)},\e\Phi^{(i)})\e U_1^{(i)}\right)(r_0,s)\de s\\
&\quad  -\int_{r_0}^r\left(\mh(\e{\bf N}^{(i)}+{\bf N}_0,\e U_1^{(i)}, \e U_2^{(i)},\e\Phi^{(i)})\e U_2^{(i)}\right)(s,\th) \de s.
\end{aligned}
\end{equation*}
Here $(\mathscr{L}_{r_0}^{(i)})^{-1}$: $t\in\mathbb{R}\mapsto\th\in\mathbb{R}$ is the inverse function of $\mathscr{L}^{(i)}(r_0,\cdot)$: $\th\in\mathbb{R}\mapsto t\in\mathbb{R}$.
Therefore,
 one obtains
\begin{equation*}
|N_1^{h}|=|N_1^{(1)}-N_1^{(2)}|\leq \|K'_{en}\|_{L^\infty(\mathbb{T}_{2\pi})}\left
|\mathscr{T}^{(1)}(r,\th)-\mathscr{T}^{(2)}(r,\th)\right|.
\end{equation*}
 Note that $ \mathscr{L}^{(i)}_{r_0}\left(\mathscr{T}^{(i)}(r,\th)\right)
     =\mathscr{L}^{(i)}(r,\th)$. Then it holds that
\begin{align*}
&\int_{\mathscr{T}^{2}(r,\th)}^{\mathscr{T}^{(1)}(r,\th)}
r_0\left(\mh(\e {\bf N}^{(1)}+ {\bf N}_0,\e U_1^{(1)}, \e U_2^{(1)},\e\Phi^{(1)})\e U_1^{(1)}\right)(r_0,s) \de s
=\mathscr{L}^{(1)}(r,\th)-\mathscr{L}^{(2)}(r,\th)\\
&\quad-\int_{0}^{\mathscr{T}^{(2)}(r,\th)}
r_0\left(\mh(\e {\bf N}^{(1)}+ {\bf N}_0,\e U_1^{(1)}, \e U_2^{(1)},\e\Phi^{(1)})\e U_1^{(1)}-\mh(\e {\bf N}^{(1)}+ {\bf N}_0,\e U_1^{(2)}, \e U_2^{(2)},\e\Phi^{(2)})\e U_1^{(2)}\right)(r_0,s)\de s,
\end{align*}
from which one derives
\begin{align*}
&\mathfrak{m}^{(1)}\left|\mathscr{T}^{(1)}(r,\th)-\mathscr{T}^{(2)}(r,\th)\right|
\leq\left|\mathscr{L}^{(1)}(r,\th)-\mathscr{L}^{(2)}(r,\th)\right|\\
&\quad+\int_{0}^{2\pi}r_0\left|\mh(\e {\bf N}^{(1)}+ {\bf N}_0,\e U_1^{(1)}, \e U_2^{(1)},\e\Phi^{(1)})\e U_1^{(1)}-\mh(\e {\bf N}^{(2)}+ {\bf N}_0,\e U_1^{(2)}, \e U_2^{(2)},\e\Phi^{(2)})\e U_1^{(2)}\right|(r_0,s)\de s
\end{align*}
with $\mathfrak{m}^{(i)}:=\displaystyle\min_{\th\in[0,2\pi]}r_0\left(\mh(\e {\bf N}^{(i)}+ {\bf N}_0,\e U_1^{(i)}, \e U_2^{(i)},\e\Phi^{(i)})\e U_1^{(i)}\right)(r_0,\th)>0$. Noting that
$$\left(\e N_1^{(1)}-\e N_1^{(2)}\right)(r_0,\th)=\left(\e N_2^{(1)}-\e N_2^{(2)}\right)(r_0,\th)=\left(\e U_1^{(1)}-\e U_1^{(2)}\right)(r_0,\th)=\left(\e U_2^{(1)}-\e U_2^{(2)}\right)(r_0,\th)\equiv 0,$$
one has
\begin{equation}\label{6-22}
\|N_1^{h}\|_{L^2(\Omega)}\leq C\sigma_p\bigg(\|(\tilde{{\bf W}}^{(1)}-\tilde{{\bf W}}^{(2)},\tilde {\bf N}^{h})\|_{L^2(\Omega)}+\|(\tilde W_{3}^{(1)}-\tilde W_{3}^{(2)})(r_0,\cdot)\|_{L^2(\mathbb{T}_{2\pi})}\bigg).
\end{equation}
Furthermore, there holds
\begin{equation*}
\begin{aligned}
|\p_{r} N_1^{h}|=&\left|K_{en}'\left(\mathscr{T}^{(1)}(r,\th)\right)\p_r \mathscr{T}^{(1)}(r,\th)-K_{en}'\left(\mathscr{T}^{(2)}(r,\th)\right)\p_r \mathscr{T}^{(2)}(r,\th)\right|\\
=&\left|\left(K_{en}'\left(\mathscr{T}^{(1)}(r,\th)\right)-K_{en}'
\left(\mathscr{T}^{(2)}(r,\th)\right)\right)\p_r \mathscr{L}^{(1)}+ K_{en}'\left(\mathscr{T}^{(2)}(r,\th)\right)(\p_r \mathscr{T}^{(1)}-\p_r
\mathscr{T}^{(2)})\right|\\
\leq& \|K''_{en}\|_{L^\infty(\mathbb{T}_{2\pi})}
\left|\mathscr{T}^{(1)}(r,\th)-\mathscr{T}^{(2)}(r,\th)\right|
\frac{1}{\mathfrak{m}^{(1)}}\|\nabla
\mathscr{L}^{(1)}(r,\th)\|_{L^\infty(\Omega)}\\
&+ \|K'_{en}\|_{L^\infty(\mathbb{T}_{2\pi})}\frac{\|\nabla
\mathscr{L}^{(1)}(r,\th)\|_{L^\infty(\Omega)}}
{\mathfrak{m}^{(1)}\mathfrak{m}^{(2)}}\bigg|\mh(\e {\bf N}^{(1)}+ {\bf N}_0,\e U_1^{(1)}, \e U_2^{(1)},\e\Phi^{(1)})\e U_1^{(1)}(r_0,\mathscr{T}^{(1)})\\
&\quad-\mh(\e {\bf N}^{(2)}+ {\bf N}_0,\e U_1^{(2)}, \e U_2^{(2)},\e\Phi^{(2)})\e U_1^{(2)}(r_0,\mathscr{T}^{(2)})\bigg|\\
&+ \|K'_{en}\|_{L^\infty(\mathbb{T}_{2\pi})}\frac{1}{\mathfrak{m}^{(2)}} \left|\nabla \mathscr{L}^{(1)}(r,\th)-\nabla \mathscr{L}^{(2)}(r,\th)\right|,
\end{aligned}
\end{equation*}
and similar computations are valid for $\p_{\th} N_1^{h}$. Then one has
\begin{equation}\label{6-23}
\|\nabla N_1^{h}\|_{L^2(\Omega)}\leq C\sigma_p\bigg(\|(\tilde{{\bf W}}^{(1)}-\tilde{{\bf W}}^{(2)},\tilde {\bf N}^{h})\|_{L^2(\Omega)}+\|(\tilde W_{3}^{(1)}-\tilde W_{3}^{(2)})(r_0,\cdot)\|_{L^2(\mathbb{T}_{2\pi})}\bigg).
\end{equation}
Collecting the estimates \eqref{6-22} and \eqref{6-23} leads to
\begin{align}\label{6-24}
\|W_1^{h}\|_{H^1(\Omega)}\leq C\sigma_p\bigg(\|(\tilde{{\bf W}}^{(1)}-\tilde{{\bf W}}^{(2)},\tilde {\bf N}^{h})\|_{L^2(\Omega)}+\|(\tilde W_{3}^{(1)}-\tilde W_{3}^{(2)})(r_0,\cdot)\|_{L^2(\mathbb{T}_{2\pi})}\bigg).
\end{align}
 The  estimate constant  $C>0$ in \eqref{6-22}-\eqref{6-24} depends only on $(\gamma,b_0, \rho_0,U_{1,0},U_{2,0}, S_{0},E_{0},r_0,r_1,\epsilon_0)$.
Furthermore, the above estimates hold for $N_2^{h}$. Therefore, one obtains
\begin{align}\label{6-25}
\|{\bf N}^{h}\|_{H^1(\Omega)}\leq C_4^\sharp
\sigma_p\bigg(\|(\tilde{{\bf W}}^{(1)}-\tilde{{\bf W}}^{(2)},\tilde {\bf N}^{h})\|_{L^2(\Omega)}+\|(\tilde W_{3}^{(1)}-\tilde W_{3}^{(2)})(r_0,\cdot)\|_{L^2(\mathbb{T}_{2\pi})}\bigg)
\end{align}
 for a constant $C_4^\sharp>0$ depending only on $(\gamma,b_0, \rho_0,U_{1,0},U_{2,0}, S_{0},E_{0},r_0,r_1,\epsilon_0)$.
 \par
In the following, it remains to estimate
\begin{align*}
\|\e{{\bf W}}^{(1)}-\e{{\bf W}}^{(2)}\|_{L^{2}(\m)}+\|(\e W_{3}^{(1)}-\e W_{3}^{(2)})(r_0,\cdot)\|_{L^{2}(\mathbb{T}_{2\pi})}.
\end{align*}
Indeed, set
${\bf R}=\tilde{{\bf W}}^{(1)}-\tilde{{\bf W}}^{(2)}.$ It follows from \eqref{6-11} that
\begin{equation}\label{6-26}
\begin{cases}
\p_r R_1
  -r B_{22}(r,\th,\e{\bf W}^{(1)},\e N_1^{(1)})\p_\th R_2
   +r B_{12}(r,\th,\e{\bf W},\e N_1)\p_rR_2\\
   +B_{21}(r,\th,\e{\bf W}^{(1)},\e N_1^{(1)})
  \p_{\th} R_1
  +\bar a_1(r) R_1
  +r\e a_2(r) R_2\\
   +\bar b_1(r)\p_r R_3
    +\bar b_2(r)\p_\th R_3+\bar b_3(r) R_3
    =\mg_4(r,\th,\e{\bf W}^{(1)},\e{\bf N}^{(1)},\e{\bf W}^{(2)},\e{\bf N}^{(2)}), \ \ &{\rm{in}} \ \ \m,\\
\frac{1}{r}(\p_\th R_1-\p_r(rR_2))=\mg_5(r,\th,\e{\bf W}^{(1)},\e{\bf N}^{(1)},\e{\bf W}^{(2)},\e{\bf N}^{(2)}), \ \ &{\rm{in}} \ \ \m,\\
\left(\p_r^2+\frac 1 r\p_r+\frac{1}{r^2}\p_{\th}^2\right)R_3+ \bar a_3(r)R_1+r\bar a_4(r)R_2-\bar b_4(r)R_3\\
 =\mg_5(r,\th,\e{\bf W}^{(1)},\e{\bf N}^{(1)},\e{\bf W}^{(2)},\e{\bf N}^{(2)}), \ \ &{\rm{in}} \ \ \m,\\
R_1(r_0,\th)=R_2(r_0,\th)=\p_r R_3(r_0,\th)=0,\ \ &{\rm{on}}\ \ \Gamma_{en},\\
R_3(r_1,\th)=0, \ \ &{\rm{on}}\ \ \Gamma_{ex},
\end{cases}
\end{equation}
where
\begin{equation*}
\begin{aligned}
&\mg_4(r,\th,\e{\bf W}^{(1)},\e{\bf N}^{(1)},\e{\bf W}^{(2)},\e{\bf N}^{(2)})=\bigg(r B_{22}(r,\th,\e{\bf W}^{(1)},\e N_1^{(1)})-r B_{22}(r,\th,\e{\bf W}^{(2)},\e N_1^{(2)})\bigg)\p_\th \e W_2^{(2)}\\
&\ -\bigg(r B_{12}(r,\th,\e{\bf W}^{(1)},\e N_1^{(1)})-r B_{12}(r,\th,\e{\bf W}^{(2)},\e N_1^{(2)})\bigg)\p_r \e W_2^{(2)}
 -\bigg( B_{21}(r,\th,\e{\bf W}^{(1)},\e N_1^{(1)})\\
&\ - B_{21}(r,\th,\e{\bf W}^{(2)},\e N_1^{(2)})\bigg)\p_\th \e W_1^{(2)}
 +G_1(r,\th,\e{\bf W}^{(1)},\e{\bf N}^{(1)})-G_1(r,\th,\e{\bf W}^{(2)},\e{\bf N}^{(2)}),\\
 &\mg_5(r,\th,\e{\bf W}^{(1)},\e{\bf N}^{(1)},\e{\bf W}^{(2)},\e{\bf N}^{(2)})=G_2(r,\th,\e{\bf W}^{(1)},\e{\bf N}^{(1)})-G_2(r,\th,\e{\bf W}^{(2)},\e{\bf N}^{(2)}),\\
 &\mg_6(r,\th,\e{\bf W}^{(1)},\e{\bf N}^{(1)},\e{\bf W}^{(2)},\e{\bf N}^{(2)})=G_3(r,\th,\e{\bf W}^{(1)},\e{\bf N}^{(1)})-G_3(r,\th,\e{\bf W}^{(2)},\e{\bf N}^{(2)}).\\
\end{aligned}
\end{equation*}
Then it  holds that
\begin{equation}\label{6-28-a}
\begin{cases}
\|\mg_4(\cdot,\e{\bf W}^{(1)},\e{\bf N}^{(1)},\e{\bf W}^{(2)},\e{\bf N}^{(2)})\|_{L^2(\Omega)}\leq C\bigg(\|\tilde {\bf N}^{h}\|_{L^2(\Omega)}+\delta_v\|{\bf R}\|_{L^2(\Omega)}\bigg),\\
\|\mg_5(\cdot,\e{\bf W}^{(1)},\e{\bf N}^{(1)},\e{\bf W}^{(2)},\e{\bf N}^{(2)})\|_{L^2(\Omega)}\leq C\bigg(\|\tilde {\bf N}^{h}\|_{L^2(\Omega)}+\delta_e\|{\bf R}\|_{L^2(\Omega)}\bigg),\\
\|\mg_6(\cdot,\e{\bf W}^{(1)},\e{\bf N}^{(1)},\e{\bf W}^{(2)},\e{\bf N}^{(2)})\|_{L^2(\Omega)}\leq C\bigg(\|\tilde {\bf N}^{h}\|_{L^2(\Omega)}+\delta_v\|{\bf R}\|_{L^2(\Omega)}\bigg).
\end{cases}
\end{equation}
 Similar arguments as for \eqref{6-6} yield
\begin{equation}\label{6-28}
\|(R_1,R_2)\|_{L^{2}(\m)}+ \|R_3\|_{H^1(\m)}
\leq C\bigg(\|\h {\bf N}^h\|_{H^1(\mn)}+(\delta_e+\delta_v)\|{\bf R}\|_{L^{2}(\m)}\bigg).
\end{equation}
By the trace theorem, one has
\begin{equation}\label{6-29}
\|R_3(r_0,\cdot)\|_{L^{2}(\mathbb{T}_{2\pi})}\leq C\bigg(\|\h {\bf N}^h\|_{H^1(\mn)}+(\delta_e+\delta_v)\|{\bf R}\|_{L^{2}(\m)}\bigg).
\end{equation}
Therefore,  collecting the estimates \eqref{6-28}-\eqref{6-29} together with \eqref{6-3} and \eqref{6-20} leads to
\begin{equation}\label{6-30}
\begin{aligned}
&\|\e{{\bf W}}^{(1)}-\e{{\bf W}}^{(2)}\|_{L^{2}(\m)}+\|(\e W_{3}^{(1)}-\e W_{3}^{(2)})(r_0,\cdot)\|_{L^{2}\mathbb{T}_{2\pi})}\leq \mc_5^\sharp\bigg(\|\e {\bf N}^h\|_{H^1(\m)}+(\delta_e+\delta_v)\|\e{{\bf W}}^{(1)}-\e{{\bf W}}^{(2)}\|_{L^{2}(\m)}\bigg)\\
&
\leq \mc_5^\sharp\bigg(\|\e {\bf N}^h\|_{H^1(\m)}+
\bigg(4\mc_3^\sharp\sigma_p+4\mc_1^\sharp(4\mc_3^\sharp\sigma_p+\sigma_p)
\bigg)\|\e{{\bf W}}^{(1)}-\e{{\bf W}}^{(2)}\|_{L^{2}(\m)}\bigg)
\end{aligned}
\end{equation}
 for a constant $\mc_5^\sharp>0$ depending only on $(\gamma,b_0, \rho_0,U_{1,0},U_{2,0}, S_{0},E_{0},r_0,r_1,\epsilon_0)$. If  $ \sigma_p $ satisfies
\begin{align}\label{6-31}
\sigma_p  \leq \frac1{2\mc_5^\sharp\bigg(4\mc_3^\sharp+4\mc_1^\sharp(4\mc_3^\sharp+1)
\bigg)},
\end{align}
there holds
\begin{align*}
\|\e{{\bf W}}^{(1)}-\e{{\bf W}}^{(2)}\|_{L^{2}(\m)}+\|(\e W_{3}^{(1)}-\e W_{3}^{(2)})(r_0,\cdot)\|_{L^{2}\mathbb{T}_{2\pi})}\leq 2\mc_5^\sharp\|\e {\bf N}^h\|_{H^1(\m)}.
\end{align*}
Furthermore, it  follows from \eqref{6-25} that
\begin{align*}
\| {\bf N}^h\|_{H^1(\m)}\leq \mc_4^\sharp(2C_5^\sharp+1)
\sigma_p\|\e {\bf N}^h\|_{H^1(\m)}.
\end{align*}
Next, if  $ \sigma_p $ satisfies
\begin{align}\label{6-33}
\mc_4^\sharp(2C_5^\sharp+1)
\sigma_p \leq \frac14,
\end{align}
Then $\mathfrak{F}_2$ is a contractive mapping in $H^1(\m)$-norm and there exists a unique fixed point ${\bf N}\in \ma_{\delta_e, r_1}$ provided that the conditions \eqref{6-4}, \eqref{6-10},  \eqref{6-31}  and \eqref{6-33} hold.
\par Now, under the choices of $( \delta_v,\delta_e)$ given by \eqref{6-3} and \eqref{6-20}, one can choose $\sigma_1^{\sharp}$ so that whenever $\sigma_p\in (0, \sigma^{\sharp}_1]$, the conditions \eqref{5-44}, \eqref{6-4}, \eqref{6-10}, \eqref{6-31}  and \eqref{6-33} hold.
\begin{enumerate}[ \rm(1)]
 \item By  \eqref{6-3} and \eqref{6-20}, the condition \eqref{5-44} holds if
 \begin{equation*}
 \sigma_p \leq \frac{\delta_3}{4\mc_1^\sharp(4\mc_3^\sharp+1)+4\mc_3^\sharp}=:\varepsilon_1
 \end{equation*}
 for $ (\mc_1^\sharp,\mc_3^\sharp) $ from \eqref{6-1} and \eqref{6-17}.
 \item The conditions \eqref{6-4} and \eqref{6-10} hold if
 \begin{equation*}
 \sigma_p \leq \min\left\{\frac{1}{16(4\mc_3^\sharp+1)(\mc_1^\sharp)^2},
 \frac{1}{8(4\mc_3^\sharp+1)\mc_2^\sharp(\mc_1^\sharp+1)}
 \right\}=:\varepsilon_2
  \end{equation*}
   for $ \mc_2^\sharp$  from \eqref{6-9}.
   \item The conditions \eqref{6-31}  and \eqref{6-33} hold if
   \begin{equation*}
    \sigma_p \leq \min\left\{\frac1{2\mc_5^\sharp\bigg(4\mc_3^\sharp+4\mc_1^\sharp(4\mc_3^\sharp+1)
\bigg)},\frac{1}{4\mc_4^\sharp(2C_5^\sharp+1)}\right\}=:\varepsilon_3
    \end{equation*}
  for $ (\mc_4^\sharp,\mc_5^\sharp) $ from \eqref{6-25} and \eqref{6-30}.
  \end{enumerate}
  Therefore, we choose $\sigma_{1}^{\sharp}$ as
 \begin{equation}
 \label{sigma-bd-step2}
\sigma_{1}^{\sharp} =\min\left\{\varepsilon_k: k=1,2,3
  \right\}.
\end{equation}
If
\begin{equation}\label{choice1-sigma-full}
\sigma_p \leq \sigma_{1}^{\sharp},
  \end{equation}
then all the conditions \eqref{5-44}, \eqref{6-4}, \eqref{6-10}, \eqref{6-31}  and \eqref{6-33} hold  under the choices of $( \delta_v,\delta_e)$ given by \eqref{6-3} and \eqref{6-20}.
\par { \bf Step 3. Uniqueness.}
\par
Under the assumption of \eqref{choice1-sigma-full},
let $(U_1^{(i)},U_2^{(i)},\Phi^{(i)}, K^{(i)}, S^{(i)})$  ($i=1$, $2$) be two solutions to Problem \ref{probl2} and assume that both solutions satisfy \eqref{1-t-3-n}-\eqref{1-t-4-n}. Set
\begin{equation*}
  (\e U_1, \e U_2, \e \Phi, \e K,\e S)=(U_1^{(1)},U_2^{(1)},\Phi^{(1)}, K^{(1)}, S^{(1)})-(U_1^{(2)},U_2^{(2)},\Phi^{(2)}, K^{(2)}, S^{(2)}).
\end{equation*}
 Similar to proof of   the
contraction of the two mappings $\mathfrak{F}_1^{{\bf N}}$ and $\mathfrak{F}_2$, one can derive
\begin{equation*}
\|(\e U_1,\e U_2)\|_{L^2(\m)}+\|\e \Phi\|_{H^1(\m)}\leq \mc_6^\sharp\sigma_p \bigg(\|(\e U_1,\e U_2)\|_{L^2(\m)} +\|(\e \Phi,\e K,\e S)\|_{H^1(\m)}\bigg),
\end{equation*}
and
\begin{equation*}
\|(\e K,\e S)\|_{H^1(\m)}\leq \mc_6^\sharp\sigma_p \bigg(\|(\e U_1,\e U_2)\|_{L^2(\m)} +\|(\e \Phi,\e K,\e S)\|_{H^1(\m)}\bigg),
\end{equation*}
where  $\mc_6^\sharp>0$  is a  constant depending only on $(\gamma,b_0, \rho_0,U_{1,0}, U_{2,0}, S_{0},E_{0},r_0,r_1,\epsilon_0)$. Collecting these estimates leads to
\begin{equation}\label{6-35}
\|(\e U_1,\e U_2)\|_{L^2(\m)}+\|(\e \Phi,\e K,\e S)\|_{H^1(\m)}\leq  2\mc_6^\sharp\sigma_p \bigg(\|(\e U_1,\e U_2)\|_{L^2(\m)} +\|(\e \Phi,\e K,\e S)\|_{H^1(\m)}\bigg),
\end{equation}
So if
\begin{equation}\label{6-37}
\sigma_p\leq \frac{1}{4\mc_6^\sharp}=:\sigma_2^\sharp,
\end{equation}
one obtains
\begin{equation*}
 (\e U_1, \e U_2, \e \Phi, \e K,\e S)=(U_1^{(1)},U_2^{(1)},\Phi^{(1)}, K^{(1)}, S^{(1)})-(U_1^{(2)},U_2^{(2)},\Phi^{(2)}, K^{(2)}, S^{(2)})=0.
\end{equation*}
The proof of Theorem \ref{th2} is completed by choosing $ \sigma_2^\ast $ as
\begin{equation*}
\sigma_2^\ast =\min\{\sigma_1^\sharp ,\sigma_2^\sharp\},
\end{equation*}
where $  \sigma_1^\sharp $ is given in  \eqref{sigma-bd-step2}.
That is, the background supersonic   flow is structurally stable within rotational flows under  perturbations of the boundary
conditions in \eqref{1-c-n}.
\section{The stability analysis within axisymmetric  rotational flows}\noindent
\par In this section, we first utilize the deformation-curl-Poisson decomposition again to reformulate the
steady axisymmetric Euler-Poisson system. Then we design an iteration scheme to prove  Theorem \ref{th3}.
\par Using the Bernoulli's function,  the density $ \rho $  can be represented  as
\begin{equation}\label{7-3}
\rho=\mm(K,S,\Phi,U_1, U_2,U_3)=
\left(\frac{\gamma-1}{\gamma e^S}\bigg(K+\Phi-\frac{1}{2}\left(U_1^2+U_2^2+U_3^2\right)\bigg)\right)
^{\frac{1}{\gamma-1}}.
\end{equation}
Substituting \eqref{7-3}   into the first equation in \eqref{2-10},  then the system \eqref{2-10} in $\mn $ is equivalent to the following system:
\begin{equation}\label{7-4}
\begin{cases}
\begin{aligned}
&\p_r\left(r\mm(K,S,\Phi,U_1, U_2,U_3)U_1\right)+\p_{z}\left(r\mm(K,S,\Phi,U_1, U_2,U_3)U_3\right)=0,\\
&U_1(\p_r U_3-\p_{z} U_1)=U_2\p_{z}U_2+\frac{e^S\mm^{\gamma-1}(K,S,\Phi,U_1, U_2, U_3),
 }{\gamma-1}{\p_{z} S}-\p_{z}K,\\
&(U_1\p_r+U_3\p_{z})(r U_2)=0,\\
&(U_1\p_r+U_3\p_{z})S=0,\\
&(U_1\p_r+U_3\p_{z})K=0,\\
&\left(\p_r^2+\frac1 r\p_r+\p_{z}^2\right)\Phi=\mm(K,S,\Phi,U_1, U_2,U_3)-b.
\end{aligned}
\end{cases}
\end{equation}
\par Set
\begin{equation*}
\begin{aligned}
&T_1(r,z)=U_1(r,z)-\bar U_1(r),\ \ T_2(r,z)=U_{2}(r,z)-\bar U_2(r), \quad  T_3(r,z)=U_{3}(r,z),\ \ \ & (r,z)\in \mn,\\
&T_4(r,z)=S(r,z)-S_0, \ \quad\ \  T_5(r,z)=K(r,z)- K_0,\ \ T_6(r,z)=\Phi(r,z)-\bar\Phi(r), \ & (r,z)\in \mn.
\end{aligned}
\end{equation*}
Define the solution space $\mj_{\delta}$ which consists of ${\bf T}=(T_1,\cdots,T_6)\in \left(C^{2, \alpha}(\overline\mn)\right)^6$ satisfying the  estimate:
\begin{equation}\label{7-1}
||{\bf T}||_{C^{2, \alpha}(\overline\mn)}:=\sum\limits_{i=1}^6||T_i||_{C^{2, \alpha}(\overline\mn)}\leq \delta,
\end{equation}
and the following  compatibility conditions:
\begin{equation}\label{5-2-f}
\p_{z} T_1=\p_{z} T_2=T_3=\p_{z}^2T_3=\p_{z} T_4
=\p_{z} T_5=\p_{z}T_6=0, \ \ {\rm{on}} \ \Sigma_{w}^\pm,
\end{equation}
where $\delta>0$  will be determined later.
\par Given $\hat{\bf T}\in \mj_{\delta}$, we will construct an iterative procedure that generates a new ${\bf T}$, and thus one can define a mapping  from $\mj_{\delta}$ to itself by choosing a suitably small positive constant $\delta$. We first solve the transport equations for $T_2$, $T_4$ and $T_5$ to obtain its expression.   Substituting these into the first two and last equations in \eqref{7-4}, we  derive a deformation-curl-Poisson system for $(T_1, T_3, T_6)$, which is elliptic. By solving the linearized second order elliptic system with mixed  boundary conditions, $(T_1, T_3, T_6)$ can be  uniquely determined. Now we give a detailed derivation of this procedure.

\par{\bf Step 1. Solving the  linearized transport equations.}
\par We first solve the following hyperbolic equations for $(T_2, T_4, T_5)$:
\begin{equation}\label{7-6}
\begin{cases}
\begin{aligned}
&\left(\p_r+\frac{\h T_3}{\h T_1+\bar U_1} \p_{z}\right)(rT_2, T_4, T_5)=0,\qquad\qquad\qquad\qquad \quad {\rm{in}} \ \ \mn,\\
&(T_2, T_4, T_5)(r_0, z)=(U_{2,en},   K_{en},S_{en})(z)-(U_{2,0},K_0,S_0), \ \ \ {\rm{on}} \ \ \Sigma_{en}.
\end{aligned}
\end{cases}
\end{equation}
  For any point $(r,z)\in \overline\mn$,  the functions $( rT_2,T_4,T_5)$ are conserved along the trajectory determined by the following ODE equation:
\begin{equation}\label{7-7}
\begin{cases}
\begin{aligned}
		&\frac{d}{ds} \lambda(s;r,z)=
		\frac{\h T_3}{\h T_1+\bar U_1}(s, \lambda(s;r,z)),
		\\
		& \lambda(r;r,z)=z.
        \end{aligned}
\end{cases}
\end{equation}
Note that $\hat{\bf T}\in \mj_{\delta}$, then it is easy to obtain  that $\lambda(r_0;r,z)\in C^{2,\alpha }(\overline{\mn})$ satisfies
\begin{equation*}
\|\lambda(r_0;r,z) \|_{C^{2,\alpha } ( \overline{\mn} )} \leq C.
\end{equation*}
Thus it holds that
\begin{eqnarray*}
( rT_2,T_4,T_5)( r,z  )=(r_0U_{2,en},   K_{en},S_{en})(\lambda(r_0;r,z))-(J_2,K_0,S_0),
\end{eqnarray*}
from which one obtains
\begin{eqnarray}\label{7-9}
\|T_2,T_4,T_5)\|_{C^{2,\alpha } ( \overline{\mn} )} \leq C\sigma_v.
\end{eqnarray}
 Here the constant $C>0$ depends only on $(\gamma,b_0, \rho_0,U_{1,0},U_{2,0}, S_{0}, E_{0},r_0,r_1)$. One can further use \eqref{1-t-1-r-a-a} to  verify the following compatibility conditions:
\begin{eqnarray}\label{7-10}
\p_{z} T_2=\p_{z} T_4
=\p_{z} T_5=0, \ \ {\rm{on}} \ \Sigma_{w}^\pm.
\end{eqnarray}
\par{\bf Step 2. Solving the linearized elliptic system.}
\par It follows from  \eqref{7-4} and \eqref{1-c-c-r-A} that
\begin{equation}\label{7-11}
\begin{cases}
\p_r(rd_1(r)T_1)+\p_{z}(rd_2(r)T_3)+\p_r(rd_3(r)T_6)=\p_rY_1+\p_{z}Y_2,\ \ &{\rm{in}} \ \ \mn,\\
\p_rT_3-\p_{z}T_1=Y_3,\ \ &{\rm{in}} \ \ \mn,\\
\left(\p_r^2+\frac1 r\p_r+\p_{z}^2\right)T_6+d_3(r)T_1-d_4(r)T_6=Y_4,\ \ &{\rm{in}} \ \ \mn,\\
T_3(r_0,z)=U_{3,en}(z),\ \  T_6(r_0,z)=\Phi_{en}(z),\ \ &{\rm{on}} \ \ \Sigma_{en},\\
T_1(r_1,z)=U_{1,ex}(z)-\bar U_1(r_1), \ \ T_6(r_1,z)=\Phi_{ex}(z)-\bar \Phi(r_1),\ \ &{\rm{on}} \ \ \Sigma_{ex},\\
T_3(r,\pm 1)=\p_{z}T_6(r,\pm 1)=0, \ \ &{\rm{on}} \ \ \Sigma_{w}^\pm,\\
\end{cases}
\end{equation}
where
 \begin{equation*}
\begin{aligned}
&d_1(r)=\bar\rho (1-\bar M_1^2), \quad d_2(r)=\bar \rho,\ \
d_3(r)=\frac{\bar \rho \bar U_1}{\bar c^2}, \ \
  d_4(r)=\frac{\bar \rho}{\bar c^2},\\
&Y_1(r,z)=r\bar U_1\bigg(\frac{\bar\rho\bar U_2}{\bar c^2} T_2+\frac{\bar \rho}{\gamma-1}T_4
- \frac{\bar \rho}{\bar c^2}T_5\bigg)-r\bar U_1\bigg(\mm( T_4+S_0, T_5+K_0,\h T_6+\bar\Phi,\h T_1+\bar U_1,T_2+\bar U_2,\\
&\qquad\qquad\qquad\h T_3 )
-\mm(S_0, K_0,\bar\Phi,\bar U_1,\bar U_2 )
+ \frac{\bar\rho\bar U_1}{\bar c^2}\h T_1+\frac{\bar\rho\bar U_2}{\bar c^2} T_2+\frac{\bar \rho}{\gamma-1}T_4
- \frac{\bar \rho}{\bar c^2}T_5- \frac{\bar \rho}{\bar c^2}\h T_6\bigg)\\
&\quad\qquad\quad-r\bigg(\mm( T_4+S_0, T_5+K_0,\h T_6+\bar\Phi,\h T_1+\bar U_1,T_2+\bar U_2,\h T_3)
-\mm(S_0, K_0,\bar\Phi,\bar U_1,\bar U_2  )\bigg)\h T_1,\\
& Y_2(r,z)
=-r\bigg(\mm( T_4+S_0, T_5+K_0,\h T_6+\bar\Phi,\h T_1+\bar U_1,T_2+\bar U_2,T_3 )-\mm(S_0, K_0,\bar\Phi,\bar U_1,\bar U_2 )\bigg)\h T_3,\\
&Y_3(r,z)=\frac{1}{\bar U_1+\h T_1}
\bigg((T_2+\bar U_2)\p_{z}T_2+\frac{e^{(T_4+S_0)}\mm^{\gamma-1}( T_4+S_0, T_5+K_0,\h T_6+\bar\Phi,\h T_1+\bar U_1,T_2+\bar U_2,\h T_3)
 }{\gamma-1}\\
 &\qquad\qquad\qquad\qquad\quad{\p_{z} T_4}-\p_{z}T_5\bigg),\\
 &Y_4(r,z)
=-\frac{\bar\rho\bar U_2}{\bar c^2} T_2-\frac{\bar \rho}{\gamma-1}T_4
+\frac{\bar \rho}{\bar c^2}T_5+\mm( T_4+S_0, T_5+K_0,\h T_6+\bar\Phi,\h T_1+\bar U_1,T_2+\bar U_2,\h T_3 )\\
&\qquad\qquad-\mm(S_0, K_0,\bar\Phi,\bar U_1,\bar U_2 )
+ \frac{\bar\rho\bar U_1}{\bar c^2}\h T_1+\frac{\bar\rho\bar U_2}{\bar c^2} T_2+\frac{\bar \rho}{\gamma-1}T_4
- \frac{\bar \rho}{\bar c^2}T_5- \frac{\bar \rho}{\bar c^2}\h T_6+b-b_0.\\
\end{aligned}
\end{equation*}
 Then a direct computation shows that
\begin{equation}\label{7-12}
\|( Y_1, Y_2)\|_{C^{2, \alpha}(\overline\mn)}+\| Y_3\|_{C^{1, \alpha}(\overline\mn)}+\| Y_4\|_{C^{0, \alpha}(\overline\mn)}
\leq C(\sigma_v+\delta^2).
\end{equation}
 Note that $\hat{\bf T}$ satisfies the compatibility condition \eqref{5-2}. This, together with \eqref{7-10}, yields
\begin{eqnarray}\label{7-13}
\p_{z} Y_1=Y_2=\p_{z}^2 Y_2
=Y_3=0, \ \ {\rm{on}} \ \Sigma_{w}^\pm.
\end{eqnarray}
\par For the problem \eqref{7-11}, we have the following conclusion.
\begin{proposition}
The linear boundary value  problem \eqref{7-11} has a unique  solution $(T_1,T_3,T_6)\in \left(C^{2,\alpha}(\overline{\mn})\right)^3 $ satisfying the estimate
\begin{equation}\label{7-14}
\begin{aligned}
\|(T_1,T_3,T_6)\|_{C^{2,\alpha}(\overline{\mn})}
&\leq  C\bigg(\|( Y_1, Y_2)\|_{C^{2, \alpha}(\overline\mn)}+\| Y_3\|_{C^{1, \alpha}(\overline\mn)}+\| Y_4\|_{C^{0, \alpha}(\overline\mn)}+\sigma_v\bigg)
\leq C(\sigma_v+\delta^2)
\end{aligned}
 \end{equation}
 for some constant $C>0$     depending only on $(\gamma,b_0, \rho_0,U_{1,0},U_{2,0}, S_{0}, E_{0},r_0,r_1)$. Furthermore, the solution  $(T_1,T_3,T_6) $  satisfies the compatibility conditions:
\begin{equation}\label{5-2}
\p_{z} T_1=T_3=\p_{z}^2T_3=\p_{z}T_6=0, \ \ {\rm{on}} \ \Sigma_{w}^\pm.
\end{equation}
 \end{proposition}
 \begin{proof}
 \par The proof is divided into three steps.

 \par  { \bf Step 2.1. Homogenize the boundary data.}
  \par We first consider the following problem
\begin{equation}\label{7-15}
\begin{cases}
(\p_r^2+\p_{z}^2)\psi_1=Y_3,\ \ &{\rm{in}} \ \ \mn,\\
\p_r\psi_1(r_0, z)=0,\ \ &{\rm{on}} \ \ \Sigma_{en},\\
\psi_1(r_1, z)=0,\ \ &{\rm{on}} \ \ \Sigma_{ex},\\
\psi_1(r, \pm 1)=0, \ \ &{\rm{on}} \ \ \Sigma_{w}^\pm.\\
\end{cases}
\end{equation}
Then the symmetric extension  argument shows that \eqref{7-15} has a unique solution $\psi_1\in C^{3,\alpha}(\overline{\mn}) $ satisfying the estimate
\begin{equation}\label{7-16}
\|\psi_1\|_{C^{3, \alpha}(\overline\mn)}\leq C\|Y_3\|_{C^{1, \alpha}(\overline\mn)}\leq C(\sigma_v+\delta^2),
\end{equation}
and the  compatibility conditions
\begin{equation}\label{7-17}
\p_r\psi_1= \p_{z}^2\psi_1=\p_{z}^2\p_r\psi_1=0,  \ \ {\rm{on}} \ \Sigma_{w}^\pm.
\end{equation}
\par Define
\begin{equation*}
\e T_1=T_1+\p_{z}\psi_1, \ \ \e T_3=T_3-\p_r\psi_1, \ \ {\rm{in}} \ \ \mn.
\end{equation*}
Then \eqref{7-11} is transformed into
\begin{equation}\label{7-18}
\begin{cases}
\p_r(rd_1(r)\e T_1)+\p_{z}(rd_2(r)\e T_3)+\p_r(rd_3(r)T_6)=\p_r Y_5+\p_{z}  Y_6,\ \ &{\rm{in}} \ \ \mn,\\
\p_r\e T_3-\p_{z}\e T_1=0,\ \ &{\rm{in}} \ \ \mn,\\
\p_r(r\p_rT_6)+ r\p_{z}^2T_6+rd_3(r)T_1-rd_4(r)T_6=Y_7,\ \ &{\rm{in}} \ \ \mn,\\
\e T_3(r_0,z)=U_{3,en}(z),\ \  T_6(r_0,z)=\Phi_{en}(z),\ \ &{\rm{on}} \ \ \Sigma_{en},\\
\e T_1(r_1,z)=U_{1,ex}(z)-\bar U_1(r_1), \ \ T_6(r_1,z)=\Phi_{ex}(z)-\bar \Phi(r_1),\ \ &{\rm{on}} \ \ \Sigma_{ex},\\
\e T_3(r,\pm 1)=\p_{z}T_6(r,\pm 1)=0, \ \ &{\rm{on}} \ \ \Sigma_{w}^\pm,\\
\end{cases}
\end{equation}
where \begin{equation*}
Y_5(r,z)=Y_1+rd_1\p_{z}\psi_1, \ \ Y_6(r,z)=Y_2-rd_2\p_r\psi_1, \ \
Y_7(r,z)=rY_4+rd_3\p_{z}\psi_1.
\end{equation*}
It follows from the second equation of \eqref{7-18} that there exists a unique potential function $\psi_2(r, z)$ such that $\e T_1=\p_r\psi_2 $ and $  \e T_3=\p_{z}\psi_2$.  Then $\psi_2$ and $ T_6$  satisfy the following elliptic equations:
\begin{equation}\label{7-19}
\begin{cases}
\p_r(rd_1(r)\p_r\psi_2)+\p_{z}(rd_2(r)\p_{z}\psi_2)+\p_r(rd_3(r)T_6)=\p_r Y_5+\p_{z} Y_6,\ \ &{\rm{in}} \ \ \mn,\\
\p_r(r\p_rT_6)+ r\p_{z}^2T_6+rd_3(r)\p_r\psi_2-rd_4(r)T_6=Y_7,\ \ &{\rm{in}} \ \ \mn,\\
\p_{z}\psi(r_0,z)=U_{3,en}(z),\ \  T_6(r_0,z)=\Phi_{en}(z),\ \ &{\rm{on}} \ \ \Sigma_{en},\\
\p_r\psi_2(r_1,z)=U_{1,ex}(z)-\bar U_1(r_1), \ \ T_6(r_1,z)=\Phi_{ex}(z)-\bar \Phi(r_1),\ \ &{\rm{on}} \ \ \Sigma_{ex},\\
\p_{z}\psi_2(r,\pm 1)=\p_{z}T_6(r,\pm 1)=0, \ \ &{\rm{on}} \ \ \Sigma_{w}^\pm.\\
\end{cases}
\end{equation}
Set
\begin{eqnarray*}
\psi(r,z)=\psi_2(r,z)-\int_{-1}^sU_{3,en}(s)\de s,\  \ \Psi(r,z)=T_6(r,z)-\phi(r,z), \ \ (r,z)\in \mn
\end{eqnarray*}
with
\begin{eqnarray*}
\phi(r,z)=\bigg(\frac{r_1-r}{r_1-r_0}\Phi_{en}(z)+ \frac{r-r_0}{r_1-r_0} (\Phi_{ex}(z)-\bar \Phi(r_1))\bigg), \ \ (r,z)\in \mn.
\end{eqnarray*}
Then $(\psi,\Psi)$ satisfies
\begin{equation}\label{7-20}
\begin{cases}
\p_r(rd_1(r)\p_r\psi)+\p_{z}(rd_2(r)\p_{z}\psi)+\p_r(rd_3(r)\Psi)=\p_r \e Y_5+\p_{z} \e Y_6,\ \ &{\rm{in}} \ \ \mn,\\
\p_r(r\p_r\Psi)+ r\p_{z}^2\Psi+rd_3(r)\p_r\psi-rd_4(r)\Psi=\e Y_7,\ \ &{\rm{in}} \ \ \mn,\\
\psi(r_0,z)=0,\ \  \Psi(r_0,z)=0,\ \ &{\rm{on}} \ \ \Sigma_{en},\\
\p_r\psi(r_1,z)=U_{1,ex}(z)-\bar U_1(r_1), \ \ \Psi(r_1,z)=0,\ \ &{\rm{on}} \ \ \Sigma_{ex},\\
\p_{z}\psi(r,\pm 1)=\p_{z}\Psi(r,\pm 1)=0, \ \ &{\rm{on}} \ \ \Sigma_{w}^\pm,\\
\end{cases}
\end{equation}
where \begin{equation*}
\e Y_5(r,z)=Y_5-rd_3\p_{r}\phi, \ \ \e Y_6(r,z)= Y_6-rd_2U_{3,en}, \ \
\e Y_7(r,z)=Y_7+rd_4\phi-\p_r(r\p_r\phi)+ r\p_{z}^2\phi.
\end{equation*}
Then a simple computation yields
\begin{equation}\label{7-21}
\|( \e Y_5, \e Y_6)\|_{C^{2, \alpha}(\overline\mn)}+\| \e Y_7\|_{C^{0, \alpha}(\overline\mn)}
\leq C(\sigma_v+\delta^2).
\end{equation}
\par  { \bf Step 2.2. The $C^{1,\alpha}(\overline\mn)$ estimates for $ (\psi,\Psi)$.} \par Define the following function space
\begin{eqnarray}\nonumber
\mathbb{H}=\left\{(\eta_1,\eta_2)\in [H^1(\mn]^2: \eta_1=0 \ {{\rm{on}}} \ \Sigma_{ex},\ \eta_2=0 \ {{\rm{on}}} \ \Sigma_{en}\cup\Sigma_{ex}
\right\}.
\end{eqnarray}
We have the following weak formulation for \eqref{7-20}: for any $(\eta_1,\eta_2)\in \mathbb{H}$, there holds
\begin{eqnarray}\label{7-22}
\mathcal{I}((\psi,\Psi),(\eta_1,\eta_2))= \mathcal{S}(\psi,\Psi),
\end{eqnarray}
where \begin{equation*}
\begin{aligned}
&\mathcal{I}((\psi,\Psi),(\eta_1,\eta_2))=\iint_{\mn}\bigg(
(rd_1(r)\p_r\psi)\p_r\eta_1+(rd_2(r)\p_{z}\psi)\p_{z}\eta_1+(rd_3(r)\Psi)
\p_r\eta_1\\
&\qquad\qquad\qquad\qquad\ +(r\p_r\Psi)\p_r\eta_2+ r\p_{z}\Psi\p_{z}\eta_2-rd_3(r)\p_r\psi\eta_2+rd_4(r)\Psi\eta_2\bigg)\de r \de z,\\
&\mathcal{S}(\psi,\Psi)=\iint_{\mn} \bigg(\e Y_5\p_r\eta_1+ \e Y_6\p_{z}\eta_1-\e Y_7\eta_2\bigg)\de r \de z-\int_{-1}^1r_1(U_{1,ex}(z)-\bar U_1(r_1))\eta(r_1,z)
\de z\\
&\quad\qquad\quad -\int_{-1}^1\left(\e Y_5\eta_1(r_1, z)\right)\de z-\int_{r_0}^{r_1}\left(\e Y_6\eta_1(r, 1)-\e Y_6\eta_1(r, -1)\right)\de r.
\end{aligned}
\end{equation*}
 Obviously, the bilinear form $\mathcal{I}[\cdot,\cdot]$ and the linear operator $\mathcal{S}$ are bounded in  $\mathbb{H}\times\mathbb{H}$ and $\mathbb{H}$, respectively. More importantly, there exists a special structure generating from the Euler-Poisson system so that the mixed terms in $\mathcal{I}[(\psi,\Psi),(\psi,\Psi)]$ can be canceled. That is, for some positive constant $C$, the coercivity of the bilinear operator $\mathcal{I}[(\psi,\Psi),(\psi,\Psi)]$ is valid,
\begin{equation}\label{7-23}
\begin{aligned}
\mathcal{I}((\psi,\Psi),(\psi,\Psi))&=\iint_{\mn}\bigg(
rd_1(r)(\p_r\psi)^2+rd_2(r)(\p_{z}\psi)^2+r(\p_r\Psi)^2
+ r(\p_{z}\Psi)^2+rd_4(r)\Psi^2\bigg)\de r \de z\\
&=\iint_{\mn}\bigg(
r\bar\rho (1-\bar M_1^2)(\p_r\psi)^2+r\bar\rho(\p_{z}\psi)^2+r(\p_r\Psi)^2
+ r(\p_{z}\Psi)^2+r\frac{\bar \rho}{\bar c^2}\Psi^2\bigg)\de r \de z\\
&\geq  \iint_{\mn}C(|\n\phi|^2+|\n\Psi|^2)\de r\de z.
\end{aligned}
\end{equation}
This, together with the Poincar\'{e} inequality, yields
\begin{equation}\label{7-25}
\mathcal{I}((\psi,\Psi),(\psi,\Psi))\geq C \|( \phi,\Psi)\|_{H^1(\mn)}^2.
\end{equation}
\par By the Lax-Milgram theorem, there exists a unique weak solution $( \psi,\Psi)\in\mathbb{H}$ to \eqref{7-20}.
Then one can adjust the proof of  in \cite[Lemmas 3.5 and 3.6]{BDX16} to show that the weak solution $(\psi,\Psi)\in\mathbb{H}$ satisfies
\begin{equation}\label{7-27}
\begin{aligned}
 \|( \psi,\Psi)\|_{C^{1,\alpha}(\overline{\mathcal{N}})} \leq C\bigg(\|( \e Y_5, \e Y_6)\|_{C^{2, \alpha}(\overline\mn)}+\| \e Y_7\|_{C^{0, \alpha}(\overline\mn)}\bigg)\leq C(\sigma_v+\delta^2).
\end{aligned}
\end{equation}
 Then one can follow from the expressions of $ \psi $, $ \Psi $, $\e Y_i$ $ (i=5,6,7) $ and combine  the estimate \eqref{7-16} to obtain that
\begin{equation}\label{C1al-1}
	\|(T_1,T_3)\|_{C^{0,\alpha}(\overline{\mathcal{N}})}+	\|T_6\|_{C^{1,\alpha}(\overline{\mathcal{N}})}
		\leq  C\bigg(\|( Y_1, Y_2)\|_{C^{2, \alpha}(\overline\mn)}+\| Y_3\|_{C^{1, \alpha}(\overline\mn)}+\| Y_4\|_{C^{0, \alpha}(\overline\mn)}+\sigma_v\bigg)
\leq C(\sigma_v+\delta^2).
	\end{equation}
The above  constant $C>0$     depends only on $(\gamma,b_0, \rho_0,U_{1,0},U_{2,0}, S_{0}, E_{0},r_0,r_1)$.
\par  { \bf Step 2.3. The $C^{2,\alpha}(\overline\mn)$ estimates for $ (T_1,T_3,T_6)$.}
\par With the estimate \eqref{C1al-1}, we consider the following elliptic equation for $T_6$:
\begin{equation}\label{7-c-e}
\begin{cases}
\left(\p_r^2+\frac1 r\p_r+\p_{z}^2\right)T_6-d_4(r)T_6=Y_4-d_3(r)T_1,\ \ &{\rm{in}} \ \ \mn,\\
  T_6(r_0,z)=\Phi_{en}(z),\ \ &{\rm{on}} \ \ \Sigma_{en},\\
 T_6(r_1,z)=\Phi_{ex}(z)-\bar \Phi(r_1),\ \ &{\rm{on}} \ \ \Sigma_{ex},\\
\p_{z}T_6(r,\pm 1)=0, \ \ &{\rm{on}} \ \ \Sigma_{w}^\pm.\\
\end{cases}
\end{equation}
 By  the compatibility condition \eqref{7-13}, one  can use the standard Schauder
estimates  in \cite{GT98} and the method of reflection with respect to $\Sigma_w^\pm $ to  obtain that
\begin{equation}\label{C1al-1-1}
		\|T_6\|_{C^{2,\alpha}(\overline{\mathcal{N}})}
		\leq  C\bigg(\|T_1\|_{C^{0, \alpha}(\overline\mn)}+\| Y_4\|_{C^{0, \alpha}(\overline\mn)}+\sigma_v\bigg)
\leq C(\sigma_v+\delta^2).
	\end{equation}
\par Next,  we consider the following linear system for $T_1$ and $ T_3$:
\begin{equation}\label{7-c}
\begin{cases}
\p_r(rd_1(r)T_1)+\p_{z}(rd_2(r)T_3)=\p_r(Y_1-rd_3(r)T_6)+\p_{z}Y_2,\ \ &{\rm{in}} \ \ \mn,\\
\p_rT_3-\p_{z}T_1=Y_3,\ \ &{\rm{in}} \ \ \mn,\\
T_3(r_0,z)=U_{3,en}(z),\ \ &{\rm{on}} \ \ \Sigma_{en},\\
T_1(r_1,z)=U_{1,ex}(z)-\bar U_1(r_1),  &{\rm{on}} \ \ \Sigma_{ex},\\
T_3(r,\pm 1)=0, \ \ &{\rm{on}} \ \ \Sigma_{w}^\pm.\\
\end{cases}
\end{equation}
By  the standard symmetric extension technique, one can apply the Schauder estimates in \cite{GT98} to  yield that
$(T_1, T_3)\in \left(C^{2, \alpha}(\overline{\mathcal{N}})\right)^2$ with the following estimate
\begin{eqnarray}\label{C1al-x}
		\|(T_1, T_3)\|_{C^{2,\alpha}(\overline{\mathcal{N}})}
		 \leq C\bigg(\|( Y_1, Y_2,T_6)\|_{C^{2, \alpha}(\overline\mn)}+\| Y_3\|_{C^{1, \alpha}(\overline\mn)}+\sigma_v\bigg)
\leq C(\sigma_v+\delta^2).
	\end{eqnarray}
  The above constant $C>0$     depends only on $(\gamma,b_0, \rho_0,U_{1,0},U_{2,0}, S_{0}, E_{0},r_0,r_1)$. Furthermore, one can verify that
  \begin{equation}\label{5-2-s}
\p_{z} T_1=T_3=\p_{z}^2T_3=\p_{z}T_6=0, \ \ {\rm{on}} \ \Sigma_{w}^\pm.
\end{equation}
  The proof of this proposition is completed.
\end{proof}
\par Up to know, for given $\hat {{\bf T}}\in \mj_\delta$,
     we have constructed a new $ {{\bf T}} $.
   Define a map $ \mathfrak{P} $ as follows
\begin{equation}\label{7-47}
\mathfrak{P}( \hat {{\bf T}})= {{\bf T}},
\quad{\rm{ for}} \ \hat {{\bf T}}\in \mj_\delta.
\end{equation}
The estimate \eqref{7-9}, together with \eqref{7-14}, yields
\begin{equation}\label{4-46}
\sum_{i=1}^5\|T_i\|_{C^{2,\alpha}(\overline{\mn})}	\leq \mc_1^\natural(\sigma_v+\delta^2)
\end{equation}
for a constant $\mc_1^\natural>0$     depending only on $(\gamma,b_0, \rho_0,U_{1,0},U_{2,0}, S_{0}, E_{0},r_0,r_1)$. Let
\begin{equation*}
\varepsilon_4=\frac{1}{2(2(\mc_1^\natural)^2+1)} \ \ {\rm{and}}\ \ \delta= 2\mc_1^\natural\sigma_v.
\end{equation*}
Then if $ \sigma_v \leq \varepsilon_4  $,  one has
\begin{equation*}
\sum_{i=1}^5\|T_i\|_{C^{2,\alpha}(\overline{\mn})}	\leq \frac12 \delta+2(\mc_1^\natural)^2\sigma_v\delta\leq \delta.
\end{equation*}
Hence $ \mathfrak{P} $ maps $ \mj_\delta $ into itself.
\par {\bf Step 3. The contraction of the mapping $\mathfrak{P}$.}
\par In the following, we will show that $ \mathfrak{P} $ is a
  contraction mapping in $ \mj_\delta$. Let $\hat {{\bf T}}^{(k)}\in\mj_\delta$,  $ k=1,2 $, one has $  {{\bf T}}^{(k)} =\mathfrak{P}(\hat {{\bf T}}^{(k)}) $.
  Define
\begin{equation*}
{\bf D}= {{\bf T}}^{(1)}-{{\bf T}}^{(2)},\quad \hat{{\bf D}}= \hat{{{\bf T}}}^{(1)}-\hat{{{\bf T}}}^{(2)}.
\end{equation*}
Then it follows from \eqref{7-6} that
\begin{equation}\label{7-6-d}
\begin{cases}
\left(\p_r+\frac{\h T_3^{(1)}}{\h T_1^{(1)}+\bar U} \p_{z}\right)(rD_2, D_4, D_5)=\bigg(\frac{\h T_3^{(1)}}{\h T_1^{(1)}+\bar U}-\frac{\h T_3^{(2)}}{\h T_1^{(2)}+\bar U}\bigg)\p_{z}(rT_2^{(2)}, T_4^{(2)}, T_5^{(2)})\ \ &{\rm{in}} \ \ \mn,\\
(D_2, D_4, D_5)(r_0, z)=0,  \ &{\rm{on}} \ \ \Sigma_{en}.
\end{cases}
\end{equation}
By the characteristic method, there holds
\begin{eqnarray}\label{4-105}
\|(D_2, D_4, D_5)\|_{C^{1,\alpha}(\overline\mn)}  \leq C\|{\bf T}^{(2)}\|_{C^{2,\alpha}(\overline \mn)} \|\hat{{\bf D}}\|_{C^{1,\alpha}(\overline\mn)}
\leq C \delta \|\hat{{\bf D}}\|_{C^{1,\alpha}(\overline\mn)}.
\end{eqnarray}
Next, it follows from \eqref{7-11} that
\begin{equation}\label{7-11-d}
\begin{cases}
\p_r(rd_1(r)D_1)+\p_{z}(rd_2(r)D_3)+\p_r(rd_3(r)D_6)=\p_r(Y_1^{(1)}-Y_1^{(2)})
+\p_{z}(Y_2^{(1)}-Y_2^{(2)}),\ \ &{\rm{in}} \ \ \mn,\\
\p_rD_3-\p_{z}D_1=Y_3^{(1)}-Y_3^{(2)},\ \ &{\rm{in}} \ \ \mn,\\
\left(\p_r^2+\frac1 r\p_r+\p_{z}^2\right)D_6+d_3(r)D_1-d_4(r)D_6=Y_4^{(1)}-Y_4^{(2)},\ \ &{\rm{in}} \ \ \mn,\\
D_3(r_0,z)=  D_6(r_0,z)=0,\ \ &{\rm{on}} \ \ \Sigma_{en},\\
D_1(r_1,z)=D_6(r_1,z)=0,\ \ &{\rm{on}} \ \ \Sigma_{ex},\\
D_3(r,\pm 1)=\p_{z}D_6(r,\pm 1)=0, \ \ &{\rm{on}} \ \ \Sigma_{w}^\pm,\\
\end{cases}
\end{equation}
where
 $ Y_j^{(i)}$ ($ i=1,2$, $ j=1,2,3,4$) are defined as in \eqref{7-11}  with $(\hat T_1, T_2,\hat T_3,T_4,T_5,\hat T_6)$ replaced by  $(\hat T_1^{(i)}, T_2^{(i)},\hat T_3^{(i)},T_4^{(i)},T_5^{(i)},\hat T_6^{(i)})$.
 Therefore, we have the following estimate:
 \begin{equation}\label{7-14-D}
\begin{aligned}
\|(D_1,D_2,D_6)\|_{C^{1,\alpha}(\overline{\mn})}
&\leq  C\bigg(\|( Y_1^{(1)}-Y_1^{(2)},Y_2^{(1)}-Y_2^{(2)})\|_{C^{1, \alpha}(\overline\mn)}+\| Y_3^{(1)}-Y_3^{(2)}\|_{C^{1, \alpha}(\overline\mn)}\\
&\qquad+\| Y_4^{(1)}-Y_4^{(2)}\|_{C^{0, \alpha}(\overline\mn)}\bigg)\leq  C \delta \|\hat{{\bf D}}\|_{C^{1,\alpha}(\overline\mn)}.
\end{aligned}
 \end{equation}
 \par Collecting all the above  estimates leads to
\begin{eqnarray}\label{contraction}
\|{\bf D}\|_{C^{1,\alpha}(\overline\mn)}\leq \mc_{2}^\natural\delta\|\hat{{\bf D}}\|_{C^{1,\alpha}(\overline\mn)},
\end{eqnarray}
where  $\mc_2^\natural>0$   is a  constant  depending only on $(\gamma,b_0, \rho_0,U_{1,0},U_{2,0}, E_{0},r_0,r_1)$. Setting $$\sigma_1^\star=\min\left\{\varepsilon_4, \frac{1}{3\mc_1^\natural\mc_2^\natural}\right\}.$$ Then for any $\sigma_v\leq \sigma_1^\star$, $\mc_{2}^\natural\delta= 2\mc_1^\natural \mc_2^\natural\sigma_v\leq \frac23$, hence the mapping $\mathfrak{P}$ is a contraction with respect to $C^{1,\alpha}$ norm   so that there exists a unique fixed point   $ {\bf T}\in \mj_\delta$. The proof of Theorem \ref{th3} is completed.
\par {\bf Acknowledgement.} Wang is  partially supported by National Natural Science Foundation of China  11925105, 12471221.
\par {\bf Data availability.} No data was used for the research described in the article.
    \par {\bf Conflict of interest.} This work does not have any conflicts of interest.

\end{document}